\documentclass[a4paper, 10pt]{article}

\input xy
\xyoption{all}

\usepackage[normalem]{ulem}
\usepackage{amsmath,amsthm,amssymb,tikz-cd}
\usepackage{stmaryrd}
\usepackage[page,header]{appendix}
\usepackage{morefloats}
\usepackage{color}
\usepackage{makeidx}
\usepackage{fancybox}
\usepackage{mathabx}
\usepackage{tikz}
\usetikzlibrary{matrix}
\usepackage{arcs}
\usepackage{hyperref}
\usepackage{bbold}
\usepackage{lipsum}
\usepackage{cancel}
\usepackage{verbatim}

\usepackage[a4paper,	bindingoffset=0in,	left=1.2in,	right=1.2in,	top=1in,	bottom=1in,	footskip=.25in]{geometry} 
\usepackage[sorting=tny,sortcites,backend=biber,style=alphabetic,sorting=ynt]{biblatex}
\AtEveryBibitem{\clearfield{url}}

\newcommand\blfootnote[1]{%
  \begingroup
  \renewcommand\thefootnote{}\footnote{#1}%
  \addtocounter{footnote}{-1}%
  \endgroup
}

\usepackage{yfonts}
\usepackage[T1]{fontenc}

\makeindex 

\usepackage{latexsym, amsmath, amssymb, amsfonts, amscd}
\usepackage{amsthm}
\usepackage{stmaryrd}
\usepackage{t1enc}
\usepackage[mathscr]{eucal}
\usepackage{indentfirst}
\usepackage{enumerate}
\usepackage[all,poly,web,knot]{xy}
\usepackage{float}

\restylefloat{figure}
\usepackage{enumitem,kantlipsum}

\newcommand{\Title}{Title}

\numberwithin{equation}{section}

{\theoremstyle{definition}\newtheorem{definition}{Definition}[section]
	
	\newtheorem{defititle}[definition]{\Title}

	\newtheorem{remark}[definition]{Remark}

	\newtheorem{ex}[definition]{Example}
	
	\newtheorem{exs}[definition]{Examples}}
\newtheorem{thmx}{Theorem}

\newtheorem{prop}[definition]{Proposition}
\newtheorem{proposition-definition}[definition]{Proposition-Definition}
\newtheorem{lemma}[definition]{Lemma}
\newtheorem{thm}[definition]{Theorem}
\newtheorem{cor}[definition]{Corollary}

\newtheorem*{prop*}{Proposition}
\newtheorem*{theorem*}{Theorem}

\newcommand{\cD}{\mathcal{D}}
\newcommand{\cG}{\mathcal{G}}

\newcommand{\cF}{\mathcal{F}}
\newcommand{\cE}{\mathcal{E}}
\newcommand{\cQ}{\mathcal{Q}}

\newcommand{\cK}{\mathcal{K}}
\newcommand{\cT}{\mathcal{T}}

\newcommand{\aF}{\lie{a}\cF}

\newcommand{\Id}{{\hbox{Id}}}

\newcommand{\dom}{{\rm dom}}

\newcommand{\bfX}{\mathbf{X}}
\newcommand{\bfY}{\mathbf{Y}}
\newcommand{\bfZ}{\mathbf{Z}}

\DeclareMathOperator{\BCH}{BCH}

\newcommand\norm[1]{\left\lVert#1\right\rVert}
\renewcommand{\Re}{\operatorname{Re}}
\renewcommand{\Im}{\operatorname{Im}}

\def\gpd{\,\lower1pt\hbox{$\longrightarrow$}\hskip-.24in\raise2pt
	\hbox{$\longrightarrow$}\,}

\renewcommand{\latticebody}{\drop@{ }}

\newcommand{\N}{\ensuremath{\mathbb N}}
\newcommand{\Z}{\ensuremath{\mathbb Z}}
\newcommand{\C}{\ensuremath{\mathbb C}}
\newcommand{\R}{\ensuremath{\mathbb R}}

\newcommand{\cL}{{\mathcal L}}

\newcommand{\cS}{\mathcal{S}}
\newcommand{\cX}{\mathcal{X}}

\newcommand{\NN}{\ensuremath{\mathbb N}}

\newcommand{\RR}{\ensuremath{\mathbb R}}

\DeclareMathOperator{\Diff}{Diff}

\DeclareMathOperator{\supp}{supp}
\DeclareMathOperator{\singsupp}{singsupp}

\DeclareMathOperator{\ssing}{singsupp}

\newcommand{\gr}{\mathfrak{gr}}
\newcommand{\Gr}{\mathfrak{gr}}

\newcommand{\End}{\mathrm{End}}

\newcommand{\dimh}{\mathrm{dim}_{\mathrm{h}}}

\newcommand{\fib}[1]{[#1]}

\def\act{\mathbin{\hbox{$<\kern-.4em\mapstochar\kern.4em$}}}
\def\ract{\mathbin{\hbox{$\mapstochar\kern-.3em>$}}}
\def\exp{\mathrm{exp}}

\def\PB(#1,#2,#3,#4){\left\{\begin{matrix}#1&\!\!\!\stackrel{?}{\longrightarrow}&\!\!\!#2\\
		\downarrow&&\!\!\!\downarrow\\
		#3&\!\!\!\stackrel{?}{\longrightarrow}&\!\!\!#4\end{matrix}\right\}}

\def\pb(#1,#2,#3,#4){ \hom(#1 \to #3, #2 \to #4)}

\newcommand{\lie}{\mathfrak}
\newcommand{\ev}{\mathrm{ev}}

\mathchardef\mhyphen="2D
\begin{document}
	\title{A pseudodifferential calculus for maximally hypoelliptic operators and the Helffer-Nourrigat conjecture}
			
			\bigskip
			
			\author{Iakovos Androulidakis, Omar Mohsen and Robert Yuncken}
		\date{}
\maketitle

\begin{abstract}\noindent 
We extend the classical regularity theorem of elliptic operators to
maximally hypoelliptic differential operators. More precisely, given
vector fields $X_1,\ldots,X_m$ on a smooth manifold which satisfy
Hörmander's bracket generating condition, we define a principal
symbol for \textit{any} linear differential operator. Our symbol takes into account the vector fields $X_i$ and their commutators. We show that for an arbitrary differential
operator, its principal symbol is invertible if and only if the
operator is maximally hypoelliptic. This answers affirmatively a
conjecture due to Helffer and Nourrigat. Our result is proven in a more
general setting, where we allow each one of the vector fields
$X_1,\ldots,X_m$ to have an arbitrary weight. In particular, our theorem
generalizes Hörmander's sum of squares theorem to higher order
polynomials.\blfootnote{MSC2020 subject classification 35H10; Secondary 47G30, 35B44, 58J40, 22A22, 22E25.}
\end{abstract}
	
	\setcounter{tocdepth}{2} 
	\tableofcontents

\section*{Introduction}
\addcontentsline{toc}{section}{Introduction}
Elliptic linear differential operators are some of the most extensively studied differential operators in analysis. This is because of their wide applications in many areas of mathematics such as algebraic geometry, complex geometry, symplectic geometry and representation theory. These applications are based on the following fundamental regularity theorem, which is proved using the pseudodifferential calculus developed by Nirenberg, Kohn, Hörmander and others.
\begin{thm}[{\cite[Theorem 19.5.1]{HormanderBooks3}}]\label{thmx:qskdhflihsdl} Let $M$ be a smooth manifold, $D:C^\infty(M)\to C^\infty(M)$ a differential operator of order $k$. The following are equivalent
\begin{enumerate}
\item The operator $D$ is elliptic, i.e., for every $\xi \in T^*M\backslash \{0\}$, $\sigma(D,x,\xi)\neq 0$, where $\sigma$ is the classical principal symbol of $D$.
\item For every (or for some) $s\in \R$, and every distribution $u$ on $M$, $Du\in H^s(M)$ implies $u\in H^{s+k}(M)$, where $H^\bullet$ are the local $L^2$ Sobolev spaces.
\end{enumerate}
Furthermore if $M$ is compact, the previous statements are equivalent to the following
\begin{enumerate}
  \setcounter{enumi}{3}  \item For every (or for some) $s\in \R$, $D:H^{s+k}(M)\to H^{s}(M) $ is Fredholm
\end{enumerate}
\end{thm} 
In a celebrated article, Hörmander proved  that for some non-elliptic differential operators, now called Hörmander's sums of squares operators, one still has the regularity of solutions. 
\begin{thm}[\cite{Hormander:SoS}]\label{Hormanderthmintro}
\label{thm:SoS}
Let $X_1,\cdots,X_{m+1}$ be vector fields satisfying Hörmander's Lie bracket generating condition, i.e., for every $x\in M$, $T_xM$ is linearly spanned by $X_1(x),\cdots,X_{m+1}(x)$ and their higher Lie brackets $[X_i,X_j](x)$, $[X_i,[X_j,X_l]](x)$ etc. Then $D=\sum_{i=1}^mX_i^2+X_{m+1}$ is hypoelliptic, i.e., if $u$ is a distribution on $M$ such that $Du$ is smooth, then $u$ is smooth.
\end{thm}
It is natural to try to extend Hörmander's theorem  by finding sufficient conditions for the hypoellipticity of arbitrary polynomials in the vector fields $X_i$. Let $P$ be a noncommutative polynomial in $m+1$ variables with coefficients in $C^\infty(M)$.
In 1979, Helffer and Nourrigat \cite{HelfferNourrigatCRAcSci} conjectured a generalization of both Theorem \ref{thmx:qskdhflihsdl} and Theorem \ref{thm:SoS} which allows one to obtain hypoellipticity of operators of the form $P(X_1,\cdots,X_{m+1})$, and also generalises several regularity results in the literature, see \cite[Sections I.2 and I.3]{HelfferNourrigatCRAcSci}.

The goal of this article is to prove the Helffer-Nourrigat conjecture, as well as its generalisation to arbitrary filtrations of the module of vector fields on $M$. Let us start with our main theorem in a restricted case (the case $X_{m+1}=0$). We'll give its much more general form afterwards. 

Thus, consider vector fields $X_1,\cdots,X_m$ satisfying the Lie bracket generating condition. This condition gives rise to the following notion of order for a differential operator. Every differential operator can then be written as $D=P(X_1,\cdots,X_m)$ where $P$ is a noncommutative polynomial with coefficients in $C^\infty(M)$. The Hörmander order of $D$ is the minimum of $\deg(P)$ for all possible $P$s. The Hörmander order leads us to consider Sobolev spaces defined by \begin{align}\label{eqn:Sobolev}
\tilde{H}^{s}(M)&:=\{u\in L^2_{loc}M:Du\in L^2_{loc}M\quad  \text{for all } D \text{ with Hörmander order} \leq s\},\quad s\in \N.
\end{align}
We extend these Sobolev spaces for any $s\in \R$ by interpolation for $s>0$ and duality for $s<0$. Trivially we have $$\bigcap_{s\in \R}\tilde{H}^{s}(M)=\bigcap_{s\in \N}\tilde{H}^{s}(M)=C^\infty(M).$$

The crucial step is to define a principal symbol for differential operators which is sensitive to the vector fields $X_1,\cdots,X_m$ and the structure of their iterated commutators. Generalizing Helffer-Nourrigat, we can specify a family $(\Gr(\cF_x))_{x\in M}$ of osculating graded nilpotent Lie groups, as well as a subset of unitary representations $\cT_x^*\cF$ of each group (Helffer-Nourrigat call this $\Gamma_x$).  We then define an operator-valued symbol $(x,\pi) \mapsto \tilde{\sigma}(D,x,\pi)$ on $\cT^*\cF  = \bigsqcup_{x\in M} \cT_x^*\cF$. We prove the following 
\begin{thmx}\label{mainthmintro}Let $M$ be a smooth manifold, $X_1,\cdots,X_m$ are vector fields satisfying the Lie bracket generating condition, $D:C^\infty(M)\to C^\infty(M)$ a differential operator of Hörmander order $k$. The following are equivalent
\begin{enumerate}
\item For every $x\in M$ and $\pi \in \cT^*_x\cF\backslash\{0\}$, $\tilde{\sigma}(D,x,\pi)$ is injective on $C^\infty(\pi)$.
\item For every (or for some) $s\in \R$, and every distribution $u$ on $M$, $Du\in \tilde{H}^s(M)$ implies $u\in \tilde{H}^{s+k}(M)$.
\end{enumerate}
Furthermore if $M$ is compact, then the previous statements are equivalent to the following \begin{enumerate}
  \setcounter{enumi}{4}
  \item For every (or for some) $s\in \R$, $D:\tilde{H}^{s+k}(M)\to \tilde{H}^{s}(M) $ is left invertible modulo compact operators
\end{enumerate}
\end{thmx}
We now explain the principal symbol $\tilde{\sigma}$ as well as the space of representations $\cT^*_x\cF$. Before we proceed, let us mention that if the vector fields satisfy Hörmander's Lie bracket generating condition of rank $1$, i.e., $X_1(x),\cdots,X_m(x)$ span $T_xM$ for all $x\in M$, then Theorem \ref{mainthmintro} is precisely Theorem \ref{thmx:qskdhflihsdl}. The Sobolev spaces $\tilde{H}^s(M)$ and $\tilde{\sigma}$ are equal to $H^s(M)$ and $\sigma$ respectively. In \cite{HelfferNourrigatconj}, Helffer and Nourrigat proved Theorem \ref{mainthmintro} in the case of rank $2$, i.e., $X_1(x),\cdots,X_m(x)$ and $[X_i,X_j](x)$ span $T_xM$ for all $x\in M$. They also proved the implication $b\implies a$ in the general case with no assumptions on the rank. The main innovation in our work is combining their work \cite{HelfferRockland,HelfferNourrigatconj} with recent advances in noncommutative geometry by Debord and Skandalis \cite{DebSka} and van Erp and the third author \cite{YunVan:PsiDOs} together with the $C^*$-algebra of singular foliations defined by the first author and Skandalis \cite{AS1} and their blowups defined by the second author \cite{NewCalgebra}. This allows us to prove Theorem \ref{mainthmintro} with no hypothesis at all on the rank.
\paragraph{Principal symbol $\tilde{\sigma}$.}
Suppose that $X_1,\cdots,X_m$ satisfy Hörmander's Lie bracket generating condition of rank $N\in \N$. Let $G$ be the free nilpotent Lie group of rank $N$ with one generator $\tilde{X}_1,\cdots,\tilde{X}_m$ for each vector field $X_1,\cdots,X_m$. We remark that in the article, we use a better choice of nilpotent group which is smaller and more natural; for simplicity of the exposition in this introduction, we temporarily use the group $G$. Let $\pi$ be an irreducible unitary representation of $G$ on a Hilbert space $L^2\pi$. Then by taking the derivative of $\pi$, one obtains linear maps $$d\pi(\tilde{X}_1),\cdots,d\pi(\tilde{X}_m):C^\infty(\pi)\to C^\infty(\pi)$$ where $C^\infty(\pi)\subseteq L^2\pi$ is the space of smooth vectors. 

We can now define $\tilde{\sigma}$. We write $D=P(X_1,\cdots,X_m)$ for some noncommutative polynomial $P$.  This is the equivalent of taking local coordinates when defining the classical principal symbol. We then define $$\tilde{\sigma}(D,x,\pi):C^\infty(\pi)\to C^\infty(\pi),\quad \tilde{\sigma}(D,x,\pi)=P_{max,x}(d\pi(\tilde{X}_1),\cdots,d\pi(\tilde{X}_m)),$$ where $P_{max,x}$ is the maximal homogeneous part of $P$ after replacing each coefficient $f\in C^\infty(M)$ by $f(x)$. Note that this definition may depend on $P$ since if the operator $D$ can be written $D=P(X_1,\cdots,X_m)=Q(X_1,\cdots,X_m)$ for two different polynomials $P,Q$, then in general $\tilde{\sigma}(D,x,\pi)$ depends on the choice of $P$ or $Q$ (see Section \ref{sec:differential operators} for examples). But one of our main results is that this is not the case when $\pi$ belongs to a certain naturally defined subset $\mathcal{T}_x^*\cF\subseteq \hat{G}$. The set $\mathcal{T}_x^*\cF$ can be thought of as a generalization of the cotangent space $T^*_xM$ in this sub-Riemannian context. The set $\mathcal{T}_x^*\cF$ only depends on the vector fields $X_1,\cdots,X_m$ and not on $D$. This set was defined by Helffer and Nourrigat in \cite{HelfferNourrigatCRAcSci} using Kirillov's orbit method \cite{KirillovArticle}. For this reason, we call it the \textbf{Helffer-Nourrigat cone.}	
\begin{thmx}\label{thm intro char set} For each $x\in M$, for any representation $\pi \in \mathcal{T}_x^*\cF\subseteq \hat{G}$, $\tilde{\sigma}(D,x,\pi)$ doesn't depend on the presentation of $D=P(X_1,\cdots,X_m)$.\end{thmx}

We remark that the set $\mathcal{T}_x^*\cF$ is very computable in practice. We refer the reader to Section \ref{sec:characteristicset} for the precise construction of $\mathcal{T}_x^*\cF$ and for various examples.

\paragraph{Main theorem.}
We will prove a much more general form of Theorem \ref{mainthmintro} as follows. Let us allow weights on the vector fields $X_1,\cdots,X_m$, meaning that we attach to each vector field $X_i$ a natural number $v_i\in \N$. The Hörmander order of $D$ is now the minimum degree of $P$ taking weights into account. The Sobolev spaces $\tilde{H}^{s}(M)$ are defined as in \eqref{eqn:Sobolev} when $s$ is a multiple of $\mathrm{gcd}(v_i)$, interpolating for other values of $s$. The principal symbol $\tilde{\sigma}$ in this case is defined as before, the only difference is that $P_{max,x}$ is the maximal weighted homogeneous part. 

We can now state the main theorem of the article.

\begin{thmx}\label{mainthmintro2}Let $M$ be a smooth manifold, $X_1,\cdots,X_m$ vector fields satisfying Hörmander's condition, $v_1,\cdots,v_m\in \N$ natural numbers (weights for $X_1,\cdots,X_m$), and $D:C^\infty(M)\to C^\infty(M)$ a differential operator of Hörmander order $k$. The following are equivalent
\begin{enumerate}
\item For every $x\in M$ and $\pi \in \cT^*_x\cF\backslash\{0\}$, $\tilde{\sigma}(D,x,\pi)$ is injective on $C^\infty(\pi)$.
\item For every (or for some) $s\in \R$, and every distribution $u$ on $M$, $Du\in \tilde{H}^s(M)$ implies $u\in \tilde{H}^{s+k}(M)$.
\end{enumerate}
Furthermore if $M$ is compact, then the previous statements are equivalent to the following \begin{enumerate}
  \setcounter{enumi}{4}
  \item For every (or for some) $s\in \R$, $D:\tilde{H}^{s+k}(M)\to \tilde{H}^{s}(M) $ is left invertible modulo compact operators
\end{enumerate}
\end{thmx}
Differential operators satisfying the conditions of Theorem \ref{mainthmintro2} are called \textbf{maximally hypoellitpic differential operators.} Theorem \ref{mainthmintro2} immediately implies Hörmander's sum of squares theorem, by taking $v_1=\cdots=v_m=1$ and $v_{m+1}=2$. The injectivity of our symbol for Hörmander's sum of squares operator is trivial to verify. See also Corollary \ref{cor:jqsfdjioqm} for another immediate application of Theorem \ref{mainthmintro2} generalising Hörmander's theorem. Let us give right away a very simple yet nontrivial example which shows the strength of Theorem \ref{mainthmintro2}.
\begin{ex}\label{ex:intro}Let $k,n\in \N$ be natural numbers. We consider $\partial_x$ and $x^k\partial_y$ on $\R^2$. We assign the weight $1$ to $\partial_x$ and $n$ to $x^k\partial_y$. Then take \begin{equation}\label{eqn:qhsdluifjqdsf}
 D=(-1)^{n(k+n)}\partial_x^{2n(k+n)}+(-1)^{k+n}(x^k\partial_y)^{2(k+n)}+\lambda\partial_y^{2n}+D',
\end{equation}  where $D'$ is any differential operator of Hörmander order $<2n(k+n)$ and $\lambda\in \C$. A simple computation using our principal symbol shows that $D$ is maximally hypoelliptic if and only if \begin{equation}\label{eqn:crit}
 (-1)^{n+1}\lambda\notin \mathrm{spec}((-1)^{n(k+n)}\partial_x^{2n(k+n)}+x^{2k(k+n)}),
\end{equation} where $(-1)^{n(k+n)}\partial_x^{2n(k+n)}+x^{2k(k+n)}$ is considered as an unbounded operator on $L^2\R$. Note that the spectrum of this operator is a discrete set converging to $\infty$.  Discreteness of the spectrum  of $(-1)^{n(k+n)}\partial_x^{2n(k+n)}+x^{2k(k+n)}$ is also a consequence of Theorem \ref{mainthmintro2}, see Remark \ref{rem:diag_Schro}.
\end{ex}
We refer the reader to \cite[Sections I.2 and I.3]{HelfferNourrigatconj} for more applications of Theorem \ref{mainthmintro2}. Finally, we prove the following theorem whose counterpart for elliptic operators is obvious. It allows us to deduce maximal hypoellipticity on a neighbourhood of $x$ from invertibility of the symbol at $x$.

\begin{thmx}\label{mainthmintro3}Let $M$ be a smooth manifold, $X_1,\cdots,X_m$ vector fields satisfying the Lie bracket generating condition, $v_1,\cdots,v_m\in \N$ natural numbers (weights for $X_1,\cdots,X_m$), and $D$ a differential operator. Let $x\in M$. If for every $\pi \in \cT^*_x\cF\backslash\{0\}$, $\tilde{\sigma}(D,x,\pi)$ is injective, then for some open neighbourhood $U\subseteq M$ of $x$, $\tilde{\sigma}(D,y,\pi)$ is injective for every $y\in U$, $\pi \in \cT^*_y\cF\backslash\{0\}$. In particular $D$ is maximally hypoelliptic on $U$.
\end{thmx}
All the above results extend to differential operators with coefficients in a vector bundle. It is worth reiterating that in addition to showing the hypoellipticity of the polynomial differential operators considered by Helffer and Nourrigat, the method of proof here provides a pseudodifferential calculus adapted to such operators with a well-defined notion of principal symbol.  The existence of this calculus is essential for applications.  For instance, the second author has shown that one can use the associated pseudodifferential calculus to prove a topological index formula for maximally hypoelliptic differential operators \cite{IndexOmar}.  Similarly, one can obtain a complete description of the leading term of the heat kernel expansion of the above maximally elliptic differential operators \cite{MohsenHeatKer}.

Finally we mention that the topic of constructing pseudodifferential operators and parametricies for differential operators on nilpotent Lie groups, or more generally on manifolds has been studied by many people including \cite{Taylor,BeaGre,MelinFirst,RotSte,FollandStein,FischerDefect,MR3362017,MR3469687,RotschildSinglePaper,ewert2021pseudodifferential,Mel82,GoodmanBook,Dynin,BO,Glowacki2,Dynin2,TaylorBook,Cummins,PongeMemoirs,BahFermClotilde,YunVan:PsiDOs,dave2017graded,MelroseEps1,ChrGelGloPol,HelfferRockland,BealsRocklandConjNecessary}.   
\paragraph{Structure of the paper.}\begin{itemize}

\item In Section \ref{sec:filtfol} we define generalised distributions, filtered foliations, the Helffer-Nourrigat cone, and our principal symbol. We then explain the subtleties with the principal symbol in our calculus.
\item In Section \ref{sec:groupoid}, we define a $C^*$-algebra $C^*\aF$ which plays a very important role in the proof of Theorem \ref{mainthmintro2}. We also prove Theorem \ref{thm intro char set}.
\item In Section \ref{sec:Pseudodifferential_Operators} we define a pseudodifferential calculus associated to weighted vector fields satisfying the Lie bracket generating condition. We prove that this calculus satisfies the standard properties expected of a pseudodifferential calculus. Finally we prove Theorem \ref{mainthmintro2} and \ref{mainthmintro3}.
\item In Appendix \ref{appendixA}, we prove Theorems \ref{second_main_tech_thm} and  \ref{main_tech_thm}.  These are technical differential geometric results which ensure that our pseudodifferential calculus is closed under composition and adjoint. The appendix can be read immediately after Section \ref{sec: bisub adiab}.
 \end{itemize}
\paragraph{Acknowledgments.} Part of this work was done while O.~Mohsen was a postdoc in Muenster university and was funded by the Deutsche Forschungsgemeinschaft---Project-ID 427320536, SFB 1442---as well as Germany's Excellence Strategy EXC 2044 390685587, Mathematics M\"{u}nster: Dynamics-Geometry-Structure.

R.~Yuncken was supported by the CNRS PICS project OpPsi, which also funded visits by the other authors. R.~Yuncken was also supported by the project SINGSTAR of the Agence Nationale de la Recherche (ANR-14-CE-0012-01)

This project began with discussions with E.~van Erp. We would like to recognize his contributions and thank him profoundly for his involvement.

It is impossible to overstate the influence that the work of C.~Debord and G.~Skandalis has had on the present research.  We would like to emphasize the historical importance of Debord and Skandalis’s work, which instigated the deformation groupoid approach to pseudodifferential calculus.  Additionally, all the authors have benefited from extensive personal discussions with them.

We would like to thank B.~Helffer and J.~Nourrigat for their active interest in this work. We thank J.-M. Lescure and S. Vassout for their valuable remarks. 
\paragraph*{Conventions}
\begin{enumerate}
\item  In this article, if $G$ is a Lie group, then its Lie algebra is the space of \textit{right} invariant vector fields on $G$. This convention differs by a sign from the one usually used in Lie group theory but agrees with the one usually used in Lie groupoid theory. A consequence of this convention is that the Baker-Campbell-Hausdorff formula is given by \begin{equation}\label{eqn:BCHintro}
  \BCH(X,Y) = X+Y-\frac{1}{2}[X,Y] + \frac{1}{12}[X,[X,Y]]-\frac{1}{12}[Y,[Y,X]]+\cdots
\end{equation}
We will call the above formula the BCH formula. If $\mathfrak{g}$ is a nilpotent Lie algebra, then the BCH formula is a finite sum which defines a group law on $\mathfrak{g}$ making it a simply connected Lie group. Throughout the article, we will treat $\mathfrak{g}$ as both a Lie algebra and a Lie group.
\item Throughout the article, especially in Sections \ref{sec:diff_multi}, \ref{sec:dis_groups}, \ref{sec:princip symb}, \ref{sec:Parametrised}, we will define various unbounded multipliers on various $C^*$-algebras. Unless the multiplier is bounded, we will never take the closure of the graph. For our applications, the natural domain of 'smooth` densities, functions, vectors etc will be sufficient.
\end{enumerate} 
\section{Filtered foliations}\label{sec:filtfol}
In this section we give the definition of a filtered foliation, the Helffer-Nourrigat cone and our principal symbol. We stress that the notion of a filtered foliation is designed to describe intrinsically the notion of `Hörmander's vector fields with weights.' This section is organized as follows.
\begin{itemize}
\item In Section \ref{sec:generalised_distr}, we give some preliminaries on modules of vector fields.
\item In Section \ref{sec:adiabatic_foliation}, we define filtered foliations, the osculating Lie algebras and Lie groups.
\item In Section  \ref{sec:Debord-Skandalis}, we define an $\R_+^\times$-action which plays a fundamental part throughout the paper.
\item In section \ref{sec:characteristicset}, we define the Helffer-Nourrigat cone.
\item In Section \ref{sec:differential operators}, we define our principal symbol.
 \item In Section \ref{sec:Hormander's Sum of Squares Theorem}, we give some examples of of maximally hypoelliptic differential operators.
\end{itemize}	
	\subsection{Generalized distributions and singular foliations}\label{sec:generalised_distr}
	 Let us recall a few things and set the notation.
	\begin{enumerate}
		\item In this article, $M$ will be a smooth manifold without boundary. We denote by $\cX(M)$ ($\cX_c(M)$) the $C^{\infty}(M,\R)$-module of \textit{real} vector fields (with compact support) of $M$. We use $C^\infty(M,E)$ ($C^\infty_c(M,E)$) to denote the space of smooth sections (with compact support) of a vector bundle $E\to M$.
\item  If $X\in \cX(M)$, $x\in M$, then $\exp(X)\cdot x$ denotes the time one flow of $X$ starting from $x$, whenever it is well-defined.		
		\item If $X_1,\cdots,X_k\in \cX(M)$, then we will write $\cD=\langle X_1,\cdots,X_k\rangle$ for the $C^\infty(M,\R)$-module consisting of $\sum_{i=1}^k f_i X_i$ with $f_i\in C^\infty_c(M,\R)$.

		\item 	Let $\cD$ be a $C^{\infty}(M,\R)$-submodule of $\cX_c(M)$, $U\subseteq M$ an open subset.  We say that a family of vector fields $X_1,\ldots,X_k\in \cD$ generates $\cD$ on $U$ if for any $Y\in\cD$ there is $f_1,\ldots,f_k \in C^\infty(M)$ such that $Y_{|U}=(\sum_{i=1}^k f_iX_i)_{|U}$. We say that the family $X_1,\ldots,X_k$ generates $\cD$ at $p\in M$ if it generates $\cD$ on some neighbourhood of $p$.  
		
		The module $\cD$ is called locally finitely generated if at every point $p\in M$ there is a finite generating family.  When the cardinality of such a family $X_1,\ldots,X_k$ is the smallest possible, it is called a minimal generating family at $p$.
		\item A \textit{generalised distribution} is by definition a locally finitely generated $C^{\infty}(M,\R)$-submodule of $\cX_c(M)$. 

		\item
		Let $\cD$ be a generalised distribution and $p\in M$. The fiber of $\cD$ at $p$ is the quotient vector space
			\begin{equation}\label{eqn:Ip}
					\cD_p = \cD / I_p \cD,\quad\text{where}\quad I_p = \{f\in C^\infty(M,\R) : f(p)=0\}.
				\end{equation}

It is a finite dimensional vector space because $\cD$ is locally finitely generated. If $X\in\cD$, then we write $\fib{X}_p$ for the class of $X$ in the fiber $\cD_p$.	

		
		The following result was proved in \cite[Proposition 1.5]{AS1} for singular foliations. The proof doesn't use Lie brackets, so it applies equally to generalised distributions.
		
		\begin{prop}
		\label{lem:generating_family}~
	 If $\cD$ is a generalised distribution, then a family of vector fields $ X_1,\ldots, X_k \in\cD$ generates $\cD$ at $p\in M$ if and only if $\fib{X_1}_p, \ldots, \fib{X_k}_p$ spans the fiber $\cD_p$.  It is a minimal generating family at $p$ if and only if $ \fib{X_1}_p, \ldots, \fib{X_k}_p$ is a basis of $\cD_p$.
		\end{prop}
		
\item We define $\mathcal{D}^*:=\bigsqcup_{p\in M}\cD_p^*$. For every $X\in \cD$, let $$\langle \cdot,X\rangle:\cD^*\to \R,\quad\xi\in \cD^*_p\mapsto\xi([X]_p).$$ We equip $\cD^*$ with the weakest topology such that the natural projection $\pi:\cD^*\to M$ and the maps $\langle \cdot,X\rangle$ for every $X\in \cD$ are continuous. By \cite[Section 2.2]{AS2}, this topology makes $\cD^*$ a locally compact Hausdorff second countable space.

\item An automorphism of $\cD$ is a diffeomorphism $\phi:M\to M$ such that the pushforward of vector fields $\phi_*:\cX_c(M)\to \cX_c(M)$ maps $\cD$ bijectively to itself. It thus induces maps $\phi_*:\cD_p \to \cD_{\phi(p)}$ between fibers, and by duality, a homeomorphism $\phi^*$ of $\cD^*$.
\item A \textit{singular foliation} is a generalised distribution $\cF$ which is closed under Lie brackets. The following lemma which is straightforward to check is extremely important.
\begin{lemma}
\label{lem:stationary_subalgebra}
Let $\mathcal{F}$ be a singular foliation, and $p$ a stationary point, \emph{i.e.,} $X(p)=0$ for all $X\in \cF$. Then $\mathcal{F}_p$ is a Lie algebra with the Lie bracket $[[X]_{p},[Y]_{p}]:=[[X,Y]]_p$.
\end{lemma}				
\end{enumerate}

\subsection{Filtered foliations}\label{sec:adiabatic_foliation}
	\begin{definition}\label{dfn:filtfol}
		A \emph{filtered foliation} of depth $N\in \N$ on a smooth manifold $M$ is a filtration by generalised distributions
		$$
		0=\cF^0 \subseteq\cF^1 \subseteq \cF^2 \subseteq \ldots \subseteq \cF^N=\cX_c(M)
		$$
		such that
		\begin{equation}\label{eq:Lie bracket equation}
		[\cF^i,\cF^j] \subseteq \cF^{i+j},\quad \forall\ i,j\in \N.	
		\end{equation}
Here, and throughout the article, we use the convention $\cF^{n} = \cX_c(M)$ for all $n\geq N$.
	\end{definition} 
\begin{ex}\label{ex:filtered} Let $X_1,\ldots,X_m$ be smooth vector fields on $M$ satisfying Hörmander's Lie bracket generating condition and let $v_1,\cdots,v_m\in \N$ be weights for each $X_i$. We define $\cF^j$ to be the generalised distribution generated by all iterated Lie brackets  $[X_{i_1},\ldots[X_{i_{k-1}},X_{i_k}]\ldots ]$ such that $v_{i_1}+\cdots+v_{i_k}\leq j$. Hörmander's Lie bracket generating condition implies that $\cF^N=\cX_c(M)$ for some $N$.\footnote{This is only true locally. The value of $N$ may be infinite if $M$ isn't compact. Since maximal hypoellipticity is a local notion, we can ignore this issue.} Inversely one can easily see that any filtered foliation where each $\cF^i$ is finitely generated is obtained this way. So locally, filtered foliations are an intrinsic way to define a family of weighted vector fields satisfying Hörmander's Lie bracket generating condition.   
	\end{ex}
The central geometric object in this paper is the adiabatic foliation associated to $\cF^\bullet$ which we now define. For $X\in \cX(M)$, we denote by $\tilde{X}\in \cX(M\times \R_+)$ the vector field $$\tilde{X}(x,t)= (X(x),0).$$ If $\cD \subseteq \cX_c(M)$ is a generalised distribution, we write $\widetilde{\cD}$ for the generalised distribution on $M\times \R_+$, generated by $\tilde{X}\in\Gamma(TM\times \R_+)$ for $X\in\cD$.

		\begin{definition}
		Let $\cF^\bullet$ be a filtered foliation.   The \emph{adiabatic foliation} associated to $\cF^\bullet$ is the singular foliation on $M\times \R_+$ given by
		\[
		\aF  := t\widetilde{\cF^1} + t^2\widetilde{\cF^2} + \ldots + t^N \widetilde{\cF^N},
		\]
where  $t \in C^\infty(M\times \R_+)$ denotes the smooth projection onto the second variable. 
\end{definition}It is clear that $\aF$ is locally finitely generated. It is involutive because of \eqref{eq:Lie bracket equation}. Hence $\aF$ is a singular foliation. Let $p\in M$. The point $(p,0)$ is stationary. So by Lemma \ref{lem:stationary_subalgebra}, $\aF_{(p,0)}$ is a Lie algebra which we now describe. Since $\cF^{i-1}\subseteq \cF^{i},$ it follows that $\cF^{i-1}_{p}$ maps naturally to $\cF^{i}_{p}$. This map is not injective in general. Nevertheless we denote by $\cF^i_{p}/\cF^{i-1}_{p}$ the quotient of $\cF^i_{p}$ by the image of $\cF^{i-1}_{p}$, and we write
$$
\gr(\cF)_{p}:=\bigoplus_{i=1}^{i=N}\frac{\cF^i_{p}}{\cF^{i-1}_{p}}=\bigoplus_{i=1}^{i=N}\frac{\cF^i}{\cF^{i-1}+I_p\cF^i},
$$where $I_p$ is defined in \eqref{eqn:Ip}.  Note that, due to the non-injectivity of the maps $\cF^{i-1}_p \rightarrow \cF^{i}_p$, the dimension of $\gr(\cF)_p$ can be strictly larger than $\dim(M)$. 

If $X\in \cF^i$, we will use $[X]_{i,p}$ to denote the class of $X$ in $\cF^i_p/\cF^{i-1}_p \subseteq \gr(\cF)_p$.
\begin{prop}\label{prop:eta_iso} 
The map
\begin{align*}
\eta_{p}:\gr(\cF)_p\to \aF_{(p,0)},\quad \sum_{i=1}^N [X_i]_{i,p}\to \left[\sum_{i=1}^N t^i \tilde{X}_i\right]_{(p,0)}\qquad \text{with }X_i\in \cF^i
\end{align*}
 is a well-defined isomorphism of vector spaces.
\end{prop}

\begin{proof}
If $X_i\in \cF^{i-1}$, then $t^iX_i=t \cdot t^{i-1}X_i$. Hence $[t^i\tilde{X}_i]_{(p,0)}=0$. Thus $\eta_p$ is well-defined. Injectivity and surjectivity follow from the definition of $\aF$. 
\end{proof}

We define a Lie algebra structure on $\gr(\cF)_p$ by declaring the map $\eta_{p}$ an isomorphism of Lie algebras. The Lie bracket is thus given by the formula $$[[X]_{i,p},[Y]_{j,p}]=\big[[X,Y]\big]_{i+j,p} \in \cF^{i+j}_p/\cF^{i+j-1}_p,\quad X\in \cF^i,Y\in \cF^j.$$ The resulting nilpotent Lie algebra is called the \emph{osculating Lie algebra} of $\cF$ at $p$.  It is nilpotent because if $i+j>N$, then $[[X]_{i,p},[Y]_{j,p}]=0$. Hence the space $\gr(\cF)_p$ is also a Lie group with a product via the BCH formula \eqref{eqn:BCHintro}.

\begin{prop}\label{prop:semi-cont D i over D i-1}The function  $p\mapsto \dim(\gr(\cF)_p)$ is upper semi-continuous.
\end{prop}
\begin{proof}We will prove the stronger assertion that $p\mapsto \dim(\cF^i_p/\cF^{i-1}_p)$ is upper semi-continuous for all $i$. Let $k\in \N$. It is straightforward to check that $\dim(\cF^i_p/\cF^{i-1}_p)\leq k$ if and only if there exists $X_1,\cdots,X_k\in \cF^i$, $Y_1,\cdots,Y_l\in \cF^{i-1}$ for some $s$ such that the $[X_1]_p,\cdots ,[X_k]_p,[Y_1]_p,\cdots,[Y_l]_p$ generate $\cF^{i}_p$. By Proposition \ref{lem:generating_family}, this implies that $X_1,\cdots,X_k,Y_1,\cdots,Y_l$ generate $\cF^i$ over an open neighbourhood of $p$. The semi-continuity of $p\mapsto \dim(\cF^i_p/\cF^{i-1}_p)$ follows.
\end{proof}
\begin{remark}The set $\{p:\dim(\Gr(\cD)_p)=\dim(M)\}$ is open by Proposition \ref{prop:semi-cont D i over D i-1}. It is also dense because it contains $\bigcap_{i=1}^NM_{i}$ where $M_{i}$ is the regular part of $\cF^i$ which is open and dense, see \cite[Proposition 1.5]{AS1}.
\end{remark}
\begin{ex}\label{exs:Filtered foliations group calculation}
Let $M=\R^2$, $N=3$,
			\[
\cF^1 = \langle \partial_x\rangle,\quad\cF^2 = \langle \partial_x,x\partial_y\rangle.
			\]
 Let $(a,b)\in M$. It is immediate to see that $$\cF^1_{(a,b)}=\R[\partial_x]_{(a,b)},\quad \cF^2_{(a,b)}=\R[\partial_x]_{(a,b)}\oplus\R[x\partial_y]_{(a,b)},\quad \cF^3_{(a,b)}=T_{(a,b)}M.$$ The natural map $\cF^2_{(a,b)}\to \cF^3_{(a,b)}$ is injective if and only if $a\neq 0$. It follows that 
 $$\gr(\cF)_{(a,b)}=\begin{cases}
  \R[\partial_x]_{1,(a,b)}\oplus \R[x\partial_y]_{2,(a,b)}\oplus 0,& \text{if}\;a\neq 0,\\
  \R[\partial_x]_{1,(a,b)}\oplus \R[x\partial_y]_{2,(a,b)}\oplus \R[\partial_y]_{3,(a,b)},&\text{if}\;a= 0.\end{cases}$$ 
 The group $\Gr(\cF)_{(0,b)}$ is the $3$-dimensional Heisenberg group for every $b\in \R$ because $[\partial_x,x\partial_y]=\partial_y$.
\end{ex}

\subsection{The Debord-Skandalis action}
\label{sec:Debord-Skandalis}
Let $\alpha$ be the $\R^\times_+$-action on $M\times \R_+$ given by $\alpha_\lambda(x,t) = (x,\lambda^{-1}t)$.  These maps are automorphisms of the adiabatic foliation, so they induce maps between the fibers of $\aF$.  At $t=0$, using Proposition \ref{prop:eta_iso}, we obtain an action by automorphisms of the osculating Lie algebras given by
\begin{equation}\label{eqn:DS_actionGR} 
 \alpha_{\lambda}\left(\sum_{i=1}^N[X_i]_{i,p}\right)=\sum_{i=1}^N\lambda^i[X_i]_{i,p}, \qquad (X_i\in\cF^i).
\end{equation}
Likewise, there is an induced action $\hat{\alpha}$ on the space $\aF^*$.  Again using Proposition \ref{prop:eta_iso}, we have
\begin{equation}\label{eqn:aF*} 
  \aF^* = \left( T^*M \times \R^\times_+\right) \sqcup \left(\gr(\cF)^*\times\{0\}\right),
\end{equation}
where $\gr(\cF)^* := \bigsqcup_{p\in M} \gr(\cF)^*_p$.  Under this identification, the $\R^\times_+$-action on $\aF^*$ is given by
\[
 \hat{\alpha}_\lambda (p,\xi,t) = (p,\xi,\lambda t), \qquad p\in M,\xi\in T^*_pM
\]
for $t>0$, while at $t=0$ it is given by the formula
 \begin{equation}
 \label{eqn:act dual space}
 \hat{\alpha}_\lambda(\xi)(X)=\xi(\alpha_{\lambda}(X)),\quad X\in \gr(\cF)_p,~\xi\in \gr(\cF)_p^*.
\end{equation}
We refer to the all of the above actions as the Debord-Skandalis action.
\subsection{The Helffer-Nourrigat cone}\label{sec:characteristicset}
Recall that $\aF^*$ is a locally compact Hausdorff topological space, see Section \ref{sec:generalised_distr}.g.
\begin{definition}\label{dfn:top char set} 
The \emph{Helffer-Nourrigat cone} at $p\in M$ is the set
$$\cT^*_p\cF=\{\xi \in\gr(\cF)^*_p:(p,\xi,0)\in \overline{T^*M\times \R^\times_+}\subseteq \aF^*\}.$$
We also let $\mathcal{T}^*\cF:=\bigsqcup_{p\in M}\mathcal{T}^*_p\cF\subseteq \gr(\cF)^*$.
\end{definition}
\begin{exs}
~\label{exs:charac_set}

\begin{enumerate}
\item In Example \ref{exs:Filtered foliations group calculation}, $\mathcal{T}^*\cF$ is equal to $\gr(\cF)^*$.

\item 
Consider 
$M=\R, N=3$ and $$\cF^1=\langle x^2\partial_x\rangle,\quad \cF^2=\langle x\partial_x\rangle.$$ 
Then 
$$
\gr(\cF)_p=\begin{cases}\R[x^2\partial_x]_{1,p}\oplus 0\oplus 0,& \text{if}\, p\neq 0\\\R[x^2\partial_x]_{1,p}\oplus \R[x\partial_x]_{2,p}\oplus \R[\partial_x]_{3,p},&\text{if} \, p=0\end{cases}.
$$ 
A sequence $(x_n,\eta_n,t_n)\in T^*M\times \R^\times_+$ converges to a point $(0,(\xi_1,\xi_2,\xi_3),0)$ in $\mathcal{T}^*_{0}\cF\times \{0\}$ if 
$$
x_n\to 0, \quad
t_n\to 0 , \quad
x_n^2t_n\eta_n\to \xi_1, \quad
x_nt_n^2\eta_n\to \xi_2, \quad
t_n^3\eta_n\to \xi_3.
$$
Hence $\xi_1\xi_3=\xi_2^2$. One can check that this is the only relation restricting the limit set. Therefore
\[
 \mathcal{T}^*_{p}\cF = 
 \begin{cases}
  \RR, & \text{if } p\neq0, \\
  \{(\xi_1,\xi_2,\xi_3)\in\R^3:\xi_1\xi_3=\xi_2^2\}, & \text{if } p=0.
 \end{cases}
\]
\item
The previous example can be made less artificial as follows.  Let
$$M=\R^2,\ N=4,\ \cF^1 = \langle \partial_x \rangle,\
  \cF^2=\langle \partial_x,x^2\partial_y\rangle,\ \cF^3=\langle  \partial_x,x\partial_y\rangle.$$ 
This is the filtered foliation associated to the H\"ormander-type operator $\partial_x^2 + x^2\partial_y$.  The osculating Lie algebras are
$$
\gr(\cF)_{(a,b)}=\begin{cases}\R[\partial_x]_{1,(a,b)}\oplus \R[x^2\partial_y]_{2,(a,b)} \oplus 0\oplus 0,& \text{if }\, a\neq 0\\
\R[\partial_x]_{1,(a,b)}\oplus \R[x^2\partial_y]_{2,(a,b)}\oplus \R[x\partial_y]_{3,(a,b)}\oplus \R[\partial_y]_{4,(a,b)},&\text{if } a=0\end{cases}.
$$ 
The Helffer-Nourrigat cone is then equal to
\[
 \mathcal{T}^*_{(a,b)}\cF = 
 \begin{cases}
  \RR^2, & \text{if } a\neq0, \\
  \{(\xi_1,\xi_2,\xi_3,\xi_4)\in\R^4:\xi_2\xi_4=\xi_3^2\}, & \text{if }a=0,
 \end{cases} 
\]
where $\xi_1,\xi_2,\xi_3,\xi_4$ are the dual variables to the generators of the osculating Lie algebras.

\item Consider $M=\R^2, N=2$, $$\cF^1=\langle x^2y^4\partial_x,x^6\partial_x,x^4y^2\partial_x,y^6\partial_x\rangle.$$ One has $\gr(\cF)_{(0,0)}=\R^6$ with basis corresponding to $x^2y^4\partial_x,x^6\partial_x,x^4y^2\partial_x,y^6\partial_x,\partial_x,\partial_y$. A simple computation shows that \begin{align*}
\mathcal{T}^*_{(0,0)}\cF&=\{(\xi_1,\xi_2,\xi_3,\xi_4,\xi_5,\eta)\in \R^6:\xi_1^3=\xi_2\xi_4^2, 
\ \xi_3^3=\xi_2^2\xi_4,\ \xi_i\xi_5\geq 0\  \forall i\},
\end{align*}
where $\xi_1,\cdots,\xi_5,\eta$ are the dual of the above generators respectively. This example shows that the Helffer-Nourrigat cone isn't necessarily Zariski closed even if all the vector fields are polynomial.
 \item Consider $M=\R^2,N=3$ and  $$\cF^1=\left\langle \left(\frac{x}{\sqrt{x^2+y^2}}+2\right)e^{\frac{-2}{x^2+y^2}}\partial_x\right\rangle,\quad\cF^2=\left\langle e^{\frac{-1}{x^2+y^2}}\partial_x\right\rangle.$$ One has $\gr(\cF)_{(0,0)}=\R^4$ with basis corresponding to $\left(\frac{x}{\sqrt{x^2+y^2}}+2\right)e^{\frac{-2}{x^2+y^2}}\partial_x,e^{\frac{-1}{x^2+y^2}}\partial_x,\partial_x,\partial_y$. A simple computation shows that \begin{align*}
   \mathcal{T}^*_{(0,0)}\cF&=\{(\xi_1,\xi_2,\xi_3,\xi_4)\in\R^4:\exists \lambda\in [1,3],\ \lambda\xi_2^2=\xi_1\xi_3\}.
\end{align*} 

\end{enumerate}

\end{exs}

\begin{prop}\label{prop:inv coadjoint orbit}Let $p\in M$, $\xi\in \mathcal{T}^*_p\cF$. Then \begin{itemize}
\item If $\lambda\in \R_+^\times$, then $\alpha_\lambda(\xi)\in \mathcal{T}^*_p\cF	$.
\item If $\mu\in \R$, then $\mu\xi\in \mathcal{T}^*_p\cF $.
\item If $g\in \Gr(\cF)_p$ , then $Ad^*(g)\xi\in  \mathcal{T}^*_p\cF$.
\end{itemize}
\end{prop}
\begin{proof}
Since $T^*M\times \R^\times_+ \subset \aF^*$ is invariant under the Debord-Skandalis action, so is its limit set.  Thus $\mathcal{T}^*_p\cF$ is stable under the Debord-Skandalis action. It is also stable under the standard (ungraded) vector space dilations, because if $(p_n,\xi_n,t_n)\in T^*M\times \R^\times_+$ converges to $(p,\xi,0)$ in $\mathcal{T}^*\cF\times \{0\}$, then $(p_n,\mu \xi_n,t_n)$ converges to $(p,\mu \xi,0)$ for any $\mu\in\RR$. 

For the coadjoint action, let $X\in \cF^i$ for some $i$. Consider $t^i\tilde{X}$ as a vector field on $M\times \R_+$, and let $\phi:M\times  \R_+\to M\times \R_+$ be its flow at time $1$. By \cite[Proposition 1.6]{AS1}, $\phi$ is an automorphism of $\aF$. After identifying $\aF_{(p,0)}$ with $\gr(\cF)_p$ using Proposition \ref{prop:eta_iso}, we claim that the induced action $\phi_*$ on $\gr(\cF)_p$ is the adjoint action $Ad_{[X]_p}$.  To see this, let $\phi_s:\gr(\cF)_p\to \gr(\cF)_p$ be the map which is associated to  $st^i\tilde{X}$. The maps $\phi_s$ form a $1$-parameter group. Their derivative at $0$ is the adjoint action $ad_{[X]_p}$. Hence $\phi_1=e^{ad_{[X]_p}}=Ad_{[X]_p}$. Since $\phi$ induces a homeomorphism of the space $\aF^*$ which leaves $M\times \R_+^\times$ fixed, it follows that it fixes the Helffer-Nourrigat cone. The result follows.
\end{proof}

\begin{remark}
\begin{enumerate}
\item By the orbit method \cite{KirillovArticle,BrownArticleTopOrbitMethod}, Proposition \ref{prop:inv coadjoint orbit} allows us to view $\mathcal{T}^*_p\cF$ as a closed subset of the unitary dual $\widehat{\Gr(\cF)}_p$ of $\Gr(\cF)_p$. We will make this identification frequently.
\item One can define $\alpha_{-1}:\gr(\cF)^*_p\to \gr(\cF)^*_p$ using \eqref{eqn:act dual space} and \eqref{eqn:DS_actionGR}. In general $\mathcal{T}^*_p\cF$ isn't invariant under $\alpha_{-1}$ as Example \ref{exs:charac_set}.d shows.
\end{enumerate}  
\end{remark}
\subsection{Differential operators and principal symbol}\label{sec:differential operators}
Let $E\to M$ be a vector bundle, $\Diff(M,E)$ the algebra of differential operators $C^\infty(M,E)\to C^\infty(M,E) $ not necessarily of compact support. Let $\Diff^k_{\cF}(M,E)$ denote the vector space of differential operator $D$ such that for every $f\in C^\infty_c(M)$, $fD$ can be written as sum of monomials $\alpha \nabla_{Y_{1}}\cdots \nabla_{Y_l}$ with $\alpha\in C^\infty_c(M,\End(E))$, $Y_i\in \cF^{a_i}$ for some $a_i\in \N$ such that $\sum_{i=1}^l a_i\leq k$. Here $\nabla$ is any connection on $E$. If $E$ is trivial, then we write $\Diff^k_{\cF}(M)$ instead of $\Diff^k_{\cF}(M,E)$. Then:
 \begin{enumerate}
\item $\Diff_{\cF}^0(M,E)=C^\infty(M,\End(E))$.
\item  For every $i\in \{1,\cdots,N\}$, $X\in \cF^i$, one has $\nabla_{X}\in \Diff_{\cF}^i(M,E)$.
\item For every $i,j$, one has $\Diff_{\cF}^i(M,E)\Diff_{\cF}^j(M,E)\subseteq \Diff_{\cF}^{i+j}(M,E)$.
\item If $f\in C^\infty(M)$, $D\in \Diff_{\cF}^i(M,E)$ with $i\geq1$, then $[f,D]\in \Diff_{\cF}^{i-1}(M,E)$.
\end{enumerate}
\begin{definition}
 We say that $D$ has Hörmander's order $k$ if $D\in \Diff^k_{\cF}(M,E)\backslash \Diff^{k-1}_{\cF}(M,E)$.
\end{definition}
\paragraph{Principal symbol.} Let $ D\in \Diff^k_{\cF}(M,E),$ $p\in M$ and $\pi$ an irreducible unitary representation of $\Gr(\cF)_p$ on a Hilbert space denoted $L^2\pi$. We denote by $C^\infty(\pi)\subseteq L^2\pi$ the subspace of smooth vector. If $X\in \gr(\cF)_p$, then the differential of $\pi$ at $X$ gives a linear map $$d\pi(X):C^\infty(\pi)\to C^\infty(\pi),\quad  d\pi(X)v=\frac{d}{dt}\Bigr|_{\substack{t=0}}\pi(-tX)v,\quad v\in C^\infty(\pi).$$ It satisfies $$d\pi([X,Y])=[d\pi(X),d\pi(Y)],\quad X,Y\in \gr(\cF)_p.$$The principal symbol of $D$ at $\pi$ is a linear map $$\sigma^k(D,p,\pi):E_p\otimes C^\infty(\pi)\to E_p\otimes  C^\infty(\pi)$$ defined as follows. Let $f\in C^\infty_c(M)$ such that $f(p)=1$. Then $fD$ can be written as sum of monomials $\alpha \nabla_{Y_{1}}\cdots \nabla_{Y_l}$ as above. The symbol $\sigma^k(D,p,\pi)$ is equal to the sum where we replace each monomial by \begin{center}
$\alpha(p)\otimes d\pi([Y_1]_{a_1,p}) \cdots d\pi([Y_l]_{a_l,p})$ 
\end{center} and we only sum over monomials such that $\sum_{i=1}^l a_i=k$. The following theorem establishes that the principal symbol is well defined when $\pi \in \cT^*_p\cF$. It will be proved at the end of Section \ref{sec:groupoid}.
\begin{thm}\label{thm:sybmol_welldefined}
Let $D\in \Diff^k_{\cF}(M,E)$, $p\in M$. If $\pi$ is an irreducible representation of $\Gr(\cF)_p$ which corresponds to an element of $\cT^*_p\cF$ by Kirillov's orbit method, then the principal symbol $\sigma^k(D,p,\pi)$ is well defined, i.e., doesn't depend on the choice of $f$ nor on the connection on $E$ nor on the way $fD$ is written as a sum of monomials of the form $\alpha \nabla_{Y_{1}}\cdots \nabla_{Y_l}$.
\end{thm}
We end this section by showing two subtleties with the definition of $\sigma^k(D,p,\pi)$ which are: 
\begin{enumerate}
\item In general, there can exist $D\in\Diff^k_\cF(M)$ such that $\sigma^k(D,p,\pi)=0$ for every $p\in M$ and $\pi\in \cT^*_p\cF$ yet $D\notin \Diff^{k-1}_\cF(M)$. This makes it more subtle to construct parametrices. We will ultimately resolve this issue by proving Theorem \ref{thm:vanish_symb}.
\item In general $\sigma^k(D,p,\pi)$ is not well defined for $\pi\notin \cT^*_p\cF $. Hence Theorem \ref{thm:sybmol_welldefined} is not trivial.
\end{enumerate}
These phenomena are observed in the following examples.
 \begin{exs}~\label{exs:counter_examples} 
 \begin{enumerate}
\item 
Let $M=\R,\, N=2,\, \cF^1=\langle x^2\partial_x\rangle$. Then $$\gr(\cF)_p=\begin{cases}\R[x^2\partial_x]_{1,p}\oplus 0,\quad \text{if}\; p\neq 0\\\R[x^2\partial_x]_{1,p}\oplus \R[\partial_x]_{2,p},\quad\text{if} \; p=0\end{cases}.$$ A straightforward computation shows that $ \cT^*_p\cF=\gr(\cF)_p^*$ for all $p\in M$. One has $x\partial_x\in\cF^2\subseteq \Diff^2_{\cF}(M)$ and $\sigma^2(x\partial_x,p,\pi)=0$ for every $p\in M,\pi\in  \cT^*_p\cF$, yet $x\partial_x\notin \Diff^1_{\cF}(M)$.
  \item  
  Consider Example \ref{exs:charac_set}.b. Let $D=(x^2\partial_x)(\partial_x)-(x\partial_x)^2\in \Diff^4_{\cF}(M)$. Let $\pi=(\xi_1,\xi_2,\xi_3)\in \gr(\cF)_0^*$. Then the way $D$ is written as $(x^2\partial_x)(\partial_x)-(x\partial_x)^2$ implies that $\sigma^4(D,0,\pi)=\xi_1\xi_3-\xi_2^2$. Yet since $(x^2\partial_x)(\partial_x)-(x\partial_x)^2=-x\partial_x\in \cF^2$, it follows that writing $D$ as $-x\partial_x$ gives $\sigma^4(D,0,\pi)=0$. Hence $\sigma^4(D,0,\pi)$ can only be well defined for $\pi$ such that $\xi_1\xi_3=\xi_2^2$ which is precisely the set $\cT^*_0\cF$.
  \end{enumerate}
\end{exs}
We remark that if $\pi$ is an irreducible unitary representation corresponding to an orbit $O\subseteq \mathfrak{g}^*$, then Kirillov's construction of $\pi$ is by letting $\pi$ act on $L^2\R^{\dim(O)/2}$. A result of Kirillov is that $C^\infty$ vector are precisely Schwartz functions on $\R^{\dim(O)/2}$, see \cite[Corollary 4.1.2]{BookNilpotentGroups}. This is convenient when trying to check Criterion a in Theorem \ref{mainthmintro2}.
\subsection{Examples of maximally hypoelliptic differential operators}\label{sec:Hormander's Sum of Squares Theorem}
Let $M$ be a smooth manifold and $X_1,\cdots,X_m,X_{m+1}$ be \text{real} vector fields which satisfy Hörmander's Lie bracket generating condition of rank $N$. Let $v_1,\cdots,v_{m+1}\in \N$ be weights for $\mathbb{N}$. We suppose that $v_{i}$ is even for all $i\in \{1,\cdots,m\}$ and $v_{m+1}$ is odd. As in Example \ref{ex:filtered}, we define a filtered foliation $\cF^\bullet$ of depth $N\max(v_1,\cdots,v_{m+1})$ on $M$ by declaring $X_{i}$ to be of order $v_{i}$. The Lie algebra $\gr(\cF)_p$ is then a graded nilpotent Lie algebra generated by $[X_i]_p\in \gr(\cF)_p$. 
\begin{prop}
 \label{prop:Hormander_symbol}
Let $G$ be connected simply connected nilpotent Lie group with lie algebra $\lie{g}$ and let $x_i\in \mathfrak{g}$ be a generating family for $\mathfrak{g}$. Then for any non-trivial irreducible unitary representation $\pi$ of $G$, $\pi\left(\sum_{i=1}^m (-1)^\frac{v_i}{2}x_i^{v_i}+x_{m+1}^{v_{m+1}}\right)$ is injective.
 \end{prop}
 \begin{proof}
 Let $w$ be a smooth vector in the kernel of $\pi\left(\sum_{i=1}^m (-1)^\frac{v_i}{2}x_i^{v_i}+x_{m+1}^{v_{m+1}}\right)$. Since the operator $\pi\left(\sum_{i=1}^m (-1)^\frac{v_i}{2}x_i^{v_i}\right)$ is positive and $\pi(x_{m+1}^{v_{m+1}})$ is skew-adjoint, it follows that $w$ is in the kernel of $\pi(x_i)$ for each $i$. Hence $w\in \ker(\pi(\mathfrak{g}))$. Since $\pi$ is non-trivial irreducible, we deduce that $w=0$.
\end{proof}
It follows that the hypothesis of Theorem \ref{mainthmintro2}.a is trivially satisfied for the differential operator $\sum_{i=1}^m (-1)^\frac{v_i}{2} X_i^{a_i}+ X_{m+1}^{v_{m+1}}$. Hence a corollary of Theorem \ref{mainthmintro2}.
\begin{cor}\label{cor:jqsfdjioqm}
The operator $\sum_{i=1}^k (-1)^\frac{v_i}{2} X_i^{v_i}+ X_{m+1}^{v_{m+1}}$ is maximally hypoelliptic.
\end{cor} 
Notice that we didn't need to calculate the Helffer-Nourrigat cone because injectivity of the symbol holds for every non-trivial irreducible unitary representation. But in general the Helffer-Nourrigat cone can be a proper subset of $\gr(\cF)^*$, as Example \ref{exs:charac_set}.c shows.
\begin{ex}\label{ex:intro_det}Consider Example \ref{ex:intro} from the introduction. Let $\cF^\bullet$ be the associated filtered foliation of depth $n+k$. One can check that \begin{align*}
\gr(\cF)_{(a,b)}=\begin{cases}
  \R[\partial_x]_{1,(a,b)}\oplus \R[x^k\partial_y]_{n,(a,b)}\oplus 0,& \text{if}\;a\neq 0,\\
  \R[\partial_x]_{1,(a,b)}\oplus \R[x^k\partial_y]_{n,(a,b)}\oplus\cdots \oplus \R[\partial_y]_{n+k,(a,b)},&\text{if}\;a= 0,\end{cases}
\end{align*}
The only non trivial Lie bracket relation on $\gr(\cF)_{(a,b)}$ are $$\Big[[\partial_x]_{1,(a,b)},[x^i\partial_y]_{n-i+k,(a,b)}\Big]=i[x^{i-1}\partial_y]_{n-i+k+1,(a,b)}$$ for all $i\geq 1$. A straightforward computation shows that  \begin{align*}
\cT_{(a,b)}^*\cF=\begin{cases}
  \gr(\cF)_{(a,b)}^*& \text{if}\;a\neq 0,\\
  \{(\eta,\xi_k,\cdots,\xi_0)\in \R^{1+k}:\exists \alpha\in \R,\beta\in \R, \xi_i=\alpha^i\beta\forall i\}\sqcup\{(\eta,\xi_k,0,\cdots,0)\in \R^2\times \{0\}\},&\text{if}\;a= 0,\end{cases}
\end{align*}
where we follow the convention $0^0=1$. One easily shows that the coadjoint orbits in $\cT_{(0,b)}^*\cF$ are \begin{itemize}
\item $\{(\eta,\xi_k,0,\cdots,0)\}$. This corresponds to the $1$-dimensional representation $\pi$, with $$d\pi([\partial_x]_{1,(0,b)})=\sqrt{-1}\eta,\quad d\pi([x^k\partial_y]_{n,(0,b)})=\sqrt{-1}\xi_k.$$
\item for each $\eta\in \R,\beta\in \R^\times$, we have the orbit $\{(\eta,\xi_k,\cdots,\xi_0)\in \R^{1+k}: \exists\alpha\in \R,\xi_i=\alpha^i\beta\forall i\}$. It corresponds to the representation $\pi$ on $L^2\R$, with 
$$d\pi([\partial_x]_{1,(0,b)})=\partial_x,\quad d\pi([x^i\partial_y]_{n-i+k,(0,b)})=\sqrt{-1}\beta x^i.$$
\end{itemize}
We can now compute the principal symbol of $$D=(-1)^{n(k+n)}\partial_x^{2n(k+n)}+(-1)^{k+n}(x^k\partial_y)^{2(k+n)}+\lambda\partial_y^{2n}+D',$$where $D'$ is of Hörmander's order $<2n(k+n)$. We have \begin{itemize}
\item If $a\neq 0$, then $\sigma^{2n(k+n)}(D,(a,b),(\eta,\xi))=\eta^{2n(k+n)}+\xi^{2(k+n)}$ which obviously vanishes only at $(\eta,\xi)=(0,0)$.
\item If $a=0$, $\pi$ corresponds to $\{(\eta,\xi_k,0,\cdots,0)\}$, then $\sigma^{2n(k+n)}(D,(0,b),\pi)=\eta^{2n(k+n)}+\xi_k^{2(k+n)}$ which again vanishes only when $\pi$ is trivial.
\item If $a=0$, $\pi$ corresponds to $\{(\eta,\xi_k,\cdots,\xi_0)\in \R^{1+k}: \exists\alpha\in \R,\xi_i=\alpha^i\beta\forall i\}$, then $$\sigma^{2n(k+n)}(D,(0,b),\pi)=(-1)^{n(k+n)}\partial_x^{2n(k+n)}+x^{2k(k+n)}\beta^{2(k+n)}+\lambda (-1)^n\beta^{2n}.$$ By Homogeneity of the principal symbol with respect to the Debord-Skandalis action, it is enough to check injectivity for $\beta=\pm 1$. Since if $\lambda=0$, $D$ is maximally hypoelliptic by Corollary \ref{cor:jqsfdjioqm}, it follows that $(-1)^{n(k+n)}\partial_x^{2n(k+n)}+x^{2k(k+n)}$ has compact resolvent and hence has a spectrum which is discrete converging to $+\infty$, see Remark \ref{rem:diag_Schro}. We thus deduce from Theorem \ref{mainthmintro2} that $D$ is maximally hypoelliptic if and only if Criterion \eqref{eqn:crit} is satisfied.
\end{itemize}
\end{ex}
\section{The \texorpdfstring{$C^*$}{Cstar}-algebra \texorpdfstring{$C^*\aF$}{adiabatic foliation}}
	\label{sec:groupoid}
 The $C^*$-algebra of a singular foliation was introduced by the first author and Skandalis \cite{AS1}. In this section we describe the $C^*$-algebra of $\aF$. We mostly follow the construction of the first author and Skandalis, slightly simplified due to the special nature of $\aF$. This section is organized as follows \begin{itemize}
\item In Section \ref{sec: bisub adiab}, we introduce graded basis. These are the local charts on which we will construct oscillatory integrals in Section \ref{sec:Pseudodifferential_Operators}, in order to define pseudodifferential operators. Graded basis are special cases of bisubmersions (cf. \cite{AS1}).
\item In Section \ref{sec:densities}, we recall all the necessary properties of densities that will be needed throughout the article.
\item In Section \ref{subsec:c0R}, we define the $C^*$-algebra $C^*\aF$.
\item In Section \ref{sec:connection_Helffer}, we show the connection between $C^*\aF$ and the Helffer-Nourrigat cone.
\item In Section \ref{sec:diff_multi}, we prove Theorem \ref{thm:sybmol_welldefined}.
 \end{itemize}
 \subsection{Graded basis}\label{sec: bisub adiab}
\begin{definition}\label{dfn:graded basis}
A \emph{graded basis} is a $4$-tuple $(V,\natural,\mathbb{U},U)$ where  
\begin{enumerate}
\item $V=\oplus_{i=1}^NV^i$ is a graded finite dimensional real vector space equipped with the \emph{graded dilations} $\alpha_\lambda(\sum_{i=1}^N v_i)=\sum_{i=1}^N\lambda^iv_i$ for $v_i\in V^i$, $\lambda\in \R_+$. 
\item $\natural:V\to \cX_c(M)$ is a linear map,
\item $U\subseteq M$ is an open subset,
\item $\mathbb{U}\subseteq V\times U\times \R_+$ is an $\R^\times_+$-invariant neighbourhood of $\{0\}\times U\times \{0\}$, where $V\times U\times \R_+$ is equipped with the $\R^\times_+$-action 
\begin{equation}\label{eqn:action bisubmersion}
 \alpha_\lambda (X,x,t)=(\alpha_\lambda(X),x,\lambda^{-1}t),
\end{equation}
\end{enumerate}
such that
 \begin{enumerate}[label=(\roman*)]
\item For every $k\in \{1,\cdots,N\}$, $\natural(V^k)\subseteq \cF^k$,
\item For every $k \in \{1,\cdots,N\}$, $\natural(\bigoplus_{i=1}^k V^i)$ generates $\cF^k$	 on $U$,
\item The map \begin{equation}\label{eqn:new_evaluation_map}
 \ev:\mathbb{U} \to M\times M\times \R_+,\quad (X,x,t)\mapsto (\exp(\natural(\alpha_t(X)))\cdot x,x,t)
\end{equation} is a submersion at every point in $\mathbb{U}\cap (V\times U\times \R_+^\times)$.
\end{enumerate}
We say that $(V,\natural,\mathbb{U},U)$ is a \textit{minimal graded basis} at $p\in U$ if in addition the following is satisfied: 
\begin{enumerate}[label=(\roman*)]
\addtocounter{enumi}{3}
\item $\dim(V^k)=\dim(\cF^k_p/\cF^{k-1}_p)$ for all $k\in \{1,\cdots,N\}$. Equivalently $\dim(V)=\dim(\gr(\cF)_p)$.
\end{enumerate}
\end{definition}
We remark that the $\R^\times_+$-invariance of $\mathbb{U}$ implies that \begin{equation}\label{eqn:U0}
 V\times U\times \{0\}\cup \{0\}\times U\times \R_+\subseteq \mathbb{U}.
\end{equation} To simplify the notation, we use $\ev$ to denote the map \eqref{eqn:new_evaluation_map} when using different graded basis.
\begin{prop}
\label{prop:graded basis exists}
Let $p\in M$. A minimal graded basis at $p$ exists.
\end{prop}
\begin{proof}
By Proposition \ref{lem:generating_family}, we can find $V,\natural,U$ which satisfy (a), (b), (c), (i) and (ii) in Definition \ref{dfn:graded basis}, as well as (iv) at the point $p\in U$. Let $D\subseteq V\times U$ an open neighbourhood of $\{0\}\times U$ such that the map $$V\times U\to M\times U,\quad (\exp(\natural(X))\cdot x,x) $$  is a submersion at every point of $D$. We then let $\mathbb{U}=\{(X,x,t)\in V\times U\times\R_+:(\alpha_t(X),x)\in D\}$.
\end{proof}
\begin{definition}\label{dfn:naturalx map} Let $(V,\natural,\mathbb{U},U)$ be a graded basis. For any $p\in U$, we define the linear map $$\natural_p:V\to \gr(\cF)_p,\quad \natural_p(X)=[\natural(X)]_p\in \cF^i_p/\cF^{i-1}_p,\quad X\in V^i.$$ It is surjective by (ii) of Definition \ref{dfn:graded basis}.
\end{definition}
If $(V,\natural,\mathbb{U},U)$ is a graded basis, then we write
 $$\mathbb{U}_{>0}:= \mathbb{U}\cap (V\times U\times \R_+^\times),\quad \mathbb{U}_{t}:=\mathbb{U}\cap (V\times U\times \{t\})\quad  t\in \R_+^\times.$$ We can also define $\mathbb{U}_0:=\mathbb{U}\cap (V\times U\times \{0\})$ but this is redundant because of \eqref{eqn:U0}. We also define
 \begin{equation}\label{eqn:ev maps tangent group bisubmersions}
 \begin{aligned}
  &\ev_{t}:\mathbb{U}_{t}\to M\times M, &\ev_{t}(X,x,t)=(\exp(\natural(\alpha_t(X)))\cdot x,x),\quad t\in \R_+^\times\\&\ev_{p,0}:V\times \{p\}\times \{0\}\subseteq \mathbb{U}\to \Gr(\cF)_p, &\ev_{p,0}(X,p,0)=\natural_p(X),\quad p\in U.
  \end{aligned}
\end{equation}
The following theorem will be used in Section \ref{sec:Pseudodifferential_Operators} to show the independence of the definition of pseudodifferential operators on the choice of a graded basis. Its proof is given in Appendix \ref{appendixA}. In the appendix a slightly more general version is proved. The one given here is sufficient for our purposes.
\begin{thm}
\label{second_main_tech_thm}
Let $(V,\natural,\mathbb{U},U), (V',\natural',\mathbb{U}',U')$ be two graded bases with $U=U'$ and $p\in U$ and suppose that $(V',\natural',\mathbb{U}',U)$ is minimal at $p$. There exists a smooth map
$$\phi:\dom(\phi)\subseteq\mathbb{U}\to \mathbb{U}'$$ 
defined on an $\R^\times_+$-invariant neighbourhood of $\{0\}\times \{p\}\times \{0\}$ such that \begin{enumerate}
\item The map $\phi$ is an $\R^\times_+$-equivariant submersion.
\item The following diagram commutes $$\begin{tikzcd}\dom(\phi)\arrow[d,"\ev_{|\dom(\phi)}"']\arrow[r,"\phi"]&\mathbb{U}'\arrow[dl,"\ev"]\\M\times M\times \R_+
\end{tikzcd}$$
\item For every $x\in U$ such that $(0,x,0)\in \dom(\phi)$, the following diagram commutes $$\begin{tikzcd}[column sep=huge]V\times \{x\}\times \{0\}\arrow[d,"\ev_{{x,0}}"']\arrow[r,"\phi_{|V\times \{x\}\times \{0\}}"]&V'\times \{x\}\times \{0\}\arrow[dl,"\ev_{x,0}"]\\\Gr(\cF)_x
\end{tikzcd}$$
\end{enumerate}
\end{thm}

\begin{definition}\label{dfn:graded Lie basis}A graded Lie basis is a graded basis $(\mathfrak{g},\natural,\mathbb{U},U)$ such that $\mathfrak{g}$ is equipped with a graded Lie bracket such that if $X\in\mathfrak{g}^i$, $Y\in \mathfrak{g}^j$ and $i+j\leq N$, then \begin{equation}\label{eqn:bracket graded Lie}
 \natural([X,Y])=[\natural(X),\natural(Y)].
\end{equation}
\end{definition}
\begin{remark}
\label{rmk:natural_p}
It follows from the definition of graded Lie basis that the linear maps $\natural_p:\lie{g} \to \gr(\cF)_p$  of Definition \ref{dfn:naturalx map} are Lie algebra homomorphisms, and so induce group homomorphisms $\natural_p:\mathfrak{g}\to\Gr(\cF)_p$, where both spaces are equipped with product by the BCH formula \eqref{eqn:BCHintro}.
\end{remark}

\begin{prop}\label{prop:graded Lie basis exist}Let $(V,\natural,\mathbb{U},U)$ be a graded basis. Then there exists a graded Lie basis $(\mathfrak{g},\natural,\mathbb{U'},U')$ with $U'=U$.
\end{prop}
\begin{proof}
 Let $\mathfrak{g}$ be the free graded nilpotent Lie algebra of step $N$ generated by elements of $V$ (with the same grading as that of $V$). We extend $\natural:V\to \cX_c(M)$ to $\natural:\mathfrak{g}\to \cX_c(M)$ by \eqref{eqn:bracket graded Lie}. We then find $\mathbb{U}'$ such that $(\mathfrak{g},\natural,\mathbb{U}',U)$ is a graded Lie basis	 like we did in the proof of Proposition \ref{prop:graded basis exists}.
\end{proof}
\begin{remark}\label{rem:Global basis}If $M$ is compact, then each $\cF^i$ is finitely generated for each $i$ and thus the proof of  Proposition \ref{prop:graded basis exists} and Proposition \ref{prop:graded Lie basis exist} imply that one can find a graded Lie basis $(\mathfrak{g},\natural,\mathbb{U},M)$ with $U=M$. We call such a basis a global graded Lie basis.
\end{remark}
Theorem \ref{main_tech_thm}, whose proof is given in Appendix \ref{appendixA}, will be used in Section \ref{sec:Pseudodifferential_Operators} to show that our pseudodifferential calculus is closed under composition.
\begin{thm}\label{main_tech_thm}Let $(\mathfrak{g},\natural,\mathbb{U},U)$ be a graded Lie basis. There exists a smooth map $$\phi:\dom(\phi)\subseteq \mathfrak{g}\times \mathfrak{g}\times U\times \R_+\to \mathbb{U}$$ defined on an $\R^\times_+$-invariant neighbourhood of $\{0\}\times \{0\}\times U\times \{0\}$, where the $\R_+^\times$ action on $\mathfrak{g}\times \mathfrak{g}\times U\times \R_+$ is given by $\alpha_\lambda(Y,X,x,t)=(\alpha_\lambda(Y),\alpha_\lambda(X),x,\lambda^{-1}t)$ such that 
\begin{enumerate}
\item $\phi$ is an $\R^\times_+$-equivariant submersion.
\item For all $(Y,X,x,t)\in \dom(\phi)$, $$\ev(\phi(Y,X,x,t))=\Big(\exp(\natural(\alpha_t(Y)))\cdot \Big(\exp(\natural(\alpha_t(X)))\cdot x\Big),x,t\Big).$$
\item The restriction of $\phi$ to the fiber over $0$ is the group law. This means that 
$$\phi(Y,X,x,0)=(\BCH(Y,X),x,0)\quad  \forall x\in U$$ 
where $\BCH(X,Y)$ is the product given by the BCH formula \eqref{eqn:BCHintro}.
\item If $(0,X,x,t)\in \dom(\phi)$, then $\phi(0,X,x,t)=(X,x,t)$. Similarly if $(Y,0,x,t)\in \dom(\phi)$, then $\phi(Y,0,x,t)=(Y,x,t)$
\end{enumerate}
\end{thm}
One should think of $\phi$ as a 'pseudo' group law which interpolates between the flow of vector fields and the group law.

The following theorem, which is straightforward to prove, will be used in Section \ref{subsection:adjoint} to prove that the algebra of pseudodifferential operators is closed under adjoint.
\begin{thm}\label{thm:inverse_bisubmersion}
Let $(V,\natural,\mathbb{U},U)$ be a graded basis. The map $\phi$ defined by
\begin{align*}&\phi:\dom(\phi)\subseteq \mathbb{U}\to \mathbb{U},\quad\dom(\phi)=\{(X,x,t)\in \mathbb{U}:(-X,\exp(\natural(\alpha_t(X)))\cdot x,t)\in \mathbb{U}\},\\
&\phi(X,x,t)=(-X,\exp(\natural(\alpha_t(X)))\cdot x,t). 
\end{align*}
has the following properties:
 \begin{enumerate}
\item $\dom(\phi)$ is $\R^\times_+$-invariant and $\phi$ is an $\R^\times_+$-equivariant open embedding.
\item $\ev\circ \phi=\iota\circ \ev$ on $\dom (\phi)$, where $\iota:M\times M\times \R_+\to M\times M\times \R_+$ is the map $\iota(x,y,t)=(y,x,t)$.
\item $V\times U\times \{0\}\subseteq  \dom(\phi)$ and $\phi(X,x,0)=(-X,x,0)$ for all $X\in V,x\in U$.
\end{enumerate}
\end{thm}
\subsection{Densities}\label{sec:densities}
\begin{enumerate}
\item If $M$ is a manifold, $E\to M$ a vector bundle, then we denote by $\Omega^{\alpha}E$ the bundle of $\alpha$-densities on $E$. We will use $\Omega^{\alpha}(M)$ or simply $\Omega^{\alpha}$ instead of $\Omega^{\alpha}TM$.
\item The space $L^2M$ denotes the completion of the space $C^\infty_c(M,\Omega^{1/2})$ under the Euclidean metric $$\langle f,g\rangle=\int_Mf\bar{g},\quad f,g\in C^\infty_c(M,\Omega^\frac{1}{2})$$ which is well defined because $f\bar{g}\in C^\infty_c(M,\Omega^1)$.
\item If $f\in C^\infty_c(M\times M,\Omega^{1/2})$, then $f$ is a Schwartz kernel, and hence defines a linear map $$f\star\cdot :C^\infty_c(M,\Omega^{1/2})\to C^\infty_c(M,\Omega^{1/2}),\quad f\star g(x)=\int_Mf(x,y)g(y).$$
\item If $\phi:M_1\to M_2$ is a smooth submersion and $E\to M_2$ a vector bundle, then integration along the fibers $\phi$ naturally defines a map $$\phi_*:C^\infty_c(M_1,\phi^*(E)\otimes\Omega^1\ker(d\phi))\to C^\infty_c(M_2,E).$$
\item If $G$ is a Lie group, then we define the $C^*$-algebra $C^*G$ to be the completion of $C^\infty_c(G,\Omega^1)$.
\item We denote by $\Omega^{1/2}_t$ the bundle $\Omega^\frac{1}{2}\ker(dt)$ over $M\times M\times \R_+$ where $t:M\times M\times \R_+\to \R_+$ is the obvious projection. It follows that if $f\in C^\infty(M\times M\times \R_+^\times,\Omega^{1/2}_t)$, then its restriction to $M\times M\times \{t\}$ for $t\in \R_+^\times$ is an element of $C^\infty(M\times M,\Omega^{1/2})$. 
\item Let $(V,\natural,\mathbb{U},U)$ be a graded basis. We have two submersions \begin{align}\label{eqn:bisub_rs}
r,s:\mathbb{\mathbb{U}}\to M\times \R_+,\quad r(X,x,t)=(\exp(\natural(\alpha_t(X)))\cdot x,t),\quad s(X,x,t)=(x,t).
\end{align} We denote by $\Omega^{1/2}_{ r,s}$ the vector bundle $\Omega^{1/2}\ker(ds)\otimes \Omega^{1/2}\ker(dr)$. Since $r,s$ are submersions, we get canonical isomorphisms 
\begin{equation}\label{eqn:qisjdfijqoipjsdofqposidf}
\begin{aligned}
 \Omega^{1/2}\ker(dr)\simeq \Omega^{1/2}\mathbb{U}\otimes r^{*}\Omega^{1/2}M\\
  \Omega^{1/2}\ker(ds)\simeq \Omega^{1/2}\mathbb{U}\otimes s^{*}\Omega^{1/2}M
\end{aligned}
\end{equation}
By combining \eqref{eqn:qisjdfijqoipjsdofqposidf}, we get that
\begin{equation}\label{eqn:qsiodfjpkjsqojdfkoqA}
\Omega^{1/2}_{ r,s}\simeq \Omega^1\ker(ds)\otimes s^*\Omega^{-1/2}M\otimes r^*\Omega^{1/2}M.
\end{equation}
If $(X,x,t)\in \mathbb{U}$, then the diffeomorphism $\exp(\natural(\alpha_t(X)))$ gives an isomorphism between $r^*\Omega^{1/2}M$ and $s^*\Omega^{1/2}M$ at $(X,x,t)$. Hence $r^*\Omega^{1/2}M$ and $s^*\Omega^{1/2}M$ are naturally isomorphic. So \begin{equation}\label{eqn:qsiodfjpkjsqojdfkoq}
\Omega^{1/2}_{ r,s}\simeq \Omega^1\ker(ds)=\Omega^1V.
\end{equation}
This will be used in Section \ref{sec:general_symbols}
\item The maps $r,s$ in \eqref{eqn:bisub_rs} are $\R_+^\times$-equivariant. It follows that $\R_+^\times$ acts on $C^\infty_c(\mathbb{U},\Omega^{1/2}_{ r,s})$ which we denote by $\alpha_{\lambda*}$.
\item Let $t\in \R_+^\times$. The map $\ev_t$ defined in \eqref{eqn:ev maps tangent group bisubmersions} is a submersion. Hence we get a canonical isomorphism \begin{equation}
 \Omega^{1/2}_{ r,s}\simeq \Omega^1\ker(d\ev_t)\otimes  \ev^*_t\Omega^{1/2}(M\times M)
\end{equation} where both sides are restricted to $\mathbb{U}_{>0}$. Hence for any $t\in \R_+^\times$, we have a map \begin{equation}
 \ev_{t*}:C^\infty_c(\mathbb{U},\Omega^{1/2}_{ r,s})\to C^\infty_c(M\times M,\Omega^{1/2})
 \end{equation}
 which first restricts to $\mathbb{U}_t$, then integrates along the fibers of $\ev_t$.
 
\item Let $p\in M$. The map $\ev_{p,0}$ a submersion. Hence we get a canonical isomorphism $$\Omega^{1/2}_{ r,s}\simeq \Omega^1\ker(d\ev_{p,0})\otimes  \ev^*_{p,0}\Omega^{1}\Gr(\cF)_p$$ where both sides are restricted to $V\times \{p\}\times \{0\}$. Hence for any $p\in M$, we have a map \begin{equation}
 \ev_{p,0*}:C^\infty_c(\mathbb{U},\Omega^{1/2}_{ r,s})\to C^\infty_c(\Gr(\cF)_p,\Omega^{1})
 \end{equation}
 which first restricts to $V\times \{p\}\times \{0\}$, then integrates along the fibers of $\ev_{p,0}$. 
\item Let $(V,\natural,\mathbb{U},U),(V',\natural',\mathbb{U}',U')$, $p$ and $\phi$ be as in Theorem \ref{second_main_tech_thm}. It follows from \ref{second_main_tech_thm}.b, that \begin{equation}\label{eqn:int_diff_basis_dens}
  \Omega^{1/2}_{r,s}\simeq \Omega^{1}\ker(d\phi)\otimes \phi^*\Omega_{r,s}^{1/2},
\end{equation} when both sides are restricted to $\dom(\phi)$. Hence we get an integration along the fibers map $$\phi_*:C^\infty_c(\dom(\phi),\Omega^{1/2}_{ r,s})\to C^\infty_c(\mathbb{U}',\Omega^{1/2}_{ r,s}).$$
\item Let $(\mathfrak{g},\natural,\mathbb{U},U)$ and $\phi$ be as in Theorem \ref{main_tech_thm}. For any $f,g\in C^\infty_c(\mathbb{U},\Omega^{1/2}_{ r,s})$, define $h$ by \begin{equation}\label{eqn:h_conv}
 h(Y,X,x,t)=f(Y,\exp(\natural(\alpha_t(x)))\cdot x,t)g(X,x,t)
\end{equation} Up to adding to each of $f$ and $g$ a function in $C^\infty_c(\mathbb{U}_{>0},\Omega^{1/2}_{r,s})$, it is always possible to suppose that $\supp(h)\subseteq \dom(\phi)$.  The map $\phi$ being a submersion together with  Theorem \ref{main_tech_thm}.b implies that \begin{equation}\label{eqn:suiqdhlfh}
 h\in C^\infty_c(\dom(\phi),\Omega^1\ker(d\phi)\otimes \phi^*\Omega^{1/2}_{r,s}).
\end{equation}Hence we can define $$\phi_*(h)\in  C^\infty_c(\mathbb{U},\Omega^{1/2}_{ r,s}).$$
As we will see shortly this will play the role of the convolution of $f$ and $g$. We refer the reader to \cite[Definition 4.2.b]{AS1} for more details on \eqref{eqn:suiqdhlfh}.
\end{enumerate}
\subsection{The \texorpdfstring{$C^*$}{Cstar}-algebra $C^*\aF$}\label{subsec:c0R}
Let $L^\infty\aF$ be the space of all functions $a$ defined on $\R_+^\times \sqcup (M\times \{0\})$ such that \begin{itemize}
\item For all $t\in \R_+^\times$, $a(t)\in K(L^2M)$ the space of compact operators on $L^2M$.
\item For all $p\in M$, $a(p,0)\in C^*\Gr(\cF)_p$. 
\item $\sup_{t\in \R_+^\times}\norm{a(t)}$ and $\sup_{p\in M}\norm{a(p,0)}$ are finite.
\end{itemize} Instead of using the notation $a(t)$ and $a(p,0)$, we will use $a_t$ and $a_{p,0}$ to avoid confusing notation later on. Clearly $L^\infty\aF$ is a $C^*$-algebra with the norm $$\norm{a}=\max\left\{\sup_{t\in \R_+^\times}\norm{a_t},\sup_{p\in M}\norm{a_{p,0}}\right\}.$$

 Let $f\in C^\infty_c(M\times M\times \R_+^\times,\Omega_t^{1/2})$. The function $f$ corresponds to an element of $L^\infty\aF$ still denoted $f$ with $f_{p,0}=0$ for all $p\in M$ and $f_{t}$ is the smoothing operator with kernel $f_{|M\times M\times \{t\}}$. 
 
 Let $(V,\natural,\mathbb{U},U)$ be a graded basis. We define a linear map $$\cQ:C^\infty_c(\mathbb{U},\Omega^{1/2}_{ r,s})\to L^\infty\aF$$ as follows. If $f\in C^\infty_c(\mathbb{U},\Omega^{1/2}_{ r,s})$, $t\in \R_+^\times$, $p\in M$ then $\cQ(f)_t=\ev_{t*}(f)$ and $\cQ(f)_{p,0}=\ev_{p,0*}(f)$.
\begin{lemma}\label{lem:Q_welldefined} The map $\cQ$ is well defined.
\end{lemma} 
\begin{proof}We need to show that $\norm{\cQ(f)}<+\infty$. 
We choose a Euclidean metric on $V$ and a Riemannian metric on $M$. This trivializes all densities used above, where for $p\in M$, the bundle $\Omega^1\Gr(\cF)_p$ is trivialized by the image of the Euclidean metric on $V$ by the map $\natural_p$. Recall that if $f\in C^\infty_c(M\times M)$ is a Schwartz kernel, then $$\norm{f}_{K(L^2M)}\leq \sup_{x\in M}\max\left\{\int_{M}|f(x,y)|dy,\int_{M}|f(y,x)|dy\right\}.$$ Similarly $$\norm{f}_{C^*\Gr(\cF)_p}\leq  \norm{f}_{L^1\Gr(\cF)_p},\quad f\in C^\infty_c(\Gr(\cF)_p).$$ Now let $f\in C^\infty_c(\mathbb{U})$. It follows that $$\sup_{p\in M}\norm{\cQ(f)_{p,0}}_{C^*\Gr(\cF)_p}\leq\sup_{p\in M}\norm{\cQ(f)_{p,0}}_{L^1\Gr(\cF)_p} \leq\sup_{p\in M} \int_{V}|f(v,p,0)|dv<+\infty.$$
Similarly $$ \sup_{x\in M}\int_{M}|\cQ(f)_t(y,x)|dy\leq \sup_{x\in M}\int_V|f(v,x,t)|dv<+\infty.$$ For $\sup_{x\in M}\int_{M}|\cQ(f)_t(x,y)|dy$, we proceed differently. Let $\phi$ be as in Theorem \ref{thm:inverse_bisubmersion}. Then, using  a partition of unity we can write $f=f_1+f_2$ with $f_1\in C^\infty_c(\dom(\phi))$ and $f_2\in C^\infty_c(\mathbb{U}\cap (V\times U\times \R_+^\times))$. By Theorem \ref{thm:inverse_bisubmersion}.b, it follows that $\cQ(f_1)^*=\cQ(\phi_{*}(f_1))$. Hence $\sup_{x\in M}\int_{M}|\cQ(f_1)_t(x,y)|dy<+\infty$. Since $\cQ(f_2)\in C^\infty_c(M\times M\times \R_+^\times)$, the lemma follows.
\end{proof}
 
\begin{definition} Let $C^\infty_c(\aF)\subseteq L^\infty\aF$ be the linear span of $C^\infty_c(M\times M\times \R_+^\times,\Omega_t^{1/2})$ together with $\cQ(C^\infty_c(\mathbb{U},\Omega_{r,s}^{1/2}))$ for all graded bases $(V,\natural,\mathbb{U},U)$. Elements of $C^\infty_c(M\times M\times \R_+^\times,\Omega_t^{1/2})$ and $\cQ(C^\infty_c(\mathbb{U},\Omega_{r,s}^{1/2}))$ will be called elements of first and second type respectively.\end{definition}
In the proof of the following proposition, it will be useful to remark that if $f\in C^\infty_c(\mathbb{U}_{>0},\Omega_{r,s}^{1/2})$, then $\cQ(f)$ is an element of first type.
\begin{prop}\label{prop:subalgebra} The space $C^\infty_c(\aF)$ is a $*$-subalgebra of $L^\infty\aF$.
\end{prop}
\begin{proof}
In Lemma \ref{lem:Q_welldefined}, we proved that $C^\infty_c(\aF)$ is closed under taking adjoint. We now prove that $C^\infty_c(\aF)$ is closed under product. The product of two elements of first type or an element of first type and another of second type is easily seen to be of first type. So we need to consider two elements of second type. Let $(V,\natural,\mathbb{U},U)$ and $(V',\natural',\mathbb{U}',U')$ be graded basis, $f\in\cQ(C^\infty_c(\mathbb{U}),\Omega_{r,s}^{1/2}) $, $g\in \cQ(C^\infty_c(\mathbb{U}'),\Omega_{r,s}^{1/2})$. Notice that if $U''\subseteq U$ is an open subset, then $(V,\natural,\mathbb{U}\cap (V\times U''\times \R_+),U'')$ is still a graded basis. By a partition of unity argument, we can reduce to two cases either $U\cap U'=\emptyset$ or $U=U'$. If $U\cap U'=\emptyset$, then $\cQ(f)\cQ(g)$ is easily seen to be an element of first type. We can thus suppose $U=U'$. By Proposition \ref{prop:graded Lie basis exist}, let $(\mathfrak{g},\natural,\mathbb{U}'',U)$ be a graded Lie basis.
\begin{lemma}\label{lem:qsjidfpiujqsipdf} There exists $\tilde{f},\tilde{g}\in C^\infty_c(\mathbb{U}'',\Omega^{1/2}_{r,s})$ such that $\cQ(\tilde{f})-\cQ(f)$ and $\cQ(\tilde{g})-\cQ(g)$ are elements of first type.
\end{lemma}
\begin{proof}
By symmetry, it is enough to construct $\tilde{f}$. Let $p\in U$, $(V_p,\natural,\mathbb{U}_p,U_p)$ be a minimal graded basis at $p$. We can further suppose that $U_p\subseteq U$. Let $\phi_p:\dom(\phi)\subseteq\mathbb{U}\to \mathbb{U}_p$ and $\psi_p:\dom(\psi)\subseteq \mathbb{U}''\to \mathbb{U}_p$ obtained from Theorem \ref{second_main_tech_thm}. The maps $\phi_p$ and $\psi_p$ are submersions and $\phi_p(0,p,0)=\psi_p(0,p,0)=(0,p,0)$ (this follows from Theorem \ref{second_main_tech_thm}.a and b). By reducing the domain of $\phi_p$ if necessary, we can suppose that $\Im(\phi_p)\subseteq \Im(\psi_p)$. Using a partition of unity on the cover $\mathbb{U}=\mathbb{U}_{>0}\cup \bigcup_{p\in U}\dom(\phi_p)$, up to adding an element of first order, we can suppose that $\supp(f)\subseteq \dom(\phi_p)$ for some $p\in U$. By Theorem \ref{second_main_tech_thm}.b and c, it follows that $\cQ(\phi_{p*}(f))=\cQ(f)$. Furthermore since $\Im(\phi)\subseteq \Im(\psi)$ and $\psi_p$ is a submersion, we can find $\tilde{f}\in  C^\infty_c(\mathbb{U}'',\Omega^{1/2}_{r,s})$ such that $\supp(\tilde{f})\subseteq \dom(\psi_p)$ and $\psi_{p*}(\tilde{f})=\phi_{p*}(f)$. Hence $\cQ(\tilde{f})=\cQ(f)$.
\end{proof}
Let $\phi:\dom(\phi)\subseteq \mathfrak{g}\times \mathfrak{g}\times U\times \R_+\to \mathbb{U}''$ given by Theorem \ref{main_tech_thm}. By using a partition of unity once more which amounts to adding an element of first type to $\cQ(\tilde{f})$ and $\cQ(\tilde{g})$, we can suppose that $\dom(\tilde{f})$ and $\dom(\tilde{g})$ are small enough so that $h$ defined by \eqref{eqn:h_conv} has $\supp(h)\subseteq \dom(\phi)$. The proof is now complete because $\cQ(h)=\cQ(\tilde{f})\cQ(\tilde{g})$ by Theorem \ref{main_tech_thm}.b and c.
\end{proof}
\begin{definition}We define  $C^*_{ }\aF$ to be the closure of $C^\infty_c(\aF)$ in $L^\infty\aF$.
\end{definition}
\begin{remark}\label{rem:Action_on_Caf} The group $\R_+^\times$ acts on $L^\infty\aF$ by the formula $$\alpha_{\lambda*}(a)_t=a_{\lambda t},\quad \alpha_{\lambda*}(a)_{p,0}=\alpha_{\lambda *}(a_{p,0}).$$ Since for any $f\in C^\infty_c(\mathbb{U}'',\Omega^{1/2}_{r,s})$, $\alpha_{\lambda*}(\cQ(f))=\cQ(\alpha_{\lambda*}(f))$, it follows that $\alpha_{\lambda*}$ leaves $C^\infty_c(\aF)$ and hence $C^*_{ }\aF$ invariant.
\end{remark}
\subsection{Connection with the Helffer-Nourrigat cone}\label{sec:connection_Helffer} We define an action of $C_0(\R_+)$ on $C^*\aF$ as follows. If $g\in C_0(\R_+)$ $a\in C^*\aF$, then $ga\in C^*\aF$ is the element $(ga)_t=g(t)a_t$ and $(ga)_{p,0}=g(0)a_{p,0}$. This action makes $C^*_{ }\aF$ a $C_0(\R_+)$-$C^*$-algebra as defined in \cite{KasparovInvent}. We denote by $C^*_{ }\Gr(\cF)$ the fiber at $0$ of $C^*_{ }\aF$. It lies in a short exact sequence \begin{equation}\label{eqn:exact_seq_C*-alg}
   0\to K(L^2M)\otimes C_0(\R^\times_+)\to C^*_{ }\aF\to C^*_{ }\Gr(\cF)\to 0.
\end{equation}
The $C^*$-algebra $C^*_{ }\Gr(\cF)$ is a $C_0(M)$-$C^*$-algebra. Its fiber at $p\in M$ is equal to $C^*_{ }\Gr(\cF)_p.$ Therefore, as a set, the spectrum of $C^*_{ }\Gr(\cF)$ is equal to $$\widehat{C^*_{ }\Gr(\cF)}=\bigsqcup_{p\in M}\widehat{\Gr(\cF)}_p.$$ It is thus a quotient of $\bigsqcup_{p\in M}\gr(\cF)^*_p=\gr(\cF)^*$ by the co-adjoint actions. We equip $\gr(\cF)^*$ with the subspace topology from the inclusion $\gr(\cF)^*\times \{0\}\subseteq\aF^*$ in \eqref{eqn:aF*} where the latter is equipped with the topology from Section \ref{sec:generalised_distr}.g.
\begin{prop}\label{prop:TopJacbson}The Fell topology on the spectrum of $C^*_{ }\Gr(\cF)$ coincides with the quotient topology from $\gr(\cF)^*$.
\end{prop}
\begin{proof}
The statement is local in $M$. Let $(\mathfrak{g},\natural,\mathbb{U},U)$ be a graded Lie basis, $p\in U$. The map $\cQ:C^\infty_c(\mathbb{U},\Omega_{r,s}^{1/2})\to C^\infty_c(\aF)$ together with Theorem \ref{main_tech_thm} (see the proof Proposition \ref{prop:subalgebra}) gives a $C_0(U)$-$C^*$-homomorphism $C^*_{ }\mathfrak{g}\otimes C_0(U)\to C^*\Gr(\cF)_{|U}$. Furthermore, the restriction of this map to the fibers at $x\in U$ is equal to $C^*\natural_x$ where $\natural_x:\mathfrak{g}\to \Gr(\cF)_x$ is the group homomorphism from Remark \ref{rmk:natural_p}. It follows that the spectrum of $C^*\Gr(\cF)_{|U}$ is homeomorphic to a closed subset of the spectrum of $C^*_{ }\mathfrak{g}\otimes C_0(U)$. On the other hand, the dual maps $\natural_x^*:\gr(\cF)_x^*\to \mathfrak{g}^*$ glue together to give a closed embedding $\gr(\cF)^*_{|U}\to \mathfrak{g}^*\times U$. The result follows from Brown's theorem \cite{BrownArticleTopOrbitMethod} applied to the group $\mathfrak{g}$.
\end{proof}
\paragraph{Limit at $0$.} By \cite[Proposition C.10.a on Page 357]{MR2288954}, if $a\in C^*_{ }\aF$, then $$\limsup_{t\to 0^+}\norm{a_t}_{K(L^2M)}\leq \sup_{p\in M}\norm{a_{p,0}}_{C^*\Gr(\cF)_p}.$$ In general the inequality is strict. 
One can resolve this issue as follows. Let 
$$J=\{a\in C^*\aF:a_t=0\,\ \forall t\in \R^\times_+\}.$$ 
The set $J$ is a closed $*$-ideal in $C^*_{ }\aF$. It is concentrated in the $0$-fiber. Hence it maps injectively by the map $C^*_{ }\aF\to C^*\Gr(\cF)$ to a closed $*$-ideal in $C^*\Gr(\cF)$, that will be denoted by $J_0$. 

\begin{definition}We denote by \begin{enumerate}
\item $C^*_z\aF$  the quotient $C^*_{ }\aF/J$
\item $C^*\cT\cF$ the quotient $C^*\Gr(\cF)/J_0$, which is the fiber of $C^*_z\aF$ at $0$.
\end{enumerate}
\end{definition}

Hence one has the exact sequence \begin{equation}\label{eqn:qjjsdfljqsdfjll}
    0\to K(L^2M)\otimes C_0(\R^\times_+)\to C^*_{z }\aF\to C^*\cT\cF\to 0.
\end{equation}
We introduce some ad hoc terminology that will be useful for discussing our fields of $C^*$-algebras.

\begin{prop}[{\cite[Proposition 3.1]{NewCalgebra}}]\label{prop:I limsup}Let $A$ be a $C_0(\R_+)$-$C^*$-algebra. The following are equivalent:
\begin{enumerate}
\item For every $a\in A$, if $a_t=0$ for every $t\in \R^\times_+$, then $a=0$.
\item For every $a\in A$, $\limsup_{t\to 0^+}\norm{a_t}=\norm{a_0}$.
\end{enumerate}
Here $a_t$ denotes the fiber of $a$ at $t\in \R_+$. If these conditions are satisfied, then we say that $A$ is half-continuous at $0$.\end{prop}
By construction, $C^*_z\aF$ is half-continuous at $0$. In the following theorem, we consider the Helffer-Nourrigat cone $\cT^*_p\cF$ as a subset of $\widehat{\Gr(\cF)}_p$ by the orbit method.
\begin{thm}[{\cite[Theorem 3.7 and Examples 3.8]{NewCalgebra}}]\label{thm:ana=top}The Helffer-Nourrigat cone $\cT^*\cF=\bigsqcup_{p\in M}\cT^*_p\cF\subseteq \widehat{C^*\Gr(\cF)}$ is equal to the support of $J_0$, \emph{i.e.}, $\pi\in \cT^*\cF$ if and only if $J_0\subseteq \ker(\pi)$.
\end{thm} 
Theorem \ref{thm:ana=top} together with Proposition \ref{prop:I limsup} imply that if $a\in C^*\aF$, then \begin{equation}\label{eqn:mainineqca}
 \limsup_{t\to 0^+}\norm{a_t}_{K(L^2M)}=\sup_{p\in M}\sup_{\pi\in\cT^*_p\cF}\norm{\pi(a_{p,0})}_{L^2\pi}.
\end{equation}
\begin{remark} Although we don't need this, we suspect that we can replace $\limsup_{t\to 0^+}$ with $\lim_{t\to 0^+}$ in \eqref{eqn:mainineqca}. The proof of Theorem \ref{thm:ana=top} shows that this is possible if and only if for any sequence $t_n\in \R^\times_+$ such that $t_n\to 0$, one has $$\cT^*\cF=\{\xi \in\gr(\cF)^*:(\xi,0)\in \overline{\bigsqcup_{n\in \N}T^*M\times \{t_n\}}\subseteq \aF^*\}.$$
We don't have an example where this fails.
\end{remark}

\subsection{Proof of Theorem \ref{thm intro char set} and Theorem \ref{thm:sybmol_welldefined}}\label{sec:diff_multi}
In this section we will prove Theorem \ref{thm:sybmol_welldefined} and thus in particular Theorem \ref{thm intro char set}.
\begin{proof}[Proof of Theorem \ref{thm:sybmol_welldefined}]
In this proof, if $X\in \cX(M)$, then $L_X$ denotes the Lie derivative which acts on $C^\infty(M,\Omega^\alpha)$ for $\alpha\in \C$ using the flow of $X$. Let $X\in \cF^i$, $p\in M$, $\tilde{X}_p$ the right invariant vector field on $\Gr(\cF)_p$ associated to $[X]_{i,p}\in \gr(\cF)_p$. Since $\tilde{X}_p$ is right invariant, it satisfies $$L_{\tilde{X}_p}(f\star g)=L_{\tilde{X}_p}(f)\star g,\quad f,g\in C^\infty_c(\Gr(\cF)_p,\Omega^1).$$ Hence $L_{\tilde{X}_p}$ defines an unbounded multiplier of $C^*\Gr(\cF)_p$ with domain $C^\infty_c(\Gr(\cF)_p,\Omega^1)$. Let $\theta_i(X)$ be the unbounded multiplier of $L^\infty\aF$ with domain $C^\infty_c(\aF)$ defined by \begin{equation}\label{eqn:theta}
 (\theta_i(X)a)_t=t^iL_{X} \circ a_t,\quad (\theta_i(X)a)_{p,0}=L_{\tilde{X}_p}(a_{p,0}),\quad a\in C^\infty_c(\aF),
\end{equation} where $L_{X} \circ a_t$ is the composition $$L^2M\xrightarrow{a_t} C^\infty_c(M,\Omega^{1/2})\xrightarrow{L_X}C^\infty_c(M,\Omega^{1/2})\subseteq L^2M.$$ \begin{lemma}\label{lem:proof_sqjdfmokj} $\theta_i(X)(C^\infty_c(\aF))\subseteq C^\infty_c(\aF)$.\end{lemma}\begin{proof}
For elements of first type this is obvious. For elements of second type, by the discussion in the proof of Proposition \ref{prop:subalgebra}, it is enough to consider elements of the form $\cQ(f)$ where $f\in C^\infty_c(\mathbb{U},\Omega^{1/2}_{r,s})$ and $(\mathfrak{g},\natural,\mathbb{U},U)$ is a graded Lie basis. We can further suppose that we are given an element $\bar{X}\in \mathfrak{g}^i$ such that $\natural(\bar{X})=X$.
\begin{lemma}There exists a vector field $Y$ defined on $\mathbb{U}$ such that \begin{enumerate}
\item If $s: \mathbb{U}\to M\times \R_+$ is the map $s(v,x,t)=(x,t)$, then $ds(Y)=0$.
\item If $\pi:\mathbb{U}\to M$ is the map $\pi(v,x,t)=\exp(\natural(\alpha_t(v)))\cdot x$, then $d\pi(Y)=t^iX\circ \pi$.
\item for every $p\in U$, the restriction of $Y$ to $\mathfrak{g}\times \{p\}\times \{0\}$ is the right invariant vector field associated to $\bar{X}\in \mathfrak{g}$.
\end{enumerate}
\end{lemma}
\begin{proof}
Let $\phi$ as in Theorem \ref{main_tech_thm}. We define $Y$ on $\{(v,x,t)\in \mathbb{U}:(0,v,x,t)\in \dom(\phi)\}$ by $$Y(v,x,t)=\frac{d}{d\tau}\Big\rvert_{\tau=0}\phi(\tau \bar{X},v,x,t).$$ By Theorem \ref{main_tech_thm}.b, $s(\phi(\tau \bar{X},v,x,t))=(x,t)$ and $$ \pi(\phi(\tau \bar{X},v,x,t))=\exp(\natural(\alpha_t(\tau \bar{X})))\cdot \pi(v,x,t)=\exp(\tau t^iX)\cdot \pi(v,x,t).$$ It follows that $Y$ satisfies a and b. By Theorem \ref{main_tech_thm}.c, it follows that $$Y(v,x,0)=\frac{d}{d\tau}\Big\rvert_{\tau=0}(\BCH(\tau \bar{X},v),x,0).$$ Hence $Y$ satisfies c. We can cover $\mathbb{U}$ by $\mathbb{U}_{>0}$ and $\{(v,x,t)\in \mathbb{U}:(0,v,x,t)\in \dom(\phi)\}$. On $\mathbb{U}_{>0}$, by Condition (iii) of Definition \ref{dfn:graded basis}, we can easily construct $Y$ satisfying a and b. The Lemma follows by a partition of unity argument.
\end{proof}
The proof of Lemma \ref{lem:proof_sqjdfmokj} is complete because $\theta_i(X)(\cQ(f))=\cQ(L_Y(f))$.
\end{proof}
If $n\in \N$, we define an unbounded multiplier $T_{n}$ of $L^\infty\aF$ with domain $C^\infty_c(\aF)$ as follows$$T_n(a)_t=t^na_t,\quad T_n(a)_{p,0}=0,\quad a\in C^\infty_c(\aF).$$ We also define $T_0(a)=a$.
Now let $D\in \Diff^k_\cF(M,\Omega^{1/2}M)$ and further suppose that $D$ is compactly supported. By the definition of $\Diff^k_\cF(M,\Omega^{1/2}M)$, we can write $D$ as a sum of monomials of the form $\alpha L_{Y_1}\cdots L_{Y_{l}}$ with $\alpha\in \C$, $Y_{i}\in\cF^{a_{i}}$ for all $i$ and $\sum_{i=1}^l  a_{i} \leq k$. This expression is slightly different from the one we used in Section \ref{sec:differential operators}, where we used $\nabla_{Y}$ instead of $L_Y$. This makes no difference in the definition of the principal symbol because $\nabla_Y-L_Y\in C^\infty_c(M,\C)$.  We now define an unbounded multiplier of $L^\infty\aF$ denoted by $\Theta(D)$ by taking the sum $$\alpha T_{k-\sum_{i=1}^l a_i } \theta_{a_1}(Y_{1})\cdots \theta_{a_l}(Y_{l})$$ for each monomial in the decomposition of $D$. By Lemma \ref{lem:proof_sqjdfmokj}, $\Theta(D)$ is well defined on $C^\infty_c(\aF)$ and $$\Theta(D)(C^\infty_c(\aF))\subseteq C^\infty_c(\aF).$$ Furthermore if $a\in C^\infty_c(\aF)$, then $(\Theta(D)a)_t=t^kD\circ a_t$ for $t>0$ and for $p>0$, $(\Theta(D)a)_{p,0}$ is a sum of $\alpha L_{\widetilde{Y}_{1p}}\cdots L_{\widetilde{Y}_{lp}}(a_{p,0})$, and one only sums over monomials such that $\sum_{i=1}^la_i=k$. Hence if $\pi\in \widehat{\Gr(\cF)_p}$, then $$\pi((\Theta(D)a)_{p,0})=\sigma^k(D,p,\pi)(\pi(a_{p,0})).$$ It is true that $\Theta(D)$ may depend on the presentation of $D$ as a sum of monomials. But its action on the nonzero fibers doesn't depend on the presentation (and is equal to $t^kD$). Hence by \eqref{eqn:mainineqca}, $\sigma(D,p,\pi)(\pi(a_{p,0}))$ for $a\in C^*\aF$ and $\pi \in \cT^*_p\cF$ also doesn't depend on the presentation of $D$ as sum of monomials. This finishes the proof of Theorem \ref{thm:sybmol_welldefined} for compactly supported differential operators. For general operators, it is clear that $\sigma^k(D,p,\pi)$ only depends on $D$ in a neighbourhood of $p$, so  Theorem \ref{thm:sybmol_welldefined}  follows for $E=\Omega^{1/2}M$. For other vector bundles, one can embed them inside $ \C^n\otimes \Omega^{1/2}M$ for $n$ big enough.
\end{proof}


\section{Pseudodifferential Operators}\label{sec:Pseudodifferential_Operators}
In this section we define an algebra $\Psi(\cF^\bullet)$ of pseudodifferential operators. We show that $\Psi(\cF^\bullet)$ admits properties very similar to the properties of the algebra of classical pseudodifferential operators \cite[Chapter 18]{HormanderBooks3}. All the results of this section easily extend to pseudodifferential operators with vector bundle coefficients. We will omit them to simplify the exposition. Throughout this section, we will treat the Helffer-Nourrigat cone as a set of representations of the osculating groups. This section is organized as follows.
 \begin{itemize}
 \item In Section \ref{sec:general_symbols}, given a graded basis $(V,\natural, \mathbb{U},U)$, we define a vector space of distributions on $\mathbb{U}$. This space will be defined in two equivalent ways. The first is by a standard quantization of symbols on $V$. The other is by invoking properties of the $\R_+^\times$ action on $\mathbb{U}$.
 \item In Section \ref{subsection:dfn of pseudo-diff}, we define our pseudodifferential operators.
 \item In Section \ref{subsection:adjoint}, we prove that $\Psi(\cF^\bullet)$ is closed under composition and adjoint.
 \item In Section \ref{sec:dis_groups}, we recall a few results from \cite{ChrGelGloPol} which will let us construct parametricies for some elements in $\Psi(\cF^\bullet)$.
 \item In Section \ref{sec:Sobolev}, we extend the definition of the Sobolev spaces in the introduction to $s\in \R$.
 \item In Section \ref{sec:princip symb}, we extend the definition of our principal symbol to $\Psi(\cF^\bullet)$.
 \item In Section \ref{sec:Parametrised}, we prove that the operators whose principal symbol vanishes are compact.
 \item In Section \ref{sec:para order 0}, we prove Theorems \ref{mainthmintro2} and \ref{mainthmintro3} when $M$ is compact.
 \item In Section \ref{sec:noncompact}, we prove Theorems \ref{mainthmintro2} and \ref{mainthmintro3} when $M$ is arbitrary.
 \end{itemize}
\subsection{Oscillatory integrals}\label{sec:general_symbols}
If $E\to M$ a vector bundle, then we denote by $\cD'(M,E)$ the topological dual of $C^\infty_c(M,E^*\otimes \Omega^1M)$. In particular $C^\infty_c(M,E)\subseteq \cD'(M,E)$. We use $\cD'(M,\Omega^\alpha):=\cD'(M,\Omega^\alpha M)$.
\begin{definition}[{\cite[Section 1.2]{AS2}, see also \cite{LesManVas}}]
 Let $ \phi:M_1\to M_2$ be a smooth submersion map, $E\to M_1$ a vector bundle. We say a distribution $u\in \cD'(M_1,E)$ is transverse to $\phi$ if $\phi_*(fu)\in C^\infty_c(M_2)$ for any $f\in C^\infty_c(M_1,E^*\otimes \Omega^1\ker(d\phi))$.\end{definition}
If $u$ is transverse to $\phi$, then one can restrict $u$ to $\phi^{-1}(x)$ for any $x\in M_2$ and obtain $u_x\in \cD'(\phi^{-1}(x),E_{|\phi^{-1}(x)})$. The following example illustrates the use of transverse distributions for pseudodifferential operators. \begin{ex}\label{ex:schwartz_trans}Let $p:M_1\times M_2\to M_2$ be the projection map. By the Schwartz kernel theorem \cite[Theorem 5.2.1]{HormanderBook1}, an element $u\in \cD'(M_1\times M_2,\Omega^{1/2})$ corresponds to a continuous map between $C^\infty_c(M_2,\Omega^{1/2})\to \cD'(M_1,\Omega^{1/2})$. It is easily seen that $u$ is transverse to $p$ if and only if it maps $C^\infty_c(M_2,\Omega^{1/2})$ to $C^\infty(M_1,\Omega^{1/2})$. 
\end{ex}
 Let $(V,\natural,\mathbb{U},U)$ be a graded basis. Recall that $\R_+^\times$ acts on $\mathbb{U}$ by \eqref{eqn:action bisubmersion}. If $u\in \cD'(\mathbb{U},\Omega^{1/2}_{r,s})$, then we define 
$$\langle\alpha_{\lambda*} u,f\rangle=\lambda^{-1} \langle u,\alpha_{\lambda^{-1}*}f\rangle.$$\begin{definition}\label{dfn:Eprimek} 
Let $k\in \C$. We define $\mathcal{E}^{\prime k}(\mathbb{U})$ to be the subspace of $u\in \cD'(\mathbb{U},\Omega^{1/2}_{r,s})$ such that \begin{enumerate}
\item $u$ is transverse to the map $s:\mathbb{U}\to M\times \R_+$ given by $s(X,x,t)=(x,t)$.
\item For any $\lambda\in \R_+^\times$,
\begin{equation}\label{eqn: equiv cond def pseudo diff}
\alpha_{\lambda*} u-\lambda^k u\in C^\infty_c(\mathbb{U},\Omega^{1/2}_{r,s}). 
\end{equation} 
\item The projection $\supp(u)\to \R_+$ is proper, where the projection comes from the inclusion $\supp(u)\subseteq \mathbb{U}\subseteq V\times U\times \R_+$. 
\end{enumerate}
\end{definition}
The following proposition is a bundle version of a proposition due to Taylor \cite[Proposition 2.2]{TaylorBook}. It shows that any $u\in \cE^{\prime k}(\mathbb{U})$ is the sum of an oscillatory integral and a Schwartz function. To state it, we need the following notation. If $V$ is a vector space and $f\in C^\infty(V,\Omega^1V)$, then we say that $f$ is Schwartz if it is Schwartz after trivializing $\Omega^1V$ by any Euclidean structure on $V$. In the next proposition, we treat $u\in \mathcal{E}^{\prime k}(\mathbb{U})$ as an element of $\cD'(\mathbb{U},\Omega^1V)$ by \eqref{eqn:qsiodfjpkjsqojdfkoq}.
 \begin{prop}\label{prop:const pseudo from symb}Let $k\in \C$. If $u\in \mathcal{E}^{\prime k}(\mathbb{U})$, then there exists a unique smooth function $A\in C^\infty((V^*\times U\times \R_+)\setminus (\{0\}\times U\times \{0\}))$ called the full symbol of $u$ such that
 \begin{enumerate}
   \item For all $\lambda\in\R_+^\times$ 
   \begin{equation}\label{eqn:A_homog_symbol}
   A(\hat{\alpha}_\lambda(\xi),x,t\lambda)= \lambda^k A(\xi,x,t),
   \end{equation}
where $\hat{\alpha}_\lambda(\xi)$ is defined in \eqref{eqn:act dual space}.
\item There exists $K\subseteq U$ compact such that $\supp(A)\subseteq V^*\times K\times \R_+$.
\item If $\chi\in C^\infty_c(V^*\times \R_+)$ equal to $1$ in a neighbourhood of $(0,0)$, then
\begin{equation}\label{eqn:homg symbols second eqn}
 f(X,x,t)=u(X,x,t)-\int_{V^*}e^{i\langle \xi,X\rangle}(1-\chi)(\xi,t)A(\xi,x,t),
\end{equation}
then $f(X,x,t)\in C^\infty(V\times U\times \R_+,\Omega^1V)$, and $\supp(f)\subseteq V\times K\times [0,a]$ for some $K\subseteq U$ compact and $a\in \R_+$. Furthermore $f$ and all its derivatives in $x$ and $t$ are Schwartz in $X$ uniformly in $x$ and $t$.
\end{enumerate}
Conversely if $A\in C^\infty((V^*\times U\times \R_+)\setminus (\{0\}\times U\times \{0\}))$ satisfies $a$ and $b$, then there exists $u\in \mathcal{E}^{\prime k}(\mathbb{U})$ such that $c$ is satisfied, i.e., whose full symbol is $A$.
\end{prop}
\begin{proof}
We choose a Euclidean metric on $V$ and a Riemmanian metric on $M$. We have thus trivialized all densities that appear above. Since $u$ is transverse to $s:\mathbb{U}\to M\times \R_+$, we can restrict $u$ to $\mathbb{U}\cap (V\times \{x\}\times \{t\})$ for any $(x,t)\in U\times \R_+$. We denote the restriction by $u_{x,t}$. By Condition c of Definition \ref{dfn:Eprimek}, $u_{x,t}$ is compactly supported. Let $v$ be the smooth function on $V^*\times U\times \R_+$ given by $v(\xi,x,t)=\hat{u}_{x,t}(\xi)$ where $\hat{u}_{x,t}$ is the Euclidean Fourier transform of $u_{x,t}$. Condition b of Definition \ref{dfn:Eprimek} implies that for every $\lambda\in \R_+$, there exists $h_\lambda\in C^\infty_c(V\times U\times \R_+)$ such that $$v(\hat{\alpha}_\lambda(\xi),x,t\lambda)-\lambda^kv(\xi,x,t)=h_{\lambda}(\xi,x,t)$$ and $h_\lambda$ and all its derivatives in $x,t$ are Schwartz in $\xi$ uniformly in $x,t$. By induction one has  \begin{equation}\label{eqn:qnjilsdfl}
 2^{-k(l+1)}v(\hat{\alpha}_{2^{l+1}}(\xi),x,2^{l+1}t)=v(\xi,x,t)+\sum_{n=0}^{l}2^{-kn}h_{2}(\hat{\alpha}_{2^n}(\xi),x,2^nt),\quad \forall l\in \N.
\end{equation} We define \begin{equation}\label{eqn:qhsdjliqsdfjihlJISD}
 A(\xi,x,t)=v(\xi,x,t)+\sum_{n=0}^{+\infty}2^{-kn}h_{2}(\hat{\alpha}_{2^n}(\xi),x,2^nt).
\end{equation} We now check that $A$ has the required properties \begin{enumerate}
\item if $(\xi,t)\neq (0,0)$, then the series is absolutely convergent because $h_2\in \cS(V^*\times U\times \R_+)$. Same for all derivatives, so it follows that $A\in C^\infty((V^*\times U\times \R_+)\setminus (\{0\}\times U\times \{0\}))$. Equation \eqref{eqn:A_homog_symbol} with $\lambda=2$ follows trivially from \eqref{eqn:qhsdjliqsdfjihlJISD}. If one defines $B$ like $A$ but replacing $2$ by $2^{1/l}$, then by \eqref{eqn:qnjilsdfl} (again replacing $2$ by $2^{1/l}$), $A=B$. Therefore $A$ satisfies \eqref{eqn:A_homog_symbol} for $\lambda=2^{1/l}$ for any $l\in \N$. By continuity, \eqref{eqn:A_homog_symbol} follows for all $\lambda\in\R_+^\times$.
\item Since $h_2$ is compactly supported, it follows that $\supp(h_2)\subseteq V^*\times K\times [0,a]$ for some $K\subseteq U$ compact and $a\in \R_+$. By Condition c of Definition \ref{dfn:Eprimek}, there exists $K'\subseteq U$ compact such that $\supp(v)\cap (V^*\times U\times [0,a])\subseteq (V^*\times K'\times [0,a])$. Let $K''=K\cup K'$. By \eqref{eqn:qnjilsdfl}, one deduces that $\supp(v)\subseteq V^*\times K''\times \R_+$. Hence $\supp(A)\subseteq V^*\times K''\times \R_+$.
\item it suffices to show that $$g(\xi,x,t)=\chi(\xi,t) v(\xi,x,t)+(1-\chi(\xi,t))\sum_{n=0}^{+\infty}2^{-kn}h_{2}(\hat{\alpha}_{2^n}(\xi),x,2^nt)\in \cS(V^*\times U\times \R_+).$$  Since $\chi\in C^\infty_c(V^*\times \R_+)$ there exists $b\in \R_+$ such that $\supp(g)\subseteq V^*\times K''\times [0,b]$. Decay at infinity of $g$ easily follows from that of $h_2$.
\end{enumerate}
Uniqueness of $A$ easily follows from \eqref{eqn:homg symbols second eqn}. Now let $A\in C^\infty((V^*\times U\times \R_+)\setminus (\{0\}\times U\times \{0\}))$ satisfing $a$ and $b$. Let $u=\int_{V^*}e^{i\langle \xi,X\rangle}(1-\chi)(\xi,t)A(\xi,x,t)d\xi$. It is clear that $u$ is transverse to $s$. By hypothesis on $A$, $\ssing(u)\subseteq \{0\}\times K \times \R_+$ and \begin{equation}\label{eqn:proof of symb const}
 \alpha_{\lambda*} u-\lambda^ku\in C^\infty(V\times U\times \R_+).
\end{equation}
Let $K'\subseteq U$ be a compact neighbourhood of $K$. Since $\mathbb{U}$ is $\R_+^\times$-invariant and $\{0\}\times K'\times \{0\}\subseteq \mathbb{U}$ there exists $\epsilon>0$ such that $$\{(X,x,t)\in V\times K'\times \R_+:\norm{\alpha_t(X)}\leq 2\epsilon\}\subseteq \mathbb{U},$$ where $\norm{\cdot}$ is the norm associated to the Euclidean structure on $V$. Let $g\in C^\infty(V\times U \times \R_+)$ be any smooth function with the following properties:
\begin{itemize}
 \item $g=1$ on a neighbourhood of $\{0\}\times K \times \R_+$,
 \item $\supp(g)\subseteq \{(X,x,t)\in V\times K'\times \R_+:\norm{\alpha_t(X)}\leq \epsilon,\norm{X}\leq 1\}$
 \item $g(X,x,t)=g(\alpha_t(X),x,1)$ for all $t\geq 1$.
\end{itemize}
Constructing such a function $g$ is straightforward. One easily verifies that $gu\in \cE^{\prime k}(\mathbb{U})$ with full symbol $A$.
\end{proof}
Notice that Proposition \ref{prop:const pseudo from symb} immediately implies that if $u\in \cE^{\prime k}(\mathbb{U})$, then $$\singsupp(u)\subseteq \{0\}\times U\times \R_+.$$

\subsection{Definition of pseudodifferential operators and independence of the choice of basis}
Let $u\in\mathcal{E}^{\prime k}(\mathbb{U})$. Following the notation of \eqref{eqn:ev maps tangent group bisubmersions}, for $t\in \R_+^\times$ and $p\in M$ we define \begin{equation}
\begin{aligned}
&\ev_{t*}:\mathcal{E}^{\prime k}(\mathbb{U})\to \cD'(M\times M,\Omega^{1/2}),\quad\ u\mapsto \ev_{t*}(u_t)\\&\ev_{p,0*}:\mathcal{E}^{\prime k}(\mathbb{U})\to \cD'(\Gr(\cF)_p,\Omega^1),\quad u\mapsto \ev_{p,0*}(u_{p,0}),
\end{aligned}
\end{equation}
where $u_t$ is the restriction of $u$ to $\mathbb{U}_t$ and $u_{p,0}$ the restriction of $u$ to $V\times \{p\}\times \{0\}$. This is well defined because $u$ is transverse to $s:\mathbb{U}\to M\times \R_+$. We will write $\ev_{1*}(u)$ and $\ev_{p,0*}(u)$ instead of $\ev_{1*}(u_1)$ and $\ev_{p,0*}(u_{p,0})$. We now define our pseudodifferential operators.
\label{subsection:dfn of pseudo-diff}
\begin{definition}
\label{dfn:pseudodiff} 
An element $P\in \cD'(M\times M,\Omega^{1/2})$ belongs to $\Psi^k(\cF^\bullet)$ if 
\begin{enumerate}
\item $P$ is properly supported, i.e., $p_{1|\supp(P)}:\supp(P)\to M$ and $p_{2|\supp(P)}:\supp(P)\to M$ are proper where $p_1,p_2:M\times M\to M$ are the projections onto the first and second coordinates respectively.  
\item The singular support of $P$ is a subset of the diagonal $M\subseteq M\times M$.
\item For every $p\in M$ and for every graded basis $(V,\natural,\mathbb{U},U)$ with $p\in U$, there exists $u\in \mathcal{E}^{\prime k}(\mathbb{U})$ such that  $P$ and $\ev_{1*}(u)$ are equal on some neighborhood of $(p,p)\in M\times M$. 
\end{enumerate} 
A distribution $u\in \mathcal{E}^{\prime k}(\mathbb{U})$ such that $\ev_{1*}(u)=P$ on a neighborhood of $(p,p)$ is called a lift of $P$.
\end{definition}
In Definition \ref{dfn:pseudodiff}, we defined pseudodifferential operators as kernels admitting a lift in $\cE^{\prime k}(\mathbb{U})$ to every graded basis $(V,\natural,\mathbb{U},U)$.  In this section, we prove that it suffices to have a lift to some graded basis at each point in $M$. Once we have done so we can easily give examples of pseudodifferential operators.
\begin{remark}\label{rem:Lifiting_smooth} In Definition \ref{dfn:pseudodiff}, since the map $\ev_1:\mathbb{U}_1\to M\times M$ is a submersion, it is enough to find $u\in \mathcal{E}^{\prime k}(\mathbb{U})$ such that  $P$ and $\ev_{1*}(u)$ differ by a smooth function on some neighborhood of $(p,p)\in M\times M$. 
\end{remark}
\begin{lemma}
\label{lem:equivariant_lifting}
Let $(V,\natural,\mathbb{U},U)$ and $(V',\natural',\mathbb{U}',U')$ be graded bases with $U=U'$, $p\in U$ and let $\phi:\mathbb{U}' \to \mathbb{U}$ be as in Theorem \ref{second_main_tech_thm}. Suppose further that $\dom(\phi)=\mathbb{U}$ and $\phi$ is surjective. Then pushforward by $\phi$ defines a surjective map $\phi_*:\cE^{\prime k}(\mathbb{U}) \to \cE^{\prime k}(\mathbb{U}')$.
\end{lemma}
\begin{proof}Let $u\in\cE^{\prime k}(\mathbb{U})$. The map $\phi_{|\supp(u)}:\supp(u)\to \mathbb{U}'$ is proper because of Condition $c$ of Definition \ref{dfn:Eprimek} and the fact that $\phi$ preserves the $\R_+$ coordinate, which follows from Theorem \ref{second_main_tech_thm}.b. Hence $\phi_*(u)$ is well defined and belongs to $\cD'(\mathbb{U}',\Omega^{1/2}_{r,s})$ by \eqref{eqn:int_diff_basis_dens}. Conditions $a$ and $b$ of Definition \ref{dfn:Eprimek} are satisfied for $\phi_*(u)$ because of Theorem \ref{second_main_tech_thm}.b and Theorem \ref{second_main_tech_thm}.a respectively. Hence $\phi_*:\cE^{\prime k}(\mathbb{U}) \to \cE^{\prime k}(\mathbb{U}')$ is well defined. 
For surjectivity, let $u\in \cE^{\prime k}(\mathbb{U}')$, $A$ its full symbol.  Let $\phi_{0,x}:V\to V'$ denote the restriction of $\phi$ to $V\times \{x\}\times \{0\}$ for $x\in U$. Let $\psi_{0,x}:V \to V'$ be the differential of $\phi_{0,x}$ at $0$. Since $\phi_{0,x}$ is a submersion, $\psi_{0,x}$ is surjective. It is also $\R_+^\times$-equivariant because $\phi_{0,x}$ is $\R_+^\times$-equivariant. Since the family $\{\psi_{0,x}\}_{x\in M}$ varies smoothly in $x$, we can choose $p_x:V\to V$ a projection onto $\ker(\psi_{0,x})$ which varies smoothly in $x$ and is $\R_+^\times$-equivariant. Let $$L=\{(X,x,t,Y)\in \mathbb{U}'\times V:Y\in \ker(\psi_{0,x})\}.$$ The space $L$ is a smooth manifold of same dimension with $\mathbb{U}$. We define the map $$\kappa:\mathbb{U} \to L,\quad \kappa(X,x,t)=(\phi(X,x,t),p_x(X)).$$ The differential of $\kappa$ at $(0,x,0)$ is bijective. Since $\kappa$ is $\R_+^\times$-equivariant, it follows that, we can find a neighbourhood of $V\times U\times \{0\}\cup \{0\}\times U\times \R_+$ on which $\kappa$ is a diffeomorphism. We restrict $\kappa$ to such neighbourhood. We construct any homogeneous function $$\tilde{A}:\{(\xi,x,t,\eta)\in V^{\prime*}\times U\times \R_+\times V^*:\eta\in \ker(\psi_{0,x})^*\}\backslash\{(0,x,0,0):x\in U\} \to \C$$ which extends $A$ on $(\xi,x,t,0)$. We take $\tilde{v}$ to be the inverse Fourier transform of $\tilde{A}$ in the direction of $V$ and $\ker(\psi_{0,x})$ (\eqref{eqn:qhsdjliqsdfjihlJISD}), and then use a smooth function $g$ like we did in the end of the proof of Proposition \ref{prop:const pseudo from symb} to make $\tilde{v}$ supported in the image of $\kappa$. We then transform $\tilde{v}$ using $\kappa$ to a distribution on $v$ on $\mathbb{U}$ (if necessary, we modify its support the same way as we did at the end of the proof of Proposition \ref{prop:const pseudo from symb}). Since $\kappa$ is $\R_+^*$-equivariant, we have constructed $v\in \cE^{\prime k}(\mathbb{U})$ such that $\phi_*(v)$ has the same full symbol as $u$. Hence $w=u-\phi_*(v)$ is smooth. Since $\phi$ is a submersion we lift $w$ to a smooth function to $\mathbb{U}$ and add it to $v$. This finishes the proof.
 \end{proof}

\begin{prop}
\label{prop:change_of_basis}
Let $P\in \cD'(M\times M,\Omega^{1/2})$ be a properly supported distribution with singular support on the diagonal. Suppose that for every $p\in M$, there is some graded basis $(V,\natural,\mathbb{U},U)$ at $p$ and an element $u\in \cE^{\prime k}(\mathbb{U})$ such that $\ev_{1*}(u)$ and $P$ are equal on a neighbourhood of $(p,p)\in M\times M$. Then $P\in \Psi^k(\cF^\bullet)$.
\end{prop}
\begin{proof}
We first show that for every $p\in M$ and any \emph{minimal graded basis}  $(V',\natural',\mathbb{U}',U')$ at $p$, we can find $u'\in  \cE^{\prime k}(\mathbb{U}')$ such that $\ev_{1*}(u')$ and $P$ are equal on a neighbourhood of $(p,p)\in M\times M$.  By hypothesis, we can find some graded basis $(V,\natural,\mathbb{U},U)$ and $u\in\cE^{\prime k}(\mathbb{U})$ such that $\ev_{1*}(u)$ and $P$ are equal on a neighbourhood of $(p,p)\in M\times M$.  
Let $\phi$ be as in Theorem \ref{second_main_tech_thm}, $g\in C^\infty(\mathbb{U})$ be a smooth function with support in $\dom(\phi)$ that is equal to one in a neighbourhood of $(0,p,1)$, and is invariant for the $\R_+^\times$-action on $\mathbb{U}$. Now we have $gu \in \cE^{\prime k}(\mathbb{U})$ with support in $\dom(\phi)$, so $\phi_*(gu)$ is well defined. By properties of $\phi$ in Theorem \ref{second_main_tech_thm}, we get that $\phi_*(gu) \in \cE^{\prime k}(\mathbb{U}')$ and $\ev_{1*}(\phi_*(gu)) = \ev_{1*}(gu)$. Moreover, since $g\equiv 1$ on a neighbourhood of $(0,p,1)$, and since $\singsupp(u) \subseteq \{0\}\times U \times \R_+$, it follows that $\ev_{1*}(gu)$ and $\ev_{1*}(u)$ differ by a smooth function in some neighbourhood of $(p,p)\in M\times M$. We can then choose any $h\in C^\infty_c(\mathbb{U}')$ such that $\ev_{1*}(\phi_*(gu)+h)$ is equal to $\ev_{1*}(u)$ on a neighbourhood of $(p,p)$. Hence $\phi_*(gu)+h$ is the required lift.

Next let $(V'',\natural'',\mathbb{U}'',U'')$ be \emph{any} graded basis at $p\in M$.  Choose a minimal graded basis $(V',\natural',\mathbb{U}',U')$ at $p$.  By the previous discussion, we can find a lift $u'\in\cE^{\prime k}(\mathbb{U}')$ of $P$ on a neighbourhood of $p$.  Again let $\phi':\dom(\phi')\subseteq \mathbb{U}''\to \mathbb{U}'$ be as in Theorem \ref{second_main_tech_thm}. By reducing $\dom(\phi')$ if necessary, we can find $L\subseteq U''$ an open neighbourhood of $p$ and $\epsilon>0$ such that $$\dom(\phi')=\{(X,x,t)\in V\times L\times \R_+:\norm{\alpha_t(X)}<\epsilon\}.$$Hence $(V'',\natural'',\dom(\phi'),L)$ is a graded basis. It is also straightforward to check that $\cE^{\prime k}(\dom(\phi'))\subseteq \cE^{\prime k}(\mathbb{U})$. Hence without loss of generality we can suppose that $\dom(\phi')=\mathbb{U}''$.

  Let $g\in C^\infty(\mathbb{U}')$ with support in $\mathrm{Im}(\phi')$ that is equal to one in a neighbourhood of $(0,p,1)$, and is invariant for the $\R_+^\times$-action on $\mathbb{U}'$. By a argument similar to before, we can find  $h \in C^\infty_c(\Im(\phi'))$ such that $gu+h\in \cE^{\prime k}(\mathbb{U}')$ is a lift of $P$ at $p$. Hence by reducing $\mathbb{U}'$ we can without loss of generality suppose that $\Im(\phi')=\mathbb{U}'$.  Using Lemma \ref{lem:equivariant_lifting}, we can find a preimage $u''\in\cE^{\prime k}(\mathbb{U}'')$ of $u'$ under $\phi'_*$ which is then the required lift of $P$.
\end{proof}

\begin{cor}
\label{cor:global_basis_1}
Let $(V,\natural,\mathbb{U}, U)$ be a graded basis.  Then $\ev_{1*}(\cE^{\prime k}(\mathbb{U})) \subseteq \Psi^k(\cF^\bullet)$.  
\end{cor}
\begin{proof}
Let $u\in\cE^{\prime k}(\mathbb{U}) $. By Condition c of Definition \ref{dfn:Eprimek}, $\ev_{1*}(u)$ is compactly supported. By Proposition \ref{prop:const pseudo from symb}, $\singsupp(\ev_{1*}(u))$ lies on the diagonal. By Proposition \ref{prop:change_of_basis}, the result follows.
\end{proof}

\begin{ex}Let $X\in \cF^k$. We will show that  the Lie derivative $L_X:C^\infty(M,\Omega^{1/2})\to C^\infty(M,\Omega^{1/2})$ is an element of $\Psi^k(\cF^\bullet)$. Let $(V,\natural,\mathbb{U},U)$ be a graded basis at $p\in M$ with $v_0\in V^k$ such that $\natural(v_0)=X$, and let $g\in C^\infty_c(U)$ with $g=1$ in a neighbourhood of $p$. We define $u \in \cD'(\mathbb{U},\Omega^{1/2}_{r,s})$ by\begin{align*}\langle u,f\rangle=\int_{\R_+}\int_Mg(x)L_{v_0}(f)(0,x,t)dxdt,\quad f\in C^\infty_c(\mathbb{U},\Omega^{-1/2}_{r,s}\otimes \Omega^1\mathbb{U}).
\end{align*}It satisfies $$\mathrm{ev}_1(u)=g L_X,\quad \alpha_{\lambda*}(u)=\lambda^ku,\, \forall \lambda\in \R_+^\times.$$ 
Hence $u\in \cE^{\prime k}(\mathbb{U})$ and $u$ is a lift of $X$ at $p$. By Proposition \ref{prop:change_of_basis}, we get that $L_X\in \Psi^k(\cF^\bullet)$. This computation can be easily generalised to differential operators. We thus obtain the following.
\end{ex}
 \begin{prop}One has $\mathrm{Diff}^k_\cF(M,\Omega^{1/2})\subseteq \Psi^k(\cF^\bullet)$.
 \end{prop}
\subsection{Properties of pseudodifferential operators}\label{subsection:adjoint}
\begin{prop}\label{prop:properties_of_pseudo_diff}
\begin{enumerate}
 \item  For any $k\in \C$, $\Psi^k(\cF^\bullet)\subseteq \Psi^{k+1}(\cF^\bullet)$. 
 \item If $h \in C^\infty(M\times M)$ and $P\in\Psi^k(\cF^\bullet)$ then $hP\in \Psi^k(\cF^\bullet)$. If $h$ vanishes on the diagonal, then $hP\in\Psi^{k-1}(\cF^\bullet)$. In particular, if $f\in C^\infty(M)\subseteq \Psi^0(\cF^\bullet)$ and $P\in \Psi^k(\cF^\bullet)$, then $[f,P]\in \Psi^{k-1}(\cF^\bullet)$. 
\item  Let $p\in M$, $k\in \C$, $n\in \mathbb{N}$ and $P\in \Psi^k(\cF^\bullet)$. If 
$$\Re(k)<-\sum_{i=1}^Ni\dim(\cF^i_p/\cF^{i-1}_x)-nN,$$ then $P$ is of class $C^n$ on some neighbourhood of $(p,p)\in M\times M$.
\item  If $k_n\in \C$ is a sequence such that $Re(k_n)\to-\infty$ and $P\in \Psi^{k_n}(\cF^\bullet)$ for every $n$, then $P\in C^\infty(M\times M,\Omega^{1/2})$.
\end{enumerate}
\end{prop}
\begin{proof}
\begin{enumerate}
\item 	If $u\in\mathcal{E}^{\prime k}(\mathbb{U})$ is a lift of $P$, then $tu\in \mathcal{E}^{\prime k+1}(\mathbb{U})$ and is again a lift of $P$.
\item Let $u\in\cE^{\prime k}(\mathbb{U})$ be a lift of $P$ as in Definition \ref{dfn:pseudodiff}.  We define $g\in C^\infty(\mathbb{U})$ by $g(X,x,t) = h(\exp(\natural(\alpha_t(X)))\cdot x,x)$. The function $g$ is $\R^\times_+$-invariant. Hence $gu\in\cE^{\prime k}(\mathbb{U})$ is a lift of $hP$. If $h$ vanishes on the diagonal, then $t^{-1}g\in C^\infty(\mathbb{U})$ and $t^{-1}gu\in\cE^{\prime k-1}(\mathbb{U})$.
\item Let $(V,\natural,\mathbb{U},U)$ be a minimal graded basis at $p$. By Proposition \ref{prop:const pseudo from symb}, a lift $u\in \cE^{\prime k}(\mathbb{U})$ of $P$ is given by a oscillatory integral. One deduces the result from the absolute convergence of the oscillatory integral.
\item This follows directly from Part $c$.\qedhere
\end{enumerate}
\end{proof}
\begin{prop}\label{prop:adjoint} Let $k\in \C$ and $P\in \Psi^k(\cF^\bullet)$. Then \begin{enumerate}
\item $P$ is transverse to the projections $p_1,p_2:M\times M\to M$
\item $P^t\in \Psi^k(\cF^\bullet)$ and $P^*\in \Psi^{\bar{k}}(\cF^\bullet)$.
\end{enumerate}
\end{prop}
\begin{proof}
Since the above statements are local, we can without loss of generality suppose that $P=\ev_{1*}(u)$ for some graded basis $(V,\natural,\mathbb{U},U)$ and $u\in \cE^{\prime k}(\mathbb{U})$. Since $p_2\circ \ev_1(X,x,t)=x$, it follows from Condition a of Definition \ref{dfn:Eprimek} that $P$ is transverse to $p_2$. Let $r$ be the map from \eqref{eqn:bisub_rs}. Proposition \ref{prop:const pseudo from symb} implies that $$WF(u)\subseteq \{(0,x,t;\xi,0,0)\in T^*(V\times M\times \R)\}.$$ This intersects trivially with $\ker(dr)^{\perp}$. By \cite[Proposition 7]{LesManVas}, we deduce that $u$ is transverse to $r$. Since $p_1\circ \ev_1(X,x,t)=\exp(\natural(\alpha_t(X)))\cdot x$, it follows that $P$ is transverse to $p_2$.

 We now show that $P^t\in \Psi^k(\cF^\bullet)$. Consider the map $\phi$ given by Theorem \ref{thm:inverse_bisubmersion}. Let $g\in C^\infty(\mathbb{U})$ be an $\R_+^\times$-equivariant function such that $\supp(g)\subseteq \dom(\phi)$ and $g=1$ on a neighbourhood of $\{0\}\times U \times \R_+$. Since $u$ and $gu$ differ by a smooth function, it follows that $\ev_{1*}(gu)$ and $P=\ev_{1*}(u)$ differ by a smooth function. Hence without loss of generality we can suppose that $\supp(u)\subseteq \dom(\phi)$. The distribution $u$ is supported in $\dom(\phi)$. Hence $\phi_*(u)\in \cD'(\mathbb{U})$ is well-defined. We claim that $\phi_*(u)\in \cE^{\prime k}(\mathbb{U})$. It satisfies Condition $a$ of Definition \ref{dfn:Eprimek} because of Theorem \ref{thm:inverse_bisubmersion}.b and that $u$ is transverse to $r$ by the argument above. It satisfies Condition $b$ because $\phi$ is $\R_+^\times$-equivariant. It satisfies Condition $3$ because $\phi$ preserves the $\R$-coordinate by  Theorem \ref{thm:inverse_bisubmersion}.b. By Theorem \ref{thm:inverse_bisubmersion}.b, $$\ev_{1*}(\phi_*(u))=\iota_*(\ev_{1*}(u))=\iota_*(P)=P^t,$$
where $\iota:M\times M\to M\times M$ is the map $(x,y)\mapsto (y,x)$. Hence $P^t\in \Psi^k(\cF^\bullet)$ by Corollary \ref{cor:global_basis_1}. Since $P^*=\bar{P^t}$, we also get $P^*\in \Psi^{\bar{k}}(\cF^\bullet)$.
\end{proof}
Since $P\in \cD'(M\times M,\Omega^{1/2})$, by Schwartz kernel theorem, it is an operator $P:C^\infty_c(M,\Omega^{1/2})\to \cD'(M,\Omega^{1/2})$. Proposition \ref{prop:adjoint} implies that   \begin{itemize}
\item $P(C^\infty_c(M,\Omega^{1/2}))\subseteq C^\infty_c(M,\Omega^{1/2})$ 
\item $P$ extends to a continuous linear map $\cD'(M,\Omega^{1/2})\to\cD'(M,\Omega^{1/2})$.
\end{itemize}
\begin{prop}\label{prop:composition}
If $P\in \Psi^k(\cF^\bullet)$, $Q\in \Psi^{l}(\cF^\bullet)$, then $P\star Q\in \Psi^{k+l}(\cF^\bullet)$.
\end{prop}
\begin{proof}
Since $P,Q$ are properly supported, the distribution $P\star Q$ is well-defined and is properly supported. It is also clear that the singular support of $P\star Q$ lies on the diagonal. It remains to check the third condition of Definition \ref{dfn:pseudodiff}. Let $p\in M, (\mathfrak{g},\natural,\mathbb{U},U)$ a graded Lie basis with $p\in M$, $u\in \mathcal{E}^{\prime k }(\mathbb{U}),v\in \mathcal{E}^{\prime l }(\mathbb{U})$ lifts of $P$ and $Q$ respectively, $\phi$ as in Theorem \ref{main_tech_thm}. We can up to adding a smooth function to $u$ and $v$, suppose that the distribution $u\star v$ defined by \begin{equation}\label{eqn:ustarv}
 u\star v(Y,X,x,t)=u(Y,\exp(\natural(\alpha_t(X)))\cdot x,t)v(X,x,t)
\end{equation}has support in $\dom(\phi)$. Notice that we are allowed to define $u\star v$ by \eqref{eqn:ustarv} because of Condition $a$ of Definition \ref{dfn:Eprimek}. We claim that $w=\phi_*(u\star v) \in \mathcal{E}^{\prime k+l}	(\mathbb{U})$. Conditions $a$ and $c$ of Definition \ref{dfn:Eprimek} are straightforward to check. For Condition $b$, one has 
 \begin{align*}
 \alpha_{\lambda}(w)-\lambda^{k+l}  w = \phi_*\Bigg(\Big((\alpha_{\lambda*}(u)-\lambda^ku)\star\alpha_{\lambda*}v\Big)+\lambda^k\Big(u\star(\alpha_{\lambda*} v-\lambda^lv)\Big)\Bigg).
 \end{align*}
 By \eqref{eqn: equiv cond def pseudo diff}, $\alpha_{\lambda}(u)-\lambda^ku$, and $\alpha_{\lambda}(v)-\lambda^lv$ are smooth. Let 
 $$\kappa_1=\Big((\alpha_{\lambda*}(u)-\lambda^ku)\star\alpha_{\lambda*}v\Big),\quad \kappa_2=\lambda^k\Big(u\star(\alpha_{\lambda*} v-\lambda^lv)\Big).$$
 One has $$WF(\kappa_1)\subseteq \{(X,0,x,t;0,\eta,0,0)\in T^*(\mathfrak{g}\times \mathfrak{g}\times U\times \R_+)\}.$$ This intersects trivially with $\ker(d\phi)^\perp$ because of Theorem \ref{main_tech_thm}.d. By \cite[Proposition 7]{LesManVas}, we deduce that $\phi_*(\kappa_1)$ is smooth. By a similar argument we deduce that $\phi_*(\kappa_2)$ is smooth. Therefore $w\in \mathcal{E}^{\prime k+l}(\mathbb{U})$. By Theorem \ref{main_tech_thm}.b,
 $$\ev_{1*}(w)=\ev_{1*}(u)\star \ev_{1*}(u).$$ 
Hence $w$ is a lift of $P\star Q$ near $p$.\end{proof}
\subsection{Distributions on graded nilpotent Lie groups} \label{sec:dis_groups}
In this and the following subsections, we will make use of Hilbert space techniques. It will greatly simplify the exposition if we assume that the underlying manifold $M$ is compact.  If we don't make this assumption, we will be forced to use local $L^2$-spaces, local Sobolev spaces, and pro-$C^*$-algebras. \textbf{For this reason, in Sections \ref{sec:dis_groups},\ref{sec:Sobolev},\ref{sec:princip symb},\ref{sec:Parametrised}, and \ref{sec:para order 0}, we shall make the assumption that $M$ is a compact manifold. In Section \ref{sec:noncompact}, we extend our results to non compact manifolds.}

Let $\mathfrak{g}=\oplus_{i=1}^N\mathfrak{g}_i$ be a graded nilpotent Lie algebra. We view $\mathfrak{g}$ as a Lie group by the BCH formula \eqref{eqn:BCHintro}. As usual, we write $\alpha_\lambda$ for the dilations on $\mathfrak{g}$ given by 
$\alpha_\lambda\left(\sum_{i=1}^NX_i\right)=\sum_{i=1}^N\lambda^i X_i $ for $ \lambda\in \R_+,$ and $\hat\alpha_\lambda$ for the dilations of the dual space $\lie{g}^*$ given by $\hat\alpha_\lambda(\xi)(X) = \xi(\alpha_\lambda(X))$. We extend the action to $\mathfrak{g}\times M$ and $\mathfrak{g}^*\times M$ by acting trivially on $M$. If $u\in  \cD'(\mathfrak{g}\times M,\Omega^1\mathfrak{g})$, then we define $\alpha_{\lambda*}u$ by $$\langle\alpha_{\lambda*}u,f\rangle=\langle u,\alpha_\lambda^*f\rangle,\quad f\in C^\infty_c(\mathfrak{g}\times M,\Omega^1M). $$
We denote by $\cE'(\mathfrak{g}\times M,\Omega^1\mathfrak{g})$ the $*$-algebra of compactly supported $u\in \cD'(\mathfrak{g}\times M,\Omega^1\mathfrak{g})$ which are transverse to the bundle projection $p:\mathfrak{g} \times M \to M$. The $*$-algebra structure comes from fiberwise convolution and adjoint. It is unital with the unit being the distribution $$\langle 1,f\rangle=\int_M f(0,x),\quad f\in C^\infty_c(\mathfrak{g}\times M,\Omega^1M).$$
\begin{definition} Let $\cE^{\prime k}(\mathfrak{g}\times M)$ be the space of distributions $u\in \cE'(\mathfrak{g}\times M,\Omega^1\mathfrak{g})$ such that 
for every $\lambda\in \R_+^\times$, 
\begin{equation}\label{homog_condition_symbols}
 \alpha_{\lambda*} u-\lambda^k u\in C^\infty_c(\mathfrak{g}\times M,\Omega^1\mathfrak{g}).
\end{equation}
 \end{definition}
Note that $1\in\cE^{\prime 0}(\mathfrak{g}\times M) $ and if $u\in\cE^{\prime k}(\mathfrak{g}\times M)$ and $v\in\cE^{\prime l}(\mathfrak{g}\times M)$ then
$u\star v\in \cE^{\prime k+l}(\mathfrak{g}\times M)$ and $u^*\in \cE^{\prime \bar{k}}(\mathfrak{g}\times M)$.
The following is analogous to Proposition \ref{prop:const pseudo from symb}. Since the proof is very similar, it is omitted.
 \begin{prop}\label{prop:Taylorgroups} 
Let $u\in \mathcal{E}^{\prime k}(\mathfrak{g}\times M)$. Then there exists a unique smooth function $B\in C^\infty((\lie{g}^*\setminus\{0\})\times M)$ called the full symbol of $u$ such that
\begin{enumerate}
 \item
 One has  \begin{equation}\label{eqn:B_homog}
   B(\hat{\alpha}_\lambda(\xi),x) = \lambda^kB(\xi,x),\quad \forall\lambda\in \R_+^\times,(\xi,x)\in (\lie{g}^*\setminus\{0\})\times M)
\end{equation}
 \item
 If $\chi\in C^\infty_c(\lie{g}^*)$ is equal to $1$ on a neighbourhood of $0$, then \begin{equation}\label{eqn:qsuidfhpiqspdjfo}
   f(X,x)=u(X,x)-\int_{\mathfrak{g}^*}e^{i\langle \xi,X\rangle}(1-\chi(\xi))B(\xi,x)d\xi\in C^\infty(\mathfrak{g}\times M,\Omega^1\mathfrak{g})
\end{equation} and $f(X,x)$ and all its derivative in $x$ are Schwartz in $X$ uniformly in $x$.
\end{enumerate}
Conversely if $B\in C^\infty((\lie{g}^*\setminus\{0\})\times M)$ is homogeneous of degree $k$, then there exists $u\in \mathcal{E}^{\prime k}(\mathfrak{g}\times M)$ such that \eqref{eqn:qsuidfhpiqspdjfo} is satisfied.
\end{prop}
\begin{definition}
Let $\cS_0(\mathfrak{g})$ be the space of $f\in C^\infty(\mathfrak{g},\Omega^1\mathfrak{g})$ such that \begin{itemize}
\item $f$ is Schwartz
\item If $\hat{f}\in C^\infty(\mathfrak{g}^*)$ denotes the Fourier transform of $f$, then $\hat{f}$ is flat at $0$.
\end{itemize}
\end{definition} Let $u\in \cE^{\prime k}(\mathfrak{g}\times M)$ and  $B\in C^\infty((\lie{g}^*\setminus\{0\})\times M)$ its full symbol, $x\in M$. In \cite[Proposition 2.2]{ChrGelGloPol}, it is shown that the linear map given by convolution with the inverse Fourier transform of $B$
\begin{align*}
f\in S_0(\mathfrak{g})\mapsto \left(Y\mapsto \int_\mathfrak{g}\int_{\mathfrak{g}^*} B(\xi,x)e^{i\langle \xi,X\rangle}f(\BCH(-X,Y)) \right)\in S_0(\mathfrak{g}),
\end{align*}
is well defined and continuous. We denote this map by $\sigma^k(u,x)$. Let $\pi\in \hat{\mathfrak{g}}$ be a non trivial irreducible unitary representation of $\mathfrak{g}$ acting on a Hilbert space $L^2\pi$. We define $\sigma^k(B,x,\pi)$ to be the unbounded operator acting on $L^2\pi$ by the formula $$\sigma^k(u,x,\pi)(\pi(f) \xi)=\pi\left(\sigma^k(u,x)(f)\right)\xi,\quad f\in S_0(\mathfrak{g}),\xi \in L^2\pi.$$ This map can be extended to a linear map $$\sigma^k(u,x,\pi):C^\infty(\pi)\to C^\infty(\pi).$$ We refer the reader to \cite[The discussion after Theorem 2.5]{ChrGelGloPol} for more details. By \cite[Proposition 2.3 and Proposition 3.3]{ChrGelGloPol}, for any $u\in\cE^{\prime k}(\mathfrak{g}\times M)$, $v\in\cE^{\prime l}(\mathfrak{g}\times M)$ we have
\begin{align}\label{eqn:algebra_prop_principal_symbols_at_0}
\sigma^k(u,x,\pi)\circ \sigma^l(u,x,\pi) =\sigma^{k+l}(u\star v,x,\pi), \quad \sigma^{\bar{k}}(u^*,x,\pi) \subseteq \sigma^k(u,x,\pi)^*.
\end{align}
\begin{thm}\label{thm:Bounded_symbol} Let $u\in  \cE^{\prime k}(\mathfrak{g}\times M)$, $x\in M$, $\pi \in \hat{\mathfrak{g}}\backslash \{1_\mathfrak{g}\}$. Then 
\begin{enumerate}
\item If $Re(k)=0$, then $\sigma^k(u,x)$ and $\sigma^k(u,x,\pi)$ extend to a bounded operator $L^2\mathfrak{g}\to L^2\mathfrak{g}$ and $L^2\pi \to L^2\pi$ respectively. Moreover $$\norm{\sigma^k(u)}:=\sup_{x\in M}\norm{\sigma^k(u,x)}=\sup_{x\in M,\pi\in \hat{\mathfrak{g}}\backslash \{1_\mathfrak{g}\}}\norm{\sigma^k(u,x,\pi)}<+\infty.$$
\item If $Re(k)<0$, then $\sigma^k(u,x,\pi)$ extends to a compact operator  $L^2\pi \to L^2\pi$.
\end{enumerate}
\end{thm}
\begin{proof}
For $k=0$, the fact that the operator $\sigma^0(u,x)$ is bounded is well known, see \cite[the paragraph preceding Proposition 2.5]{TaylorBook} or \cite{GoodmanIntegrals}. This immediately implies that  $\sigma^0(u,x,\pi)$ is bounded and that $\norm{\sigma^0(u,x)}=\sup_{\pi\in \hat{\mathfrak{g}}\backslash \{1_\mathfrak{g}\}}\norm{\sigma^0(u,x,\pi)}$. One can show that since $u$ varies smoothly in $x$, one has a uniform bound for $\sigma^0(u,x)$ as $x$ varies in $M$, see  \cite{ChrGelGloPol}. For $k\in \C$ with $\Re(k)=0$, the theorem follows by the above applied to $ u^*\star u\in \cE^{\prime 0}(\mathfrak{g}\times M)$.

The second part follows from the Plancherel formula. To see this, let $n$ be big enough so that if $C\in C^\infty((\mathfrak{g}^*\backslash\{0\})\times M)$ denotes the full symbol of $(u^*\star u)^n$, then $C$ is integrable at infinity. Let $O\subseteq \mathfrak{g}^*$ be the orbit associated to $\pi$. Since $\pi$ is non trivial and all orbits are closed \cite[Theorem 3.1.4]{BookNilpotentGroups}, it follows that $\int_O|C(\xi,x)|d\mu_O(\xi)<+\infty$ where $\mu_O$ is the canonical measure on $O$ coming from its symplectic structure. Now let $a_n\in C^\infty_c(\mathfrak{g},|\Lambda|^1\mathfrak{g})$ be an approximate of the identity in $C^*\mathfrak{g}$. By the Plancherel formula \cite[Theorem 4.3.1]{BookNilpotentGroups} applied to $(\sigma^k(u,x,\pi)^*a_n^2\sigma^k(u,x,\pi))^n$ and the Lebesgue monotone convergence theorem applied to the left hand side and Lebesgue dominated convergence theorem applied to the right hand side, we deduce $$Tr((\sigma^k(u,x,\pi)^*\sigma^k(u,x,\pi))^n)=\int_OC(\xi,x)d\mu_O(\xi)<+\infty.$$ Hence $(\sigma^k(u,x,\pi)^*\sigma^k(u,x,\pi))^n$ is a bounded compact operator. Hence $\sigma^k(u,x,\pi)$ is compact.
\end{proof}
The following theorem is a generalization of Helffer and Nourrigat's Theorem \cite{HelfferRockland} to left-invariant pseudodifferential operators on a graded nilpotent Lie group.  It is central to what follows.

\begin{thm}[\cite{Glowacki1,Glowacki2,ChrGelGloPol}]\label{Parametrix_symbols}If $u\in \cE^{\prime k}(\mathfrak{g}\times M)$, then the following are equivalent:
\begin{enumerate}
\item For every $x\in M$, and $\pi \in \hat{\mathfrak{g}}\backslash \{1_\mathfrak{g}\}$, $\sigma^k(u,x,\pi)$ is injective.
\item There exists $v\in \cE^{\prime -k}(\mathfrak{g}\times M)$ such that  $1-v\star u\in C^\infty_c(\mathfrak{g}\times M,\Omega^1\mathfrak{g})$.
\end{enumerate}
Furthermore if $k=0$, then the previous statements are equivalent to the following \begin{enumerate} \setcounter{enumi}{2}
 \item For every $x\in M$, and $\pi \in \hat{\mathfrak{g}}\backslash \{1_\mathfrak{g}\}$, then the bounded extension of $\sigma^0(u,x)$ is left invertible.
\end{enumerate}
\end{thm}

\begin{definition}
\label{def:Rockland1}
If $u$ and $u^*$ satisfy the conditions of Theorem \ref{Parametrix_symbols}, then we say that $u$ satisfies the strong $*$-Rockland condition.
\end{definition}

The word ``strong'' is used here because later, when we treat the Helffer-Nourrigat conjecture, it will suffice to consider distributions $u\in\mathcal{E}^{\prime k}(\mathfrak{g} \times M)$ for which $\sigma^k(u,x,\pi)$ is injective only on those representations $\pi$ which belong to the Helffer-Nourrigat cone.
We remark that if $u$ satisfies the strong $*$-Rockland condition, then  \begin{itemize}
\item There exists $v\in \cE^{\prime -k}(\mathfrak{g}\times M)$ such that  $1-v\star u$ and $1-u\star v$ are in $C^\infty_c(\mathfrak{g}\times M,\Omega^1\mathfrak{g})$.
\item For every $x\in M$, and $\pi \in \hat{\mathfrak{g}}\backslash \{1_\mathfrak{g}\}$, $\sigma^k(u,x,\pi):C^\infty(\pi)\to C^\infty(\pi)$ is a bijection with inverse $\sigma^{-k}(v,x,\pi)$.
\item If $\Re(k)=0$, then the bounded extension of $\sigma^k(u,x)$ is invertible with inverse $\sigma^{\bar{k}}(v,x)$.
\end{itemize}

\begin{thm}[{\cite[Theorem 6.1]{ChrGelGloPol}}]\label{thm:christ_Gall}We can find $u_k\in \mathcal{E}^{\prime k}(\mathfrak{g} \times M)$ for each $k\in \C$ such that \begin{enumerate}
\item $u_0=1$.
\item For any $k,l\in \C$, $u_k\star u_l-u_{k+l}\in C^\infty_c(\mathfrak{g}\times M,\Omega^1\mathfrak{g})$
\item For any $k\in \C$, $u_k^*-u_k\in C^\infty_c(\mathfrak{g}\times M,\Omega^1\mathfrak{g})$
\end{enumerate}
In particular $u_k$ satisfies the strong $*$-Rockland condition for each $k\in \C$.
\end{thm}
Let $\pi\in \hat{\mathfrak{g}}\backslash\{1_\mathfrak{g}\}$, $k>0$, $u\in \mathcal{E}^{\prime k}(\mathfrak{g} \times M)$ satisfying the strong $*$-Rockland condition. We define the Sobolev space $H^k(\pi)\subseteq L^2\pi$ to be the domain of the closure of $\sigma^k(u,x,\pi)$. By Theorem \ref{Parametrix_symbols} and Theorem \ref{thm:Bounded_symbol}, it follows that the space $H^k(\pi)$ is independent of $u$. We take $H^{-k}(\pi)$ to be its dual. We have thus obtained a family of Sobolev spaces $H^k(\pi)$ for $k\in \R$ with $H^0(\pi)=L^2\pi$ and such that one has bounded compact inclusions $H^k(\pi)\to H^l(\pi)$ for $k>l$. We remark that \begin{align}\label{eqn:inter}
 \bigcap_{k\in \R}H^k(\pi)=C^\infty(\pi),
\end{align} because for any $D$ right invariant differential operator on $\mathfrak{g}$ of classical order $k$, one has $H^{kN}(\pi)\subseteq \dom\left(\overline{\pi(D)}\right)$. 
\begin{prop}\label{prop:symbextSob} Let $k\in \C$, $u\in \mathcal{E}^{\prime k}(\mathfrak{g} \times M)$, $s\in \R$ then $\sigma^k(u,x,\pi)$ extends to a bounded operator $$\sigma^k(u,x,\pi):H^{s+\Re(k)}(\pi)\to H^{s}(\pi).$$ 
\end{prop}
\begin{proof}
This follows directly from Theorem \ref{thm:Bounded_symbol}.
\end{proof}
We denote by $C^*(\mathfrak{g}\times M)$ the completion of $C^\infty_c(\mathfrak{g}\times M,\Omega^1\mathfrak{g})$ with respect to the norm $$\norm{f}=\sup_{x\in M} \norm{f(\cdot,x)}_{B(L^2\mathfrak{g})},$$ where $f(\cdot,x)$ acts on $L^2\mathfrak{g}$ by convolution. Equivalently $C^*(\mathfrak{g}\times M)=C^*\mathfrak{g}\otimes C(M)$. If $k\in \C$ and $u\in \cE^{\prime k}(\mathfrak{g}\times M)$, then by taking convolution to the left by $u$, $u$ can be regarded as an unbounded multiplier of $C^*(\mathfrak{g}\times M)$ with domain $C^\infty_c(\mathfrak{g}\times M,\Omega^1\mathfrak{g})$.
\begin{prop}\label{prop:bounded_prophdslijqf}
Let $k\in \C$ and $u\in \cE^{\prime k}(\mathfrak{g}\times M)$. 
\begin{enumerate}
\item If $\Re(k)<0$, then $u$ extends to a compact multiplier of $C^*(\mathfrak{g}\times M)$, i.e., $u\in C^*(\mathfrak{g}\times M)$.
\item If $\Re(k)=0$, then $u$ extends to a bounded multiplier of $C^*(\mathfrak{g}\times M)$.
\end{enumerate}
\end{prop}
\begin{proof}
For Part a, if $n$ is big enough, then by Proposition \ref{prop:Taylorgroups}, the full symbol of $(u^*\star u)^n$ is integrable. Hence $(u^*\star u)^n\in C_c^0(\mathfrak{g},\Omega^1\mathfrak{g})\subseteq C^*(\mathfrak{g}\times M)$. For Part b, by replacing $u$ with $u^*u$, we can suppose that $k=0$. It is enough to show that $u_x$ acting on $L^2\mathfrak{g}$ is bounded uniformly on $x$. To see this, we refer the reader again to \cite[the paragraph preceding Proposition 2.5]{TaylorBook}.
\end{proof}
We denote by $\overline{\cE^{\prime 0}(\mathfrak{g}\times M)} $ the closure of $\cE^{\prime 0}(\mathfrak{g}\times M)$ inside $B(L^2(\mathfrak{g}))\otimes C(M)$. It is clear that the closure $\overline{\cE^{\prime -1}(\mathfrak{g}\times M)}$ of $\cE^{\prime -1}(\mathfrak{g}\times M)$ inside $B(L^2(\mathfrak{g}))\otimes C(M)$ is equal to $C^*(\mathfrak{g}\times M)$. Let $\Sigma^*(\mathfrak{g}\times M)$ be the closure in $B(L^2\mathfrak{g})\otimes C(M)$ of elements of the form $\sigma^0(u)$ for $u\in \cE^{\prime 0}(\mathfrak{g}\times M)$. \begin{thm}\label{thm:shortexactseqsymb}The natural sequence \begin{equation}\label{eqn:jqfdsqsuidfjpiqsf}
0\to C^*(\mathfrak{g}\times M)\to \overline{\cE^{\prime 0}(\mathfrak{g}\times M)}\to \Sigma^*(\mathfrak{g}\times M)\to 0 
\end{equation}  is exact.
\end{thm}
\begin{proof}
First the map $\overline{\cE^{\prime 0}(\mathfrak{g}\times M)}\to \Sigma^*(\mathfrak{g}\times M)$ is well defined because for any $x\in M$, $\pi \in \hat{g}\backslash \{1_\mathfrak{g}\}$, and $u\in \cE^{\prime 0}(\mathfrak{g}\times M)$ \begin{equation}\label{eqn:qhsiudfjiq}
 \norm{\sigma^0(u,x,\pi)}\leq \norm{u}_{B(L^2(\mathfrak{g}))\otimes C(M)}.
\end{equation}To see this notice that, the full symbol $B$ of $u$ is integrable near $0$, hence defines a tempered distribution on $\mathfrak{g}^*\times M$. Let $v\in \cD'(\mathfrak{g}\times M,\Omega^1\mathfrak{g})$ be its Fourier transform. Then $w=u-v\in \cS(\mathfrak{g}\times M,\Omega^1\mathfrak{g})$ is a Schwartz function and for $x\in M$, $\pi\in \hat{\mathfrak{g}}\backslash \{1_\mathfrak{g}\}$, $\pi(u_x)-\sigma^0(u,x,\pi)=\pi(w(\cdot,x))$. By replacing $\pi$ with $\pi\circ \alpha_\lambda$, \eqref{eqn:qhsiudfjiq} follows from the Riemann-Lebesgue lemma, i.e., $\lim_{\lambda\to +\infty}\norm{\pi\circ \alpha_\lambda(w)}=0$. This follows easily from the Plancherel formula $$\norm{\pi\circ \alpha_\lambda(w)}\leq \norm{\pi\circ \alpha_\lambda(w)}_{HS}=\mathrm{Tr}(\pi\circ \alpha_\lambda(w^*\star w))=\int_{\alpha_\lambda(O_\pi)}\widehat{w^*\star w}(\xi,x)d\xi,$$where $\norm{\cdot}_{HS}$ denotes the Hilbert-Schmidt norm, $\widehat{w^*\star w}$ denotes the Fourier transform of $w^*\star w\in \cS(\mathfrak{g}\times M,\Omega^1\mathfrak{g})$.

The map $\overline{\cE^{\prime 0}(\mathfrak{g}\times M)}\to \Sigma^*(\mathfrak{g}\times M)$ is obviously surjective. We now prove exactness. Let $u\in \cE^{\prime 0}(\mathfrak{g}\times M)$ such that $\sigma^0(u)$ is positive invertible in $\overline{\Sigma^*(\mathfrak{g}\times M)}$. We claim that the square root of $\sigma^0(u)$ is in $\sigma^0(\cE^{\prime 0}(\mathfrak{g}\times M))$. To see this, by replacing $u$ with $\epsilon u$ for $\epsilon>0$, we can suppose that $\norm{1-\sigma^0(u)}_{\Sigma^*(\mathfrak{g}\times M)}<1$. Hence the square root of $\sigma^0(u)$ is given by $$\sigma^0(u)^{1/2}=\sum_{n=0}^\infty a_n (1-\sigma^0(u))^n,$$where $\sum_{n=0}^\infty a_n (1-z)^n$ is the analytic expansion of $f(z)=\sqrt{z}$ at $z=1$ whose radius of convergence is equal to $1$. It then follows that $\sigma^0(u)^{1/2}\in  \sigma^0(\cE^{\prime 0}(\mathfrak{g}\times M))$, by the argument on Page 54 before Theorem 5.9 in \cite{ChrGelGloPol}, replacing the Neumann series by $\sum_{n=0}^\infty a_n (1-z)^n$.  Now let $u\in \cE^{\prime 0}(\mathfrak{g}\times M)$ be arbitrary, $C=\norm{\sigma^0(u)}^2_{\Sigma^*(\mathfrak{g}\times M)}$. Then for $\epsilon>0$, $\sigma^0((C+\epsilon)1-u^*u)$ is positive invertible. Hence there exists $v\in \cE^{\prime 0}(\mathfrak{g}\times M)$ such that $v^*v+u^*u-(C+\epsilon)1\in C^\infty_c(\mathfrak{g}\times M,\Omega^1\mathfrak{g})$. Hence $$\norm{u}_{\frac{\overline{\cE^{\prime 0}(\mathfrak{g}\times M)}}{C^*(\mathfrak{g}\times M)}}^2\leq C+\epsilon$$ Since this holds for all $\epsilon$, it follows that \begin{equation}\label{eqn:hsqf}
 \norm{u}_{\frac{\overline{\cE^{\prime 0}(\mathfrak{g}\times M)}}{C^*(\mathfrak{g}\times M)}}^2\leq C.
\end{equation} There is an obvious map\begin{equation}\label{eqn:jqsuidfjpiqsf}
\sigma^0(\cE^{\prime 0}(\mathfrak{g}\times M))\to \frac{\overline{\cE^{\prime 0}(\mathfrak{g}\times M)}}{C^*(\mathfrak{g}\times M)}
\end{equation}  which sends an element $\sigma^0(u)$ to $u$. By \eqref{eqn:hsqf}, the map \eqref{eqn:jqsuidfjpiqsf} is continuous, hence it extends to the closure \begin{equation}
\Sigma^*(\mathfrak{g}\times M)\to \frac{\overline{\cE^{\prime 0}(\mathfrak{g}\times M)}}{C^*(\mathfrak{g}\times M)}.
\end{equation} 
Hence \eqref{eqn:jqfdsqsuidfjpiqsf} is exact.
\end{proof}
Theorem \ref{thm:shortexactseqsymb} appears in \cite{ewert2021pseudodifferential} and in \cite{FischerDefect}. 
\subsection{Sobolev Spaces}\label{sec:Sobolev}
Since we still assume $M$ is compact, by Remark \ref{rem:Global basis}, there exists a global graded Lie basis $(\lie{g},\natural,\mathbb{U},M)$ which we fix for the rest of this section. By a straightforward partition of unity argument, we deduce that
 $$\Psi^k(\cF^\bullet) = \ev_{1*}(\cE^{\prime k}(\mathbb{U})) + C_c^\infty(M\times M,\Omega^{1/2}	).$$ 
If $P\in \Psi^k(\cF^\bullet)$, then an element $u\in \cE^{\prime k}(\mathbb{U})$ such that $P-\ev_1(u)\in C_c^\infty(M\times M,\Omega^{1/2})$ will be called \textit{a global lift} of $P$.  Let $u\in \cE^{\prime k}(\mathbb{U})$. We denote by $u_0$ the restriction of $u$ to $\mathfrak{g}\times M \times \{0\}$. Clearly $u_0\in\cE^{\prime k}(\mathfrak{g}\times M)$. We remark that if $A$ and $B$ are the full symbols of $u$ and $u_0$ respectively, then $B$ is the restriction of $A$ to $(\lie{g}^*\setminus \{0\})\times M\times\{0\}$.
\begin{prop}
\label{rmk:symbol_convolution}
Let $P\in \Psi^k(\cF^\bullet),Q\in \Psi^k(\cF^\bullet)$ and $u\in \cE^{\prime k}(\mathbb{U})$ and $v\in \cE^{\prime l}(\mathbb{U})$ global lifts of $P$ and $Q$ respectively. Then one can find $w\in \cE^{\prime k+l}(\mathbb{U})$ and $u'\in \cE^{\prime \bar{k}}(\mathbb{U})$ such that \begin{enumerate}
\item $w$ and $u'$ are global lifts of $P\star Q$ and $P^*$ respectively
\item $w_0=u_0\star v_0$ and $u'_0=u_0^*$.
\end{enumerate} 
\end{prop}
\begin{proof}
In the proof of Proposition \ref{prop:composition}, we obtained a lift $w=\phi_*(u\star v)$ of $P\star Q$. By Theorem \ref{main_tech_thm}.c, we see that $w_0 = u_0 \star v_0$. Similarly, in the proof of Proposition \ref{prop:adjoint}, the lift $u'$ of $P^t$ also satisfies $(u'_0)^t=u^t_0$ because of Theorem \ref{thm:inverse_bisubmersion}.d. By taking complex conjugation of $u'$, one obtains a lift of $P^*$.
\end{proof}

In the classical calculus, the elements of $\Psi^k(M)$ inside $\Psi^{k+1}(M)$ can be identified as those with vanishing principal symbol.  This is no longer the case in the present situation due to the subtle nature of the notion of principal symbol as illustrated in Example \ref{exs:counter_examples}. However, the result is true at the level of lifts to graded basis, as we now show.
\begin{prop}
\label{prop:one_less_order}
Let $P\in\Psi^{k+1}(\cF^\bullet)$ and $u$ a global lift.  If $u_0\in C^\infty_c(\mathfrak{g}\times M,\Omega^1\mathfrak{g})$, then $P\in \Psi^k(\cF^\bullet)$.
\end{prop}
\begin{proof}
We can find $h\in C^\infty_c(\mathbb{U},\Omega^{1/2}_{r,s})$ with restrictions $h_0 = u_0$ on $\mathfrak{g}\times M \times \{0\}$ and $h_1= 0$ on $\mathbb{U}_{1}$, so by replacing $u$ by $u-h$ we may assume that $u_0= 0$.  It follows that $u'= t^{-1}u \in \cE^{\prime k}(\mathbb{U})$ with $\ev_{1*}(u') =\ev_{1*}(u)$ which is equal to $P$ modulo a smoothing operator.
\end{proof}

\begin{prop}\label{prop:lifting_prop}
Let $v\in \cE^{\prime k}(\mathfrak{g}\times M)$.  Then there exists $u\in \cE^{\prime k}(\mathbb{U})$ such that $u_0=v$.
\end{prop}

\begin{proof}
 Let $B \in C^\infty(\left(\mathfrak{g}^*\setminus\{0\}\right)\times M)$ be the smooth function obtained from $v$ by Proposition \ref{prop:Taylorgroups}. One can extend $B$ to a smooth function $A\in C^\infty((\mathfrak{g}\times M\times\R_+)\backslash (\{0\}\times M\times \{0\}))$ satisfying $a$ and $b$ of Proposition \ref{prop:const pseudo from symb}. Then, by Proposition \ref{prop:const pseudo from symb}, there exists $u'\in \cE^{\prime k}(\mathbb{U})$ whose full symbol is $A$. Hence $v-u'_0$ is Schwartz class. By Definition \ref{dfn:Eprimek}.c and since $v$ is compactly supported, we get that $v-u_0\in C^\infty_c(\mathfrak{g}\times M,\Omega^1\mathfrak{g})$. Now extend $v-u'_0$ to a function $h\in C^\infty_c(\mathbb{U},\Omega^{1/2}_{r,s})$ with $h_{|\mathfrak{g}\times M \times\{0\}} = v-u'_0$. Putting $u=u'+h$ gives the desired element.
\end{proof}
\begin{definition} 
\label{def:operator_strong_Rockland}
If $P\in \Psi^k(\cF^\bullet)$, then we say that $P$ satisfies the strong $*$-Rockland condition if there exists a global lift $u\in \cE^{\prime k}(\mathbb{U})$ such that $u_0$ satisfies the strong $*$-Rockland condition of Definition \ref{def:Rockland1}.
\end{definition}
As mentioned previously, ultimately we will only need the Rockland condition for representations in the Helffer-Nourrigat cone. Nonetheless, the strong $*$-Rockland condition is useful for defining our Sobolev spaces.
\begin{ex}\label{ex:diff pseudo strong rock}
 Let $(x_{ij}) \in \lie{g}_i$ be a finite family of elements of $\lie{g}$ such that the vectors $x_{i1},x_{i2},\ldots$ generate $\lie{g}_i$ for every $i$.  Let $s\in \N$ such that $i|s$ for all $1\leq i\leq N$.  Put $X_{ij} = \natural (x_{ij}) \in \cX(M)$ and define the differential operator
$$D=\sum_{ij}(-1)^\frac{s}{i}L_{X_{ij}}^{\frac{2s}{i}}\in \Diff^{2s}_{\cF}(M,\Omega^{1/2}).$$ 
By Proposition \ref{prop:Hormander_symbol}, $D$ satisfies the strong $*$-Rockland condition. Here $L$ denotes the Lie derivatives as in Section \ref{sec:diff_multi}.
\end{ex}
The following theorem if $P$ is a differential operator is due to Rothschild \cite{RotschildSinglePaper}.
\begin{thm}\label{thm:strong_rock_par}
If $P\in \Psi^k(\cF^\bullet)$ satisfies the strong $*$-Rockland condition then there exists $Q\in \Psi^{-k}(\cF^\bullet)$ such that $Q\star P-\Id$ and $P\star Q-\Id$ belong to $\Psi^{-1}(\cF^\bullet)$.
\end{thm}
\begin{proof}
Let $u\in \cE^{\prime k}(\mathbb{U})$ be a global lift of $P$ that satisfies the strong $*$-Rockland condition. Let $v_0\in \cE^{\prime -k}(\mathfrak{g}\times M)$ be a parametrix for $u_0$ as in Theorem \ref{Parametrix_symbols} and $v\in \cE^{\prime -k}(\mathbb{U})$ an extension, which exists by Proposition \ref{prop:lifting_prop}. Proposition \ref{rmk:symbol_convolution} implies that $\ev_1(v) \star \ev_1(u)-\Id$ can be written as $\ev_1(u'')$ with the symbol of $u''$ vanishing on $\left(\mathfrak{g}^*\setminus\{0\}\right)\times M\times \{0\}$. Hence $u''_0$ is smooth. By Proposition \ref{prop:one_less_order}, $\ev_1(u'')\in \Psi^{-1}(\cF^\bullet)$. We have thus constructed $Q=\ev_1(v)$ such that $Q\star P-\Id$ and $P\star Q-\Id$ belong to $\Psi^{-1}(\cF^\bullet)$. 
\end{proof} 

\begin{thm}\label{thm:bounded}
Let $P\in \Psi^k(\cF^\bullet)$.
\begin{enumerate}
\item If $\Re(k)<0$, then $P$ extends to a compact operator $L^2M\to L^2M$.
\item If $\Re(k)=0$, then $P$ extends to a bounded operator $L^2M\to L^2M$.
\end{enumerate} 
\end{thm}
\begin{proof}
\begin{enumerate}
\item Since $M$ is compact, there exists $C>0$ such that $\sup_{x\in M}\dimh(\Gr(\cF)_x)\leq C$. Let $P\in \Psi^k(\cF^\bullet)$, and let $n\in \N$ be big enough such that $2\Re(k)n<-C$. By Proposition \ref{prop:properties_of_pseudo_diff}.d, $(P^*\star P)^n$ extends to a compact operator. Hence $P$ also extends to a compact operator.

\item 
By replacing $P$ with $P^*\star P$, we can suppose that $k=0$. Let $u\in \cE^{\prime 0}(\mathbb{U})$ be a global lift of $P$. Fix $C>\norm{\sigma^0(u_0)}^2_{\Sigma^*(\mathfrak{g}\times M)}$.  Then there exists $v_0\in \cE^{\prime0}(\mathfrak{g}\times M)$ such that $C1-u_0^*\star u_0-v_0^*\star v_0\in C^\infty_c(\mathfrak{g}\times M)$, see proof of Theorem \ref{thm:shortexactseqsymb}.  Using Proposition \ref{prop:lifting_prop} we can extend $v_0$ to $v\in\cE^{\prime 0}(\mathbb{U})$. Put $R=\ev_{1*}(v)$. By Proposition \ref{rmk:symbol_convolution} and Proposition \ref{prop:one_less_order}, we have
\[
  Q:=C\Id- P^*\star P - R^*\star R \in \Psi^{-1}(\cF^\bullet).
\]
It follows, using Part a, that $P$ is bounded with $\|P\|^2 \leq C+\|Q\|$.\qedhere
\end{enumerate}
\end{proof}

\begin{prop}\label{thm:generalised_laplacians}There exists a family of operators $P_k\in \Psi^k(\cF^\bullet)$ for $k\in \C$ such that for all $k,k'\in \C$\begin{enumerate}
\item $P_k$ satisfies the strong $*$-Rockland condition
\item $P_k\star P_{k'}-P_{k+k'}\in \Psi^{k+k'-1}(\cF^\bullet)$.
\item $P_k-P_k^*\in \Psi^{k-1}(\cF^\bullet)$.
\end{enumerate}
\end{prop}
\begin{proof}
One extends the $u_k\in \cE^{\prime k}(\mathfrak{g}\times M)$ from Theorem \ref{thm:christ_Gall} using Proposition \ref{prop:lifting_prop}. By Proposition \ref{prop:one_less_order} and Proposition \ref{rmk:symbol_convolution}, the family $P_k$ satisfies the above properties.
\end{proof}
We fix a choice of the operators $P_k$. Using these operators, one can define Sobolev spaces as follows. Let $\tilde{H}^0(M):=L^2M$. For $k>0$, we define $\tilde{H}^k(M)\subseteq L^2M$ to be the domain of $\overline{P_k}$. We equip it with a Hilbert space structure by identifying it with the graph of $\overline{P_k}$. Note that for a different choice of $P_k$, by Theorem \ref{thm:strong_rock_par} and Theorem \ref{thm:bounded}, one gets the same domain. By Proposition \ref{thm:generalised_laplacians}.b and Theorem \ref{thm:bounded} we have a compact bounded inclusion $$\tilde{H}^k(M)\hookrightarrow \tilde{H}^{l}(M),\quad k>l.$$ For $k<0$, we define $\tilde{H}^k(M)$ to be the dual of $\tilde{H}^{-k}(M)$.  We  thus get a chain of Hilbert spaces $$\cdots \subseteq \tilde{H}^{1}(M)\subseteq \tilde{H}^0(M)\subseteq \tilde{H}^{-1}(M)\subseteq \cdots.$$Equivalently by Proposition \ref{prop:adjoint}, $P_k$ extends continuously to $P_k:\cD'(M,\Omega^{1/2})\to \cD'(M,\Omega^{1/2})$. We can then define $\tilde{H}^k(M)$ to be the subspace of $u\in \cD'(M,\Omega^{1/2})$ such that $P_k(u)\in L^2(M)$. 

We have $$C^{\infty}(M):=\bigcap_{k\in\R}\tilde{H}^k(M),$$ because $\tilde{H}^{Nk}(M)\subseteq H^{k}(M)$ where $k\in \N$ and $H^k(M)$ is the classical Sobolev space.
\begin{prop}Let $k\in \C$, $P\in \Psi^k(\cF^\bullet)$. Then for any $s\in \R$, $P$ extends to a bounded operator $$P:\tilde{H}^{s+\Re(k)}(M)\to \tilde{H}^s(M)$$
\end{prop}
\begin{proof}
This follows easily from Theorem \ref{thm:bounded}.
\end{proof}
\subsection{Principal symbol}\label{sec:princip symb}
Let $x\in M$, and $\pi:\Gr(\cF)_x\to B(L^2\pi)$ be a non-trivial unitary irreducible representation. As in the previous section, we fix a global graded Lie basis $(\lie{g},\natural, \mathbb{U}, M)$. By Remark \ref{rmk:natural_p}, $\natural_x:\mathfrak{g}\to \Gr(\cF)_x$ is a group homomorphism. Hence $\pi\circ  \natural_x$ is a representation of $\mathfrak{g}$. We will implicitly use this inclusion $\widehat{\Gr(\cF)_x}\subseteq \widehat{\mathfrak{g}}$ in the Subsections \ref{sec:princip symb}, \ref{sec:Parametrised}, and \ref{sec:para order 0}. Let $P\in \Psi^k(\cF^\bullet)$. If $u\in \cE^{\prime k}(\mathbb{U})$ is a global lift of $P$, then we define the principal symbol of $P$ by the formula \begin{equation}\label{eqn:dfn principal symbol}
 \sigma^k(P,x,\pi)=\sigma^k(u_0,x,\pi):C^\infty(\pi)\to C^\infty(\pi).
\end{equation} The main difficulty with \eqref{eqn:dfn principal symbol} is to show that the right hand side is independent of the choice of the global lift $u$. This is our main goal in this subsection.
\begin{thm}\label{thm:principal symbol is well-defined}
Let $k\in \C$ and $P\in \Psi^k(\cF^\bullet)$. If $\pi\in \cT^*_x\cF\setminus\{1_{\Gr(\cF)_x}\}$, then the principal symbol $\sigma^k(P,x,\pi)$ is well-defined. Furthermore if $s\in \R$, then there exists $C>0$ (only depends on $s,k$) such that 
\begin{equation}\label{eqn norm symb}
\norm{\sigma^0(P,x,\pi)}_{B(H^{s+\Re(k)}(\pi),H^{s}(\pi))}\leq C\norm{P}_{B(H^{s+\Re(k)}(M),H^s(M))}.
\end{equation}
\end{thm}
Like the proof of Theorem \ref{thm:sybmol_welldefined}, the proof of Theorem \ref{thm:principal symbol is well-defined} is based on letting $P$ act on $C^\infty_c(\aF)$ and then invoking \eqref{eqn:mainineqca}. For every $u\in \cE^{\prime k}(\mathbb{U})$, we define a linear map $\cQ(u):C^\infty_c(\aF)\to C^\infty_c(\aF)$ as follows. Let $a\in C^\infty_c(\aF)$. We define $\cQ(u)a\in L^\infty\aF$ by $(\cQ(u)a)_t=\ev_{t*}(u)\star a_t$, that is the composition $$L^2M\xrightarrow{a_t} C^\infty_c(M,\Omega^{1/2})\xrightarrow{\ev_{t*}(u)} C^\infty_c(M,\Omega^{1/2})\subseteq L^2M,$$ and $(\cQ(u)a)_{p,0}$ by $\ev_{p,0*}(u)\star a_{p,0}$.
\begin{prop}If $a\in C^\infty_c(\aF) $, then $\cQ(u)a\in C^\infty_c(\aF)$
\end{prop}
\begin{proof}
If $a$ is an element of first type, then it is clear that $\cQ(u)a$ is also an element of first type. Therefore by Lemma \ref{lem:qsjidfpiujqsipdf}, we can suppose that $a=\cQ(f)$ for some $f\in C^\infty_c(\mathbb{U},\Omega^{1/2}_{r,s})$. One can then treat $f$ as an element of $\bigcap_{k\in \Z}\cE^{\prime k}(\mathbb{U})$. The $w$ constructed in the proof of Proposition \ref{prop:composition} belongs to $\bigcap_{k\in \Z}\cE^{\prime k}(\mathbb{U})$, hence it is smooth. The proposition follows from the equality $\cQ(w)=\cQ(u)a$.
\end{proof}
We remark that if $u\in C^\infty_c(\mathbb{U},\Omega^{1/2}_{r,s})$, then $\cQ(u)a$ defined above is equal to the product of $\cQ(u)\in C^\infty_c(\aF)$ defined in Section \ref{subsec:c0R} and $a$.
\begin{proof}[Proof of Theorem \ref{thm:principal symbol is well-defined}]
By replacing $P$ with $P_{-s}\circ P\circ P_{-s-k}$ where $P_{\bullet}$ are the operators from Proposition \ref{thm:generalised_laplacians}, we can suppose $s=k=0$. It suffices to prove \eqref{eqn norm symb} with $k=s=0$ and $C=1$. By also replacing $u$ by a $u$ plus a smooth function, we can suppose that $w=P-\ev_{1*}(u)\in C^\infty(M\times M,\Omega^{1/2})$ vanishes on a neighbourhood of the diagonal. Let $v=\alpha_{2*}u-u\in C^\infty_c(\mathbb{U},\Omega^{1/2}_{r,s})$. Hence \begin{align}\label{eqn:qjisdfjimqsjdf}
\alpha_{2^k*}u-u=\sum_{j=0}^{k-1}\alpha_{2^j*}v.
\end{align}
Let $a\in C^\infty_c(\aF)$, $X\in \cF^i$ for $i$ be fixed, $b=\theta_i(X)(a)\in C^\infty_c(\aF)$ where $\theta_i$ is the map from Lemma \ref{lem:proof_sqjdfmokj}. In the proof of the following lemma, we will use the $\R_+^\times$ action on $C^*\aF$, see Remark \ref{rem:Action_on_Caf}.\begin{lemma} The sum $\sum_{j=0}^{\infty}\cQ(\alpha_{2^j*}v)b$ converges in $C^*\aF$. \end{lemma}\begin{proof}
 To see this, one has \begin{align*}
\sum_{j=n}^{m}\cQ(\alpha_{2^j*}v)b=\sum_{j=n}^{m}\alpha_{2^j*}(\cQ(v))b=\sum_{j=n}^{m}2^{-ji}\alpha_{2^j*}\Big(\cQ(v)\theta_i(X)(\alpha_{2^{-j}*}(a))\Big).
\end{align*}

The map $\theta_i(X)$ is an unbounded multiplier of $C^*\aF$, its adjoint denoted $\theta_i^*(X)$ is defined like in \eqref{eqn:theta} as follows. If $c\in L^\infty_\aF$, then $$ (\theta_i^*(X)c)_t=-t^iL_{X} \circ c_t,\quad (\theta_i^*(X)c)_{p,0}=L_{\tilde{X}^L_p}(c_{p,0}),\quad c\in C^\infty_c(\aF),$$ where $\tilde{X}^L_p$ is the left invariant vector field associated to $[X]_p\in \gr(\cF)_p$. One can check that $\theta_i^*(X)(C^\infty_c(\aF))\subseteq C^\infty_c(\aF)$ and that if $c,c'\in C^\infty_c(\aF)$, then $c\theta_i(X)(c')=\theta_i^*(X)(c)c'\in C^\infty_c(\aF)$. Hence $$\sum_{j=n}^{m}\cQ(\alpha_{2^j*}v)b=\sum_{j=n}^{m}2^{-ji}\alpha_{2^j*}\Big(\theta_i^*(X)(\cQ(v))\Big)a.$$Therefore
 \begin{equation*}
   \norm{\sum_{j=n}^{m}\cQ(\alpha_{2^j*}v)b}_{C^*\aF}\leq \norm{a}_{C^*\aF}\norm{\theta_i^*(X)(\cQ(v))}_{C^*\aF}\sum_{j=n}^m2^{-ji}.\qedhere
\end{equation*}\end{proof}
We finally define $w\star b\in C^\infty_c(\aF)$ by $$(w\star b)_t=w\star b_t,\quad (w\star b)_{p,0}=0.$$ Notice that regardless of the type of $b$, $w\star b$ is always an element of first type because $w$ vanishes on a neighbourhood of the diagonal. Let \begin{equation*}\label{eqn:proof symb princip}
 c=w\star b+\cQ(u)b+\sum_{j=0}^{\infty}\cQ(\alpha_{2^j*}v)b\in C^*\aF.
\end{equation*}
Let $t\in \R_+^\times$. Since $v$ is compactly supported, $c_t$ is actually a finite sum and by \eqref{eqn:qjisdfjimqsjdf}, $c_t=P\star b_t$. On the other hand if $p\in M,\pi \in \cT^*_p\cF\setminus\{1_{\Gr(\cF)_p}\}$, then $\pi(c_{p,0})=\sigma^0(u,p,\pi)(\pi(b_{p,0}))$. To see this, notice that $\sigma^0(u,p,\pi)$ was defined using the full symbol of $u$ which is constructed by an identical formula to the sum used in the definition of $c$, see the proof of Proposition \ref{prop:const pseudo from symb}. We remark that since $b=\theta_i(X)(a)$, $\pi(c_{p,0})=\pi(b_{p,0})=0$ if $\pi$ is the trivial representation. By \eqref{eqn:mainineqca}, it follows that \begin{align*}
\sup_{p\in M,\pi \in \cT^*_p\cF\setminus\{1_{\Gr(\cF)_p}\}}\norm{\sigma^{0}(u,p,\pi)(\pi(b_{p,0}))}_{L^2\pi}&=\sup_{p\in M,\pi \in \cT^*_p\cF}\norm{\pi(c_{p,0})}_{L^2\pi}\\&=\limsup_{t\to 0^+}\norm{c_t}_{K(L^2M)}\\&\leq \norm{P}_{B(L^2M)}\limsup_{t\to 0^+}\norm{b_t}_{K(L^2M)}\\&=\norm{P}\sup_{p\in M,\pi \in \cT^*_p\cF\setminus\{\hat{1}_{\Gr(\cF)_p}\}}\norm{b_{p,0}}_{L^2\pi}.
\end{align*}
The result then follows from \cite[Proposition 3.1]{ChrGelGloPol}.
\end{proof}



From \eqref{eqn:algebra_prop_principal_symbols_at_0} and Proposition \ref{prop:symbextSob}, we deduce the following.
\begin{prop}\label{prop:prop of symb}
Let $k,l\in \C$, $P\in \Psi^k(\cF^\bullet),Q\in \Psi^l(\cF^\bullet)$, $x\in M$ and $\pi\in \cT^*\cF_x\setminus\{1_{\Gr(\cF)}\}$.\begin{enumerate}
\item $\sigma^k(P,x,\pi)\sigma^l(Q,x,\pi)= \sigma^{k+l}(P\star Q,x,\pi)$
\item $\sigma^{\bar{k}}(P^*,x,\pi)\subseteq \sigma^k(P,x,\pi)^*$, where the inclusion as unbounded operators.
\item for every $s\in \R$, $\sigma^k(P,x,\pi)$ extends to a bounded operator $H^{s+\Re(k)}(\pi)\to H^{s}(\pi)$.
\end{enumerate} 
\end{prop}

\subsection{Pseudodifferential operators whose principal symbol vanishes}\label{sec:Parametrised}
In Examples \ref{exs:counter_examples}, we gave an example of $D\in \Diff^k_{\cF}(M)$ such that $\sigma^k(D,x,\pi)=0$ for all $x\in M,\pi\in\cT^*_x\cF$ yet $D\notin \Diff^{k-1}_{\cF}(M)$. Nevertheless in this section we will prove the following \begin{thm}\label{thm:vanish_symb}
Let $k\in \C$, $P\in \Psi^k(\cF^\bullet)$. The following are equivalent \begin{enumerate}
\item For all $x\in M$, $\pi\in\cT^*_x\cF\backslash \{0\}$, $\sigma^k(P,x,\pi)=0$.
\item For all (or for some) $s\in \R$ the operator $P:\tilde{H}^{s+\Re(k)}(M)\to\tilde{H}^{s}(M) $ is compact. 
\end{enumerate} 
\end{thm}
By replacing $P$ with $P_{s}\star P\star P_{-s-k}$ where $P_{-s-k}$ and $P_{s}$ are the operators from Proposition \ref{thm:generalised_laplacians}, we can suppose $s=k=0$. The proof of Theorem \ref{thm:vanish_symb} in this case will be based on constructing some $C^*$-algebras which we now do.
\begin{enumerate}
\item  Let $\overline{\Psi^0(\cF^\bullet)}$ denote the closure of $\Psi^0(\cF^\bullet)$ in $B(L^2M)$.
\item Let $L^\infty\cT\cF$ denote the space of all functions $a$ defined for every $x\in M$ and $\pi\in\cT^*_x\cF\setminus\{1_{\Gr(\cF)}\}$ such that $a(x,\pi)\in B(L^2\pi)$ and $$\norm{a}:=\sup_{x\in M,\pi\in \cT^*_x\cF\setminus\{1_{\Gr(\cF)}\}}\norm{a(x,\pi)}<+\infty.$$ Obviously $L^\infty\cT\cF$ is a $C^*$-algebra. We have an obvious map $\Psi^0(\cF^\bullet)\to L^\infty\aF$ which sends $P$ to $\sigma^0(P)$ defined by $\sigma^0(P)(x,\pi)=\sigma^0(P,x,\pi)$. It is well defined by \eqref{eqn norm symb}, and extends to a $C^*$-homomorphism $\overline{\Psi^0(\cF^\bullet)}\to L^\infty\aF$. Its image is denoted by $\Sigma^*\cT\cF$. Theorem \ref{thm:vanish_symb} follows from the following theorem whose proof will be given at the end of this section. \begin{thm}\label{eqn:qjsokdfsdfjkYU} The natural sequence \begin{equation}\label{eqn:short exact sequn pseudo diff}
0\to K(L^2M)\to \overline{\Psi^0(\cF^\bullet)}\xrightarrow{\sigma^0} \Sigma^*\cT\cF\to 0
\end{equation}
is exact, i.e., if $P\in \Psi^0(\cF^\bullet)$, then $$\norm{P}_{B(L^2M)/K(L^2M)}=\sup_{x\in M,\pi\in\cT^*_x\cF\setminus\{1_{\Gr(\cF)}\} }\norm{\sigma^0(P,x,\pi)}.$$
\end{thm}
\item As in the previous two subsections, we fix a global graded Lie basis $(\lie{g},\natural,\mathbb{U},M)$. We define $\Psi^k(\aF)$ to be the subspace of linear maps $C^\infty_c(\aF)\to C^\infty_c(\aF)$ which can be written as the sum of a linear map $b\to ab$ for $a\in C^\infty_c(\aF)$ and a linear map $g\cQ(u)$ for $g\in C^\infty_c(\R_+^\times ),u\in \cE^{\prime k}(\mathbb{U})$.
\begin{prop}\label{prop:BoundedPsi}For any $k,l\in \C$, the following holds.
\begin{enumerate}
\item The space $\Psi^{k}(\aF)$ doesn't depend on the choice of the graded Lie basis $(\mathfrak{g},\natural,\mathbb{U},M)$.
\item If $P\in \Psi^k(\aF)$ and $Q\in \Psi^l(\aF)$, then $PQ\in \Psi^{k+l}(\aF)$
\item If $P\in \Psi^k(\aF)$, then there exists a unique operator $P^*\in \Psi^{\bar{k}}(\aF)$ such that $$P(a)^*b=a^*P^*(b),\quad \forall a,b\in C^\infty_c(\aF).$$
\item If $P\in \Psi^k(\aF)$ and $\Re(k)<0$, then $P$ extends to a compact multiplier $C^*\aF\to C^*\aF$, i.e., there exists a unique element $a\in C^*\aF$ such that $P(b)=ab$ for all $a\in C_c^\infty(\aF)$.
\item If $P\in \Psi^k(\aF)$ and $\Re(k)=0$, then $P$ extends to a bounded operator $C^*\aF\to C^*\aF$.
\end{enumerate}
\end{prop}
\begin{proof}
The proof of a,b,c is almost identical to the proof of Propositions \ref{prop:change_of_basis}, \ref{prop:composition} and \ref{prop:adjoint}. The main difference is that one replaces $\ev_{1*}$ with $\cQ$. The proof of d and e is very similar to the proof of Theorem \ref{thm:bounded}, and will be omitted.
\end{proof}
\item By Proposition \ref{prop:BoundedPsi}, elements of $\Psi^{0}(\aF)$ act by bounded multipliers on $C^*\aF$. We denote by $\overline{\Psi^{0}(\aF)}$ the closure of $\Psi^{0}(\aF)$ inside $M(C^*\aF)$, the $C^*$-algebra of bounded multipliers, see \cite{MR1325694}. Since elements of $\Psi^{0}(\aF)$ act by bounded multipliers on $C^*\aF$, they also act on the quotient $C^*_z\aF$. We denote by $\overline{\Psi^{0}_z(\aF)}$ the closure of $\Psi^{0}(\aF)$ inside $M(C^*_z\aF)$. Clearly there is a quotient map $\overline{\Psi^{0}(\aF)}\to \overline{\Psi^{0}_z(\aF)}$.
\item Let $u\in \cE^{\prime 0}(\mathfrak{g}\times M)$. In Proposition \ref{prop:bounded_prophdslijqf}, $u$ is shown to define a bounded multiplier of $C^*(\mathfrak{g}\times M)$. Since (see proof of Proposition \ref{prop:TopJacbson}) the $C^*$-algebra $C^*\Gr(\cF)$ is a quotient of $C^*(\mathfrak{g}\times M)$, it follows that $u$ can be regarded as a multiplier of  $C^*\Gr(\cF)$. The closure of such multipliers in $M(C^*\Gr(\cF))$ will be denoted by $\overline{\cE^{\prime 0}(\Gr(\cF))}$. The same construction can also be applied to $C^*\cT\cF$ which is a quotient of $C^*\Gr(\cF)$. Hence we also obtain $\overline{\cE^{\prime 0}_z(\Gr(\cF))}\subseteq M(C^*\cT\cF)$. It follows from Theorem \ref{thm:shortexactseqsymb} that if $u\in  \cE^{\prime 0}(\mathfrak{g}\times M)$ such that $\sigma^0(u,x,\pi)=0$ for all $x\in M$, $\pi\in \cT^*_x\cF$, then $u$ seen as a bounded multiplier of $C^*\cT\cF$ is compact, i.e., belongs to $C^*\cT\cF$.
\end{enumerate}
\begin{prop}
The $C^*$-algebra $\overline{\Psi^{0}(\aF)}$ and $\overline{\Psi^{0}_z(\aF)}$ are  $C_0(\R_+)$-$C^*$-algebras which lie in the short exact sequence
\begin{align}
 0\to \overline{\Psi^0(\cF^\bullet)}\otimes C_0(\R^\times_+)\to \overline{\Psi^{0}(\aF)}\to \overline{\cE^{\prime 0}(\Gr(\cF))}\to 0\label{eqn:B}\\
  0\to \overline{\Psi^0(\cF^\bullet)}\otimes C_0(\R^\times_+)\to \overline{\Psi^{0}_z(\aF)}\to \overline{\cE^{\prime 0}_z(\Gr(\cF))}\to 0.\label{eqn:B'}
\end{align}
\end{prop}
\begin{proof}
The $C^*$-algebra $\overline{\Psi^{0}(\aF)}$ is a $C_0(\R_+)$-$C^*$-algebra by construction. It is clear that the non-zero fibers are $\overline{\Psi^0(\cF^\bullet)}$. It is also clear that the morphism $C^*\aF \to C^*\Gr(\cF)$ from \eqref{eqn:exact_seq_C*-alg} gives a $C^*$-homomorphism 
$$\overline{\Psi^{0}(\aF)} \to M(C^*\Gr(\cF)),$$ 
whose image is $\overline{\cE^{\prime 0}(\Gr(\cF))}$. Furthermore this map vanishes on $\overline{\Psi^0(\cF^\bullet)}\otimes C_0(\R^\times_+)$. 
Let $$\phi:\frac{\overline{\Psi^{0}(\aF)}}{\overline{\Psi^0(\cF^\bullet)}\otimes C_0(\R^\times_+)}\to \overline{\cE^{\prime 0}(\Gr(\cF))}$$ be the resulting $C^*$-morphism. We will show that it is an isomorphism by constructing an inverse. 

Let $u_0\in \cE^{\prime 0}(\mathfrak{g}\times M)$. We can extend $u_0$ to an element $u$ of $\cE^{\prime 0}(\mathbb{U})$ and hence to an element $\cQ(u)$ of $\overline{\Psi^{0}(\aF)}$. Two such extensions differ by an element of $\overline{\Psi^0(\cF^\bullet)}\otimes C_0(\R^\times_+)$ so get a well-defined map 
$$\psi:\cE^{\prime 0}(\mathfrak{g}\times M)\to \frac{\overline{\Psi^{0}(\aF)}}{\overline{\Psi^0(\cF^\bullet)}\otimes C_0(\R^\times_+)}.$$ 
It is a consequence of the proof of Theorem \ref{thm:bounded}.b that the norm of the image of $u$ in $$\frac{\overline{\Psi^{0}(\aF)}}{\overline{\Psi^0(\cF^\bullet)}\otimes C_0(\R^\times_+)}$$ depends only on its restriction to $t=0$, and it follows that $\psi$ is a bounded map, so extends to the closure $\overline{\cE^{\prime 0}(\mathfrak{g}\times M)}$. We claim that $\psi$ descends to a map 
 $$\psi:\overline{\cE^{\prime 0}(\Gr(\cF))}\to \frac{\overline{\Psi^{0}(\aF)}}{\overline{\Psi^0(\cF^\bullet)}\otimes C_0(\R^\times_+)}.$$
 To see this, let us first observe that the $C^*$-algebra $$\frac{\overline{\Psi^{0}(\aF)}}{\overline{\Psi^0(\cF^\bullet)}\otimes C_0(\R^\times_+)}$$ is fibered over $M$. 
This follows from the proof of Proposition \ref{prop:properties_of_pseudo_diff}.b. Moreover, the map $\psi$ is clearly $C(M)$-linear. Hence one only needs to show that the map $\psi_x$ between fibers descends to a map 
 $$\overline{\cE^{\prime 0}(\Gr(\cF)_x)}\to \left(\frac{\overline{\Psi^{0}(\aF)}}{\overline{\Psi^0(\cF^\bullet)}\otimes C_0(\R^\times_+)}\right)_x$$
 for every $x\in M$. This can be easily achieved, since we can define an analogue of the map $\psi$ using a minimal graded basis at $x$, and the resulting map provides the factorization of $\psi_x$ through the quotient $\cE^{\prime 0}(\mathfrak{g}\times \{x\}) \to \cE^{\prime 0}(\Gr(\cF)_x)$.  
 
 Finally the fact that $\psi$ and $\phi$ are inverses of each other can be readily checked on the image of elements $u\in \cE^{\prime 0}(\mathfrak{g}\times M)$, which are dense in both sides. This proves that \eqref{eqn:B} is exact. Exactness of \eqref{eqn:B'} follows immediately because all the terms were defined by their action on \eqref{eqn:qjjsdfljqsdfjll} which is a quotient of \eqref{eqn:exact_seq_C*-alg}.
\end{proof}
\begin{proof}[Proof of Theorem \ref{eqn:qjsokdfsdfjkYU}]
By \eqref{eqn norm symb}, we have $$ \sup \big\{\norm{\sigma^0(P,x,\pi)}_{B(L^2\pi)} : x\in M,\ \pi\in \cT^*\cF_x\setminus\{1_{\Gr(\cF)_x}\} \big\}\leq \norm{P}_{B(L^2M)}.
$$Since the left-hand side doesn't change if we replace $P$ with $P+Q$ for $Q\in \Psi^{-1}(\cF^\bullet)$, it follows that $$ \sup \big\{\norm{\sigma^0(P,x,\pi)}_{B(L^2\pi)} : x\in M,\ \pi\in \cT^*\cF_x\setminus\{1_{\Gr(\cF)_x}\} \big\}\leq \norm{P}_{B(L^2M)/K(L^2M)}.
$$ For the other inequality, let $u$ a global lift. Without loss of generality we can suppose that $P=\ev_1(u)$. Since $u_1-u_t\in C^\infty_c(\cF)$, it follows that 
$$\norm{P}_{B(L^2M)/K(L^2M)} 
\leq \norm{u_t}_{B(L^2M)},\quad \forall t>0.$$ 
By  applying \cite[Proposition C.10.a on Page 357]{MR2288954} to $\cQ(u)\in  \overline{\Psi^{0}_z(\aF)}$, we get
$$\limsup_{t\to 0^+} \norm{u_t}_{B(L^2M)} \leq 
\sup \{\norm{\pi(u_0)} : x\in M,\pi \in \cT^*\cF_x \}.$$ Hence $$\norm{P}_{B(L^2M)/K(L^2M)} \leq \sup \{\norm{\pi(u_0)} : x\in M,\pi \in \cT^*\cF_x \}.$$
Since the right-hand side doesn't change if one adds to $P$ an element of the form $\ev_1(v)$ with $v\in \cE^{\prime-1}(\mathfrak{g}\times M)$, it follows from Theorem \ref{thm:shortexactseqsymb} that 
\begin{equation*} 
\norm{P}_{\overline{\Psi^0(\cF^\bullet)}/K(L^2M)}\leq \sup \big\{\norm{\sigma^0(P,x,\pi)} : x\in M,\ \pi\in \cT^*\cF_x\setminus\{1_{\Gr(\cF)_x}\} \big\}.\qedhere 
\end{equation*}
\end{proof}
\subsection{Parametrix and Proof of Theorem \ref{mainthmintro2} and Theorem \ref{mainthmintro3} when $M$ is compact}\label{sec:para order 0}
 We introduce the following larger class of operators.
 \begin{definition}\label{dfn:psi tilde}Let $k\in \C$. We define $\tilde{\Psi}^k(\cF^\bullet)$ to be the set of all linear maps $P:C^\infty(M,\Omega^{1/2})\to C^\infty(M,\Omega^{1/2})$ such that, for every $s\in\R$, $P$ extends to a bounded operator 
$\tilde{H}^{s+\Re(k)}(M) \to \tilde{H}^{s}(M)$ that lies in the closure of $\Psi^k(\cF^\bullet) \subseteq \cL(\tilde{H}^{s+\Re(k)}(M) , \tilde{H}^{s}(M))$.  
\end{definition}

In the terminology of Higson \cite{Higson:local_index}, operators in $\tilde{\Psi^k}(\cF^\bullet)$ are of \emph{analytic order $k$}.
It is clear that $\Psi^k(\cF^\bullet)\subseteq \tilde{\Psi}^k(\cF^\bullet)$ and that for $k,l\in \C$ one has $$\tilde{\Psi}^k(\cF^\bullet)\tilde{\Psi}^l(\cF^\bullet)\subseteq \tilde{\Psi}^{k+l}(\cF^\bullet),\quad \tilde{\Psi}^k(\cF^\bullet)^*\subseteq \tilde{\Psi}^{\bar{k}}(\cF^\bullet).$$By \eqref{eqn norm symb}, if $Q\in \tilde{\Psi}^{k}(\cF^\bullet)$, $s\in \R$, $x\in M$, $\pi \in \cT^*_x\cF$, then $$\sigma^k(Q,x,\pi):H^{s+\Re(k)}(\pi)\to H^{s}(\pi)$$ is well defined. Hence $$\sigma^k(Q,x,\pi):C^\infty(\pi)\to C^\infty(\pi)$$ is well defined.
\begin{thm}\label{Rocklandthm}Let $k\in\C$ and $P\in \Psi^k(\cF^\bullet)$. The following are equivalent:
\begin{enumerate}
\item For all $x\in M$ and $\pi \in \cT^*\cF_x\setminus\{1_{\Gr(\cF)}\}$, $\sigma^k(P,x,\pi):C^\infty(\pi)\to C^\infty(\pi)$ is injective.
\item For all $x\in M$ and $\pi \in \cT^*\cF_x\setminus\{1_{\Gr(\cF)}\}$ and $s\in \R$, the bounded extension $\sigma^k(P,x,\pi):H^{s+\Re(k)}(\pi)\to H^s(\pi)$ is left invertible.
\item For all $s\in \R$ the bounded extension $P:\tilde{H}^{s+\Re(k)}(M)\to \tilde{H}^{s}(M)$ is left invertible modulo compact operators.
\item For some $s\in \R$ the bounded extension $P:\tilde{H}^{s+\Re(k)}(M)\to \tilde{H}^{s}(M)$ is left invertible modulo compact operators.
\item There exists $s\in \R$ such that for any distribution $u$ on $M$, if  $Pu\in \tilde{H}^{s}(M)$ then $u\in \tilde{H}^{s+\Re(k)}(M)$.
\item For all $s\in \R$ and for any distribution $u$ on $M$, if  $Pu\in \tilde{H}^{s}(M)$ then $u\in \tilde{H}^{s+\Re(k)}(M)$.
\item There exists $Q\in \widetilde{\Psi}^{-k}(\cF^\bullet)$ such that $Q\star P-Id\in C^\infty(M\times M,\Omega^{1/2})$.
\item There exists $Q\in \Psi^{-k}(\cF^\bullet)$ such that for all $x\in M$ and $\pi \in \cT^*\cF_x\setminus\{1_{\Gr(\cF)}\}$, $\sigma^{0}(Q\star  P,x,\pi)=\mathrm{Id}$.
\end{enumerate}
Moreover if $k=0$, then the previous statements are also equivalent to the following
\begin{enumerate}  
\setcounter{enumi}{9}
\item The element $\sigma^0(P)\in \Sigma^*\cT\cF$ is left invertible.
\end{enumerate}
If $P$ satisfies any of the above, then we say that $P$ is \textbf{maximally hypoelliptic}.
\end{thm}
\begin{proof}
By replacing $P$ with $P_{-k}\star P$ where $P_{-k}$ is the operator in Proposition \ref{thm:generalised_laplacians}, without loss of generality we can suppose $k=0$. We will prove the following cycle \begin{align*}
a\implies h\implies g\implies f\implies e\implies d\implies c\implies j \implies b\implies a
\end{align*}
 \begin{enumerate}
 \item[$a\implies h$.]  Let $P\in \Psi^0(\cF^\bullet)$ which satisfies $a$, $(\mathfrak{g},\natural,\mathbb{U},M)$ a global graded Lie basis. We claim that there exists
$P'\in \Psi^0(\cF^\bullet)$ such that $\sigma^0(P^*\star P)=\sigma^0(P')$ and $P'$ satisfies the strong $*$-Rockland condition. To see this, let $u\in \cE^{\prime 0}(\mathbb{U})$  be a lift of $P$ and $v$ a lift of $P^*\star P$ such that $v_0=u_0^*\star u_0$, which exists by Remark \ref{rmk:symbol_convolution}. In \cite{Rocklandcondt}, Hebisch shows that there exists an element $w_0\in \cE^{\prime 0}(\mathfrak{g}\times M)$ such that \begin{itemize}
\item for every $x\in M$ and $\pi \in \cT^*\cF_x\setminus\{1_{\Gr(\cF)_x}\}$ we have $\sigma^0(w_0,x,\pi)=0$,
\item for every $x\in M$ and $\pi \in \hat{\mathfrak{g}}\setminus \cT^*\cF$, $\sigma^0(w_0,x,\pi)$ is injective.
\end{itemize}
The proof in \cite{Rocklandcondt} is for a single group and not a family $\mathfrak{g}\times M$, but one can easily modify the argument given there to handle the family case. Using Proposition \ref{prop:lifting_prop}, one extends $w_0$ to an element $w\in \cE^{\prime 0}(\mathbb{U})$. It is clear that $P'=\ev_1(w)^*\star \ev_1(w)+P^*\star P$ satisfies the strong $*$-Rockland condition, and $\sigma^0(P^*\star P)=\sigma^0(P')$. The operator $P'$ admits a left parametrix $Q'\in \Psi^{0}(\cF^\bullet)$ modulo $\Psi^{-1}(\cF^\bullet)$ by Theorem \ref{thm:strong_rock_par}. We can take $Q=Q'\star P^*$.
\item[$h\implies g$.] We can suppose $\sigma^0(P)=\mathrm{Id}$. By Theorem \ref{thm:vanish_symb}, $P:\tilde{H}^s(M)\to \tilde{H}^s(M)$ is equal to $\mathrm{Id}$ plus a compact operator for every $s$. Hence $P:\tilde{H}^s(M)\to \tilde{H}^s(M)$ is Fredholm with Fredholm index equal to $0$. Since smoothing operators are dense in $\cK(L^2M)$, there exists a smoothing operator $R\in C^\infty(M\times M,\Omega^{1/2})$ such that $P+R:L^2M\to L^2M$ is invertible. Now consider the operator $P+R:\tilde{H}^s(M)\to \tilde{H}^s(M)$. If $s>0$, then since $\tilde{H}^s(M)\subseteq L^2M$, and $P+R$ is injective on $L^2M$, it follows that $P+R:\tilde{H}^s(M)\to \tilde{H}^s(M)$ is also injective. By the vanishing of the index, $P+R:\tilde{H}^s(M)\to \tilde{H}^s(M)$ is surjective. For $s<0$, we argue similarly. Since $L^2M\subseteq \tilde{H}^s(M) $, and $P+R$ is surjective on $L^2M$, it follows that $P+R:\tilde{H}^s(M)\to \tilde{H}^s(M)$ has dense image. Being a Fredholm operator with vanishing index, we deduce that $P+R:\tilde{H}^s(M)\to \tilde{H}^s(M)$ is bijective. Let $Q=(P+R)^{-1}$. It follows from the above that $Q:C^\infty(M,\Omega^{1/2})\to C^\infty(M,\Omega^{1/2})$ and it extends to $\tilde{H}^s(M)\to  \tilde{H}^s(M)$. Since $C^*$-subalgebras are closed under holomorphic calculus, $Q\in\overline{\Psi^0(\cF^\bullet)}\subseteq B(\tilde{H}^s(M))$ for every $s$, and we get $Q\in \tilde{\Psi}^0(\cF^\bullet)$. Furthermore $QP=\mathrm{Id}-QR$. Since $R$ is a smoothing operator, it follows that $QR$ is also a smoothing operator.
 \item[$g\implies f$.]This is trivial.
  \item[$f\implies e$.]This is trivial.
 \item[$e\implies d$.] Suppose $P$ satisfies $e$ with $s\in \R$. Then we define the Hilbert space $H=\{u\in \tilde{H}^{s-1}(M):Pu\in \tilde{H}^s(M)\}$ with the norm $\norm{u}_{H}=\norm{u}^2_{\tilde{H}^{s-1}(M)}+\norm{Pu}^2_{\tilde{H}^s(M)}$. By $e$, $H=\tilde{H}^s(M)$ as a vector space. Furthermore the inclusion $\tilde{H}^s(M)\hookrightarrow H$ is continuous. By the open mapping theorem, there exists $C>0$ such that $$ \norm{u}^2_{\tilde{H}^s(M)}\leq C\norm{Pu}_{\tilde{H}^s(M)}^2+\norm{u}^2_{\tilde{H}^{s-1}(M)}.$$ By Theorem \ref{thm:bounded}, $P:\tilde{H}^s(M)\to \tilde{H}^s(M)$ is left invertible modulo compact operators. 
 \item[$d\implies c$.]Let $P'=P_{s}\star P\star P_{-s}$. Theorem \ref{eqn:qjsokdfsdfjkYU} implies that $\sigma^0(P')$ satisfies $a$. Hence $h$ holds for $P'$. By Theorem \ref{thm:vanish_symb}, $P$ satisfies $c$. 
 \item[$c\implies j$.]This follows from Theorem \ref{eqn:qjsokdfsdfjkYU}.
  \item[$j\implies b$.] The implications $j\implies a$ and $h\implies b$ are trivial. We have proved $a\implies h$.
 \item[$b\implies a$.]This is trivial.\qedhere
  \end{enumerate}\end{proof}
By applying Theorem \ref{Rocklandthm} to $P$ and $P^*$, one deduces the following
 \begin{thm}\label{RocklandthmBi}Let $k\in\C$ and $P\in \Psi^k(\cF^\bullet)$. The following are equivalent:
\begin{enumerate}
\item For all $x\in M$ and $\pi \in \cT^*\cF_x\setminus\{1_{\Gr(\cF)}\}$, $\sigma^k(P,x,\pi):C^\infty(\pi)\to C^\infty(\pi)$ is bijective.
\item For all $x\in M$ and $\pi \in \cT^*\cF_x\setminus\{1_{\Gr(\cF)}\}$, $\sigma^k(P,x,\pi)$ and $\sigma^{\bar{k}}(P^*,x,\pi)$ are injective on $C^\infty(\pi)$.
\item For all $x\in M$ and $\pi \in \cT^*\cF_x\setminus\{1_{\Gr(\cF)}\}$ and $s\in \R$, the bounded extension $\sigma^k(P,x,\pi):H^{s+\Re(k)}(\pi)\to H^s(\pi)$ is invertible.
\item For all (or for some) $s\in \R$ the bounded extension $P:\tilde{H}^{s+\Re(k)}(M)\to \tilde{H}^{s}(M)$ is Fredholm.
\item There exists $Q\in \widetilde{\Psi}^{-k}(\cF^\bullet)$ such that $Q\star P-\mathrm{Id}$ and $P\star Q-\mathrm{Id}$ belong to $C^\infty(M\times M,\Omega^{1/2})$.
\item There exists $Q\in \Psi^{-k}(\cF^\bullet)$ such that for all $x\in M$ and $\pi \in \cT^*\cF_x\setminus\{1_{\Gr(\cF)}\}$, $\sigma^{0}(Q\star P,x,\pi)=\sigma^{0}(P\star  Q,x,\pi)=\mathrm{Id}$.
\end{enumerate}
Moreover if $k=0$, then the previous statements are also equivalent to the following
\begin{enumerate}  
\setcounter{enumi}{9}
\item The element $\sigma^0(P)\in \Sigma^*\cT\cF$ is invertible.
\end{enumerate}
If $P$ satisfies any of the above, then we say that $P$ is \textbf{$*$-maximally hypoelliptic}.
\end{thm}
\begin{cor}Let $k\in\C$ with $\Re(k)>0$ and $P\in \Psi^k(\cF^\bullet)$ a $*$-maximally hypoelliptic. \begin{enumerate}
\item The maximal and minimal domain agree and are equal to $\tilde{H}^{\Re(k)}(M)$. Hence if $P$ is symmetric it is essentially self-adjoint.
\item For any $x\in M$, $\xi \in \cT^*\cF\backslash \{0\}$, the minimal and maximal domain of $\sigma^k(P,x,\pi)$ are equal to $H^k(\pi)$.
\end{enumerate} 
\end{cor}
\begin{remark}\label{rem:diag_Schro} Let $k\in \R_+^\times$ and $P\in \Psi^k(\cF^\bullet)$ a symmetric $*$-maximally hypoelliptic operator. By Theorem \ref{RocklandthmBi} and Theorem \ref{thm:Bounded_symbol}.b, if $x\in M$ and $\pi \in \cT^*\cF_x\setminus\{1_{\Gr(\cF)}\}$, then the closure of $\sigma^k(P,x,\pi)$ acting on $L^2\pi$ is selfadjoint and has compact resolvent. Hence it is diagonalizable with eigenvalues converging in absolute value to $+\infty$. The eigenvectors also belong to $C^\infty(\pi)$ by \eqref{eqn:inter}.
\end{remark}
\begin{thm}\label{prop:local}Let $k\in \C$, $P\in \Psi^k(\cF^\bullet)$. Then the set of all $x\in M$ such that $\sigma^k(P,x,\pi):C^\infty(\pi)\to C^\infty(\pi)$ is injective for every $\pi \in \cT^*_x\cF\setminus\{1_{\Gr(\cF)}\}$ is an open subset of $M$.
\end{thm}
\begin{proof}
By replacing $P$ with $P_{-k}\star P$, we can suppose that $P\in \Psi^0(\cF^\bullet)$. Let $(\mathfrak{g},\natural,\mathbb{U},M)$ be a global graded Lie basis, and suppose $x\in M$ is such that $\sigma^k(P,x,\pi)$ is injective on smooth vectors for every $\pi \in \cT^*\cF_x\setminus\{1_{\Gr(\cF)_x}\}$. As in the proof of $a\implies h$ in Theorem \ref{Rocklandthm}, we can find $P'\in \Psi^0(\cF^\bullet)$ such that $\sigma^0(P^*\star P)=\sigma^0(P')$ and $P'$ has a global lift $u\in \cE^{\prime 0}_s(\mathfrak{g}\times M)$ which has the property that $\sigma^0(u,x,\pi)$ is injective on smooth vectors for all $\pi \in \hat{\mathfrak{g}}\setminus\{1_{\mathfrak{g}}\}$. Then by \cite[Theorem 2.5.d]{ChrGelGloPol}, one deduces that for a neighbourhood $V$ of $x$, for every $y\in V$ and every $\pi \in \cT^*\cF_y\setminus\{1_{\Gr(\cF)_y}\}$, $\sigma^0(P',y,\pi)$ is injective on smooth vectors. The result follows.
\end{proof}

\subsection{Proof of Theorem \ref{mainthmintro2} and Theorem \ref{mainthmintro3} in general}\label{sec:noncompact}
In this section, we no longer suppose that $M$ is compact. Let $m=\dim(M)$, $x\in M$ and $U,V\subseteq M$ open subsets such that $x\in V\subseteq \bar{V}\subseteq U$ and $U$ is a chart of $x$ diffeomorphic to the unit ball in $\R^{m}$.  Let $f\in C^\infty_c(M)$ be a positive function such that $f=1$ on $V$ and $\supp(f)\subseteq U$.  A simple computation shows that
 \begin{equation}\label{eq:local_foliation}
0\subseteq f\cF^1\subseteq \cdots \subseteq f\cF^{N-1}\subseteq \cF^N=\cX_c(M)
\end{equation} 
is still a filtered foliation on $M$. Consider  $S^{m}$, the $m$-dimensional sphere considered as a $1$-point compactification of $\R^m$. We define a filtered foliation $\cG^\bullet$ on $S^m$ to be the push-forward of \eqref{eq:local_foliation}, except for $\cG^N$ which we declare to be $\cX(S^{m})$.  

Let $P\in \Psi^k(\cF^\bullet)$ be a pseudodifferential operator on $M$, $f'\in C^\infty_c(M)$ with $f'=1$ on a neighbourhood of $x$ and $\supp(f')\subseteq V$.
 It is straightforward to see that  $f'P^*\star Pf'\in \Psi^{k}(\cG^\bullet)$, where we consider $V\times V\subseteq S^m\times S^m$, and use the fact that the support of $f'P^*\star Pf'$ is subset of $V\times V$. Now let $g\in C^\infty(S^m)$ be any smooth function such that $S^m=\supp(g)\cup\supp(f')$ and $g=0$ on a neighbourhood of $f^{\prime-1}([\frac{1}{2},+\infty[)$. We consider the operator $$Q=f'P^*\star Pf'+gP_{\frac{k}{2}}^*\star P_{\frac{k}{2}}g\in \Psi^{k}(\cG^\bullet),$$ where $P_{\frac{k}{2}}$ is obtained from Proposition \ref{thm:generalised_laplacians} applied to the filtration $\cG^\bullet$. Theorems \ref{mainthmintro2} and Theorem \ref{mainthmintro3} for $P$ easily follow from Theorem \ref{Rocklandthm} and Theorem \ref{prop:local} applied to $Q$.
\begin{appendix}
\section{Proofs of Theorems \ref{main_tech_thm} and \ref{second_main_tech_thm} }\label{appendixA}

\subsection{Baker-Campbell-Hausdorff formula for flows of vector fields}\label{sect: appendix 1}
In this section, we give an analytic interpretation of the BCH formula \eqref{eqn:BCHintro} for the Lie algebra of vector fields on a manifold $M$. Let $ t\cX_c(M)[[t]]$ be the Lie algebra of formal power series with coefficients in $\cX_c(M)$ and constant term $0$. For $\bfX, \bfY \in  t\cX_c(M)[[t]]$, let
\[
  \BCH(\bfX,\bfY) = \bfX + \bfY - \frac{1}{2}[\bfX,\bfY] + \frac{1}{12}[\bfX,[\bfX,\bfY]]-\frac{1}{12}[\bfY,[\bfY,\bfX]]+ \cdots\in t\cX_c(M)[[t]]
\]
which is well-defined because $\bfX,\bfY$ have no constant term. If $\bfX = \sum_{i=1}^\infty t^iX_i \in  t\cX_c(M)[[t]]$, then we write $\bfX_n(t) = \sum_{i=1}^n t^iX_i$ for the truncation of $\bfX$ to order $n$, where now this can be understood concretely as a vector field on $M$ with coefficients depending polynomially on $t$.  We remark that when we talk of the time-one flow by such a vector field, $x \mapsto \exp(\bfX_n(t))\cdot x$, we mean the flow for a fixed but arbitrary $t$, not the time-dependent flow.
\begin{thm}
  \label{thm:BCH}
  For any $\bfX,\bfY\in t\cX_c(M)[[t]]$, $x\in M$, $n\in \NN$, we have 
 \begin{equation}\label{eqn BCH}
\exp(\bfX_n(t))\cdot\left(\exp(\bfY_n(t)) \cdot x\right) = \exp(\BCH(\bfX,\bfY)_n(t)) \cdot x + o(t^n)
   \end{equation}
That is, the two sides agree to order $n$ as functions of $t$ with uniform bounds in $x$ as it varies in a compact set.
\end{thm}
The previous theorem appears implicitly in Hörmander's work, see \cite[proof of Lemma 4.5]{Hormander:SoS}. It is proved in \cite[Proposition 4.3]{SteinWaingerNagel}.
\begin{definition}
Let $\bfX,\bfY\in  t\cX_c(M)[[t]]$, $x\in M$, $n\in \NN$. We write $\bfX\sim_{n,x}\bfY$ if $$\exp(\bfX_n)x-\exp(\bfY_n)x=o(t^n).$$
\end{definition}
Clearly $\sim_{n,x}$ is an equivalence relation.
 \begin{prop}\label{prop: appendix 1}Let $\bfX,\bfY,\bfZ\in  t\cX_c(M)[[t]]$, $x\in M$, $n\in \NN$. If $\bfX\sim_{n,x}\bfY$, then $$\BCH(\bfZ,\bfX)\sim_{n,x}\BCH(\bfZ,\bfY).$$
\end{prop}
\begin{proof}
This follows directly from Theorem \ref{thm:BCH} and the smooth dependence of the flow $\exp(Z)$ on the initial conditions.
\end{proof}

Finally we need the following lemma.
\begin{lemma}\label{lem:k_appendix} Let $(V,\natural,\mathbb{U},U)$ be a graded basis. Then there exists a smooth function $k:\dom(k)\subseteq V\times U\times U\to V$ defined on an open set $\dom(k)$ such that 
\begin{enumerate}
\item for any $x\in U$, $(0,x,x)\in \dom(k)$.
\item for any $(v,y,x)\in \dom(k)$, $\exp(\natural(k(v,y,x)))\cdot x=y$.
\item for any $x\in U,v\in V$ such that $(v,\exp(\natural(v))\cdot x,x)\in \dom(k)$, one has $k(v,\exp(\natural(v))\cdot x,x)=v$.
\end{enumerate}
\end{lemma}
\begin{proof}
The map $$\phi:V\times U\to TU,\quad \phi(v,x)=\natural(v)(x)$$is a bundle morphism over $U$ between the trivial vector bundle $V\times U$ and the tangent bundle $TU$. It is surjective by Condition (ii) of Definition \ref{dfn:graded basis}. Let $p:V\times U\to  \ker(\phi)\subseteq V\times U$ be a smooth projection onto the kernel of $\phi$. We view $\ker(\phi)$ as a manifold of dimension $\dim(V)$. We define the map $$\psi:V\times U\to \ker(\phi)\times U,\quad \psi(v,x)=(p(v,x),\exp(\natural(v))\cdot x).$$The map $\psi$ is a smooth map between manifolds of equal dimension. Its differential $d\psi$ is injective at $(0,x)$ for any $x\in U$. Hence there exists an open neighbourhood $W$ of $\{0\}\times U$ such that $\psi:W\to \psi(W)$ is a diffeomorphism. Let \begin{align*}
\dom(k)=\{(v,y,x)\in V\times U\times U:(v,x)\in W\ \&\ (p(v,x),y)\in \psi(W) \}.
\end{align*}
We define $k(v,y,x)$ by $$\psi^{-1}(p(v,x),y)=(k(v,y,x),x).$$It is straightforward to check that $k$ has the required properties.
\end{proof}

\subsection{Proof of Theorem \ref{main_tech_thm}}\label{sec:proof of main_tech_thm}
Let $k$ be as in Lemma \ref{lem:k_appendix} applied to the graded Lie basis $(\mathfrak{g},\natural,\mathbb{U},U)$. We will use the notation $$\pi(Y,X,x,t):=\exp(\natural (\alpha_t (Y)))\cdot\Big(\exp(\natural(\alpha_t (X))) \cdot x\Big).$$
We define $\phi:\dom(\phi)\subseteq \mathfrak{g}\times\mathfrak{g}\times U\times \R_+^\times\to \mathbb{U}$ by the formula 
\begin{align}
\phi(Y,X,x,t)=\begin{cases} \left(\alpha_{t^{-1}} \left( k\Big(\BCH(\alpha_t (Y) ,\alpha_t (X)),\pi(Y,X,x,t),x \Big)\right),x,t\right),  &\mathrm{if}\, t> 0\\
(\BCH(Y,X),x,0),  &\mathrm{if}\, t=0,\end{cases}
\label{Eqn:phi_appendix}
\end{align}
where $\BCH$ is the BCH formula in the nilpotent Lie algebra $\mathfrak{g}$. The domain of $\phi$ is the set in which the above formula is valid, i.e., \begin{align*}
 \dom(\phi)=\{(Y,X,x,t)\in \mathfrak{g}\times\mathfrak{g}\times U\times \R_+^\times&:(\BCH(\alpha_t(Y) ,\alpha_t (X)),\pi(Y,X,x,t),x)\in \dom(k)\\
 & \& \left(k\Big(\BCH(\alpha_t (Y) ,\alpha_t (X)),\pi(Y,X,x,t),x \Big),x,1\right)\in \mathbb{U}\},
\end{align*}

Let us show that $\phi$ is smooth. It suffices to show that the map $\psi:\dom(\phi)\to\lie{g}$ given by
$$\psi: (Y,X,x,t)\mapsto \begin{cases}\alpha_{t^{-1}} (k(BCH(\alpha_t( Y),\alpha_t( X)),\pi(Y,X,x,t),x)),& \mathrm{if}\, t>0 \\ BCH(Y,X),& \mathrm{if}\, t= 0\end{cases}$$ is smooth. The map
\[
 (Y,X,x,t) \mapsto \alpha_t( \psi (Y,X,x,t))
\]
is smooth and vanishes at $t=0$.  Since $\alpha_{t^{-1}}$ is given by division by some $t^k$ on each coordinate of $\lie{g}$, and it follows that if we prove that $\psi$ is continuous at $t=0$ then it is automatically smooth.

Let us show continuity. Fix $p\in U$. The function $k$ is smooth, so restricting to a neighbourhood of $(0,p,p) \in \lie{g}\times U\times U$, we have a constant $C>0$ such that 
$$|k(Y,y,x)-k(Y,y',x)|\leq C|y-y'|.$$ Here, the norms represent any choice of a norm on $\lie{g}$ and a chart near $p\in M$. Let $Z=BCH(Y,X)$. Notice that $\alpha_t (Z)=BCH(\alpha_t (Y),\alpha_t (X))$. It now follows that for $t$ small enough,  
\begin{align*}
|k(\alpha_t(Z), \pi(Y,X,x,t),x)-\alpha_t (Z)|&=|k(\alpha_t(Z),\pi(Y,X,x,t),x)-k(\alpha_t(Z),\exp(\natural(\alpha_t(Z)))\cdot x,x)|\\
 &\leq C|\pi(Y,X,x,t)-\exp(\natural(\alpha_t(Z)))\cdot x|.
\end{align*}
Now we consider $\natural(\alpha_t(X))$, $\natural(\alpha_t(Y))$ and $\natural(\alpha_t(Z))$ as elements of $t\cX_c(M)[[t]]$. By \eqref{eqn:bracket graded Lie}, it follows that $$\BCH(\natural(\alpha_t(Y)),\natural(\alpha_t(X)))_N=\natural(\alpha_t(Z))$$Hence Theorem \ref{thm:BCH} implies that $$|\pi(Y,X,x,t)-\exp(\natural(\alpha_t(Z)))\cdot x|=o(t^{N}).$$Therefore $$|\alpha_{t^{-1}}(k(\alpha_t(Z),\pi(Y,X,x,t),x))-Z|=o(1).$$Continuity of $\psi$ follows. Hence $\phi$ is smooth.

It is straightforward to check that $\phi$ satisfies Theorem \ref{main_tech_thm}.b, c, d. For a, it is clear that $\phi$ is a submersion at $(0,0,p,0)$ for any $p\in U$. Since $\phi$ is $\R^\times_+$-equivariant, it follows that $\phi$ is a submersion on an $\R^+$-equivariant neighbourhood of $\{0\}\times \{0\}\times U\times \{0\}$. Restricting to such a neighbourhood, we can ensure that $\phi$ is a submersion. This finishes the proof of Theorem \ref{main_tech_thm}.


\subsection{Proof of Theorem \ref{second_main_tech_thm}}\label{sec: proof of second_main_tech_thm}
In this section, we will prove the following which easily implies Theorem \ref{second_main_tech_thm}. 
\begin{thm}
Let $(V,\natural,\mathbb{U},U), (V',\natural',\mathbb{U}',U')$ be two graded bases with $U=U'$. There exists a smooth map
$$\phi:\dom(\phi)\subseteq\mathbb{U}\to \mathbb{U}'$$ 
defined on an $\R^\times_+$-invariant neighbourhood of $\{0\}\times U\times \{0\}$ such that \begin{enumerate}
\item $\phi$ is $\R^\times_+$-equivariant.
\item the following diagram commutes $$\begin{tikzcd}\dom(\phi)\arrow[d,"\ev_{|\dom(\phi)}"']\arrow[r,"\phi"]&\mathbb{U}'\arrow[dl,"\ev"]\\M\times M\times \R_+
\end{tikzcd}$$
\item for every $x\in U$, the following diagram commutes $$\begin{tikzcd}[column sep=large]V\times \{x\}\times \{0\}\arrow[d,"\ev_{{x,0}}"']\arrow[r,"\phi_{|V\times \{x\}\times \{0\}}"]&V'\times \{x\}\times \{0\}\arrow[dl,"\ev_{x,0}"]\\\Gr(\cF)_x
\end{tikzcd}$$
\item If $V'$ is minimal at $p$, then $\phi$ is a submersion at $(0,p,0)$.
\end{enumerate}
\end{thm}
\begin{lemma}\label{lem:prem second_main_tech_thm}
There exists a smooth function $\psi:V\times U\to V'$, which we denote by $(X,x) \mapsto \psi_x(X)$ for $x\in U$, $X\in V$, with the following properties:
\begin{enumerate}
\item For every $x\in U$, $\psi_x$ is a polynomial map (of degree $\leq N$).
\item There exist polynomials $\psi^i_x:V\to V'$ for $0\leq i\leq N-1$ such that $$\psi_x(\alpha_t(X))=\sum_{i=0}^{N-1}t^i\alpha_t(\psi^i_x(X)),\quad \forall X\in V,t\in \R_+.$$
\item For any $x\in M$, if we regard $\natural(\alpha_t(X))$ and $\natural'(\psi_x(\alpha_t(X)))$ as elements of $ t\cX_c(M)[[t]]$, then \begin{equation}\label{eqn:lem_app_qhsdljfiq}
 \natural(\alpha_t(X))\sim_{N,x}\natural'(\psi_x(\alpha_t(X))).
\end{equation}
\item For any $x\in M$, the following diagram commutes $$\begin{tikzcd}[column sep=large]V\arrow[d,"\natural_x"']\arrow[r,"\psi_x^0"]&V'\arrow[dl,"\natural_{x}'"]\\\Gr(\cF)_x
\end{tikzcd}$$

\end{enumerate} 
\end{lemma}
Let us first give the proof of Theorem \ref{second_main_tech_thm} assuming Lemma \ref{lem:prem second_main_tech_thm}. We will give the proof of Lemma \ref{lem:prem second_main_tech_thm} after this.
\begin{proof}[Proof of Theorem \ref{second_main_tech_thm}.]Let $k$ be as in Lemma \ref{lem:k_appendix} applied to $(V',\natural',\mathbb{U}',U')$.
We define a smooth map $\phi:\dom(\phi)\subseteq \mathbb{U}\to \mathbb{U}'$ by the formula
 \begin{align*}
\phi(X,x,t)& =\begin{cases} \left(\alpha_{t^{-1}}\bigg(k\Big(\psi_x(\alpha_t(X)),\exp(\natural(\alpha_t(X)))\cdot x,x\Big)\bigg),x,t\right)& \text{if } t\neq 0,\\
(\psi^0_x(X),x,0) &\text{if } t=0.
\end{cases}
\end{align*}
The domain of $\phi$ is the set in which the above formula is valid, i.e.,  \begin{align*}
 \dom(\phi)=\{(X,x,t)\in \mathbb{U}&:(\psi_x(\alpha_t( X) ),\exp(\natural(\alpha_t(X)))\cdot x,x)\in \dom(k),\\& \& \Big(k\Big(\psi_x(\alpha_t(X)),\exp(\natural(\alpha_t(X)))\cdot x,x\Big),x,1\Big)\in \mathbb{U}'\}
 \end{align*}
As in the proof Theorem \ref{main_tech_thm}, to show that $\phi$ is smooth it is enough to show continuity at $t=0$. Let $p\in U$ be fixed and $C>0$ be such that 
 $$|k(Y,y,x)-k(Y,y',x)|\leq C|y-y'|$$
for $(Y,y,x)$ and $(Y,y'x')$ in some neighbourhood of $(0,p,p)$.   It follows that for $t$ small enough
\begin{align}\label{eqn:append 1 psi}
&\lvert k\Big(\psi_x(\alpha_t(X)), \exp(\natural(\alpha_t(X)))\cdot x,x\Big)-\psi_x(\alpha_t(X))\rvert \nonumber\\
=&|k\Big(\psi_x(\alpha_t(X)),\exp(\natural(\alpha_t(X)))\cdot x,x\Big)-k\Big(\psi_x(\alpha_t(X)),\exp(\natural'(\psi_x(\alpha_t(X))))\cdot x,x\Big)|\nonumber\\
\leq &C|\exp(\natural(\alpha_t(X)))\cdot x-\exp(\natural'(\psi_x(\alpha_t(X))))\cdot x|=o(t^N),
 \end{align}
 where in the last inequality we used \eqref{eqn:lem_app_qhsdljfiq}. Lemma \ref{lem:prem second_main_tech_thm}.b implies that \begin{equation}\label{eqn:append 2 psi}
 \lim_{t\to 0^+}\alpha_{t^{-1}}(\psi_x(\alpha_t(X)))=\psi^0_x(X).
 \end{equation}
By \eqref{eqn:append 1 psi} and \eqref{eqn:append 2 psi}, we deduce that $\phi$ is continuous at $t=0$. It is then straightforward to check that $\phi$ satisfies Theorem \ref{second_main_tech_thm}.\end{proof}
\begin{proof}[Proof of Lemma \ref{lem:prem second_main_tech_thm}.] Let $\mathfrak{g}$ be the free nilpotent graded Lie algebra of step $N$ generated by elements of $V\oplus V'$. We can extend $\natural\oplus \natural':V\oplus V'\to \cX_c(M)$ to $\natural:\mathfrak{g}\to \cX_c(M)$ using \eqref{eqn:bracket graded Lie}. It is then enough to construct $\psi:\mathfrak{g}\times U\to V'$ satisfying Lemma \ref{lem:prem second_main_tech_thm}. In  the proof it will be convenient to say that a polynomial map $P:\mathfrak{g}\to V'$ is positive if there exist polynomials maps $P^0,\cdots,P^{N-1}:\mathfrak{g}\to V'$ such that $$P(\alpha_t(X))=\sum_{i=0}^{N-1}t^i\alpha_t(P^i(X)),\quad \forall X\in \mathfrak{g},t\in \R_+.$$We call $P^0$ the homogeneous part of $P$. We start by constructing the linear part of $\psi$.
\begin{lemma}\label{lem:qsjdfqshdjfhq}There exists a smooth map $\phi^1:\mathfrak{g}\times U\to V'$ such that \begin{enumerate}
\item for every $x\in U$, $\phi_x^1:\mathfrak{g}\to V'$ is linear and $\phi_x^1(\mathfrak{g}^i)\subseteq \oplus_{j\leq i}V^{\prime j}$, i.e., $\phi_x^1$ is positive.
\item for every $x\in U$, $X\in \mathfrak{g}$, $\natural(\phi_x^1(X))$ and $\natural(X)$ are vector fields on $M$ which are equal at $x$.
\end{enumerate}
\end{lemma}
\begin{proof}
Fix a basis of $\mathfrak{g}$ and $V'$. We further suppose that each element of the basis belongs to $\mathfrak{g}^n$ or $V^{\prime n}$ for some $n$. Since $\phi^1_x$ is linear for $x\in U$, it is enough to define it on basis element of $\mathfrak{g}$. Let $X\in \mathfrak{g}^n$ be a basis element. By Condition (ii) of Definition \ref{dfn:graded basis}, we can find smooth functions $f_1,\cdots,f_m\in C^\infty(M)$ such that $$\natural(X)=\sum_{i=1}^mf_i \natural(v_i),\quad \text{on } U,$$where $v_i$ are the basis elements of $\oplus_{i\leq n}V^i$. We then define $\phi^1$ by \begin{equation}\label{qsjdkfjqmsjfosdkf}
\phi^1(X,x)=\sum_{i=1}^mf_i(x)v_i.\qedhere
\end{equation}
\end{proof}
Fix $x\in U$. In what follows we will say that an element $X\in \mathfrak{g}$ vanishes at $x$ if $\natural(X)(x)=0$. Since $V'\subseteq \mathfrak{g}$, the vector $-\phi^1_x(X)+X\in \mathfrak{g}$ is well defined. It vanishes at $x$ by Lemma \ref{lem:qsjdfqshdjfhq}. We now construct the quadratic part of $\psi_x$. Consider $$\BCH(-\phi^1_x(X),X)=\Big(-\phi^1_x(X)+X\Big)+\frac{1}{2}[\phi^1_x(X),X]+O(\norm{X}^3).$$The first term vanishes at $x$ but the second doesn't. So let $$\phi^2_x(X)=\frac{1}{2}\phi^1_x([\phi^1_x(X),X]).$$ Since $\phi_x^1$ is linear and positive, $\phi^2_X$ is quadratic and positive. Now consider \begin{align*}
\BCH(-\phi^1_x(X)-\phi^2_x(X),X)&=\Big(-\phi^1_x(X)+X\Big)+\Big(-\phi^2_x(X)+\frac{1}{2}[\phi^1_x(X),X]\Big)\\&+\frac{1}{2}[\phi^2_x(X),X]+\frac{1}{12}[\phi^1_x(X),[\phi^1_x(X),X]]+\frac{1}{12}[X,[X,\phi^1_x(X)]]+O(\norm{X}^4).
\end{align*}We define the cubic part of $\psi_x$ by $$\phi^3_x(X)=\phi_x^1\Big(\frac{1}{2}[\phi^2_x(X),X]+\frac{1}{12}[\phi^1_x(X),[\phi^1_x(X),X]]+\frac{1}{12}[X,[X,\phi^1_x(X)]]\Big)$$We continue this procedure until we have define $\phi^{N}_x$. Then let $$\psi_x(X)=\sum_{i=1}^N\phi^i_x(X).$$ The construction of $\psi$ implies that for any $X\in \mathfrak{g}$, there exists $Y_1,\cdots,Y_N\in \mathfrak{g}$ such that $$\BCH(-\psi_x(\alpha_t(X)),\alpha_t(X))=t^1Y_1+\cdots+t^NY_N+O(t^{N+1}),\quad \forall t\in \R_+$$and $Y_1,\cdots,Y_N$ vanish at $x$. Hence trivially $$\natural(\BCH(-\psi_x(\alpha_t(X)),\alpha_t(X)))\sim_{N,x} 0,$$where we now view $\natural(\BCH(-\phi_x(\alpha_t(X)),\alpha_t(X)))\in t\cX_c(M)[[t]]$. Since $\natural$ satisfies \eqref{eqn:bracket graded Lie}, it follows that $$\natural(\BCH(-\psi_x(\alpha_t(X)),\alpha_t(X)))=\BCH(-\natural(\psi_x(\alpha_t(X))),\natural(\alpha_t(X))).$$By Proposition \ref{prop: appendix 1}, we get that $$\natural(\alpha_t(X))\sim_{N,x}\natural(\psi_x(\alpha_t(X))).$$

It is clear that $\psi_x$ is positive for any $x\in U$ and that the map $\psi$ depends smoothly on $x$. It remains to show Lemma \ref{lem:prem second_main_tech_thm}.d. By \eqref{qsjdkfjqmsjfosdkf}, we get that if $L:\mathfrak{g}\to V'$ denotes the homogeneous part of $\phi^1_x$, then $$\natural_{x}(L(X))=\natural_{x}(X),\quad \forall X\in \mathfrak{g}.$$It is also clear that the homogeneous part of $\phi^2_x$ is $$X\mapsto \frac{1}{2}L([L(X),X]).$$By Remark \ref{rmk:natural_p}, $\natural_x(\frac{1}{2}L([L(X),X]))=0$. Same for the homogeneous part of $\phi^3_x,\cdots,\phi^{N}_x$. The proof of the lemma is thus complete.
\end{proof}
\end{appendix}
\begin{refcontext}[sorting=nyt]
\printbibliography
\end{refcontext}
	{\footnotesize
		\vskip -2pt National and Kapodistrian University of Athens
		\vskip -2pt e-mail: \texttt{iandroul@math.uoa.gr}
		\medskip
		\vskip-2pt University of Paris-Saclay
		\vskip-2pt e-mail: \texttt{omar.mohsen@universite-paris-saclay.fr}
		\medskip
		\vskip -2pt University of Lorraine
		\vskip-2pt e-mail: \texttt{robert.yuncken@univ-lorraine.fr}
	}
\end{document}